\documentclass[a4paper, 11pt]{amsart}
\usepackage[T1]{fontenc}
\usepackage[type1, tt=false]{libertine}
\usepackage[libertine]{newtxmath}

\usepackage{enumitem}
\usepackage{hyperref}
\usepackage{tikz-cd}

\usepackage{booktabs}
\usepackage{longtable}
\usepackage{multirow}

\newtheorem{thm}{Theorem}[section]
\newtheorem{prop}[thm]{Proposition}
\newtheorem{lem}[thm]{Lemma}
\newtheorem{coro}[thm]{Corollary}
\newtheorem{fact}[thm]{Fact}
\newtheorem{conj}[thm]{Conjecture}
\newtheorem{ques}[thm]{Question}
\newtheorem*{FibrationConj}{Geometric fibration conjecture}

\theoremstyle{definition}
\newtheorem{defn}[thm]{Definition}

\theoremstyle{remark}
\newtheorem{rmk}[thm]{Remark}

\numberwithin{thm}{section}
\numberwithin{equation}{section}
\numberwithin{table}{section}

\newcommand{\N}{\mathbb{N}}
\newcommand{\Z}{\mathbb{Z}}

\newcommand{\R}{\mathbb{R}}
\newcommand{\C}{\mathbb{C}}
\newcommand{\Ha}{\mathbb{H}}
\newcommand{\K}{\mathbb{K}}

\newcommand{\Hyp}{\mathbf{H}}
\newcommand{\RP}{\mathbb{R} \mathbf{P}}
\newcommand{\CP}{\mathbb{C} \mathbf{P}}
\newcommand{\HP}{\mathbb{H} \mathbf{P}}
\newcommand{\KP}{\mathbb{K} \mathbf{P}}
\newcommand{\Sph}{\mathbf{S}}

\DeclareMathOperator{\GL}{GL}
\DeclareMathOperator{\SL}{SL}
\DeclareMathOperator{\U}{U}
\DeclareMathOperator{\SU}{SU}
\DeclareMathOperator{\OO}{O}
\DeclareMathOperator{\SO}{SO}
\DeclareMathOperator{\Pin}{Pin}
\DeclareMathOperator{\Spin}{Spin}
\DeclareMathOperator{\Sp}{Sp}

\DeclareMathOperator{\M}{M}
\DeclareMathOperator{\Sym}{Sym}
\DeclareMathOperator{\Herm}{Herm}
\DeclareMathOperator{\SkewHerm}{SkewHerm}
\DeclareMathOperator{\Cl}{Cl}

\newcommand{\g}{\mathfrak{g}}
\newcommand{\h}{\mathfrak{h}}
\newcommand{\kk}{\mathfrak{k}}
\newcommand{\p}{\mathfrak{p}}
\newcommand{\q}{\mathfrak{q}}
\newcommand{\ab}{\mathfrak{a}}

\DeclareMathOperator{\Ktheory}{K}
\DeclareMathOperator{\KO}{KO}
\DeclareMathOperator{\KU}{KU}
\DeclareMathOperator{\KSp}{KSp}

\newcommand{\rr}{\mathbf{r}}
\newcommand{\cc}{\mathbf{c}}
\newcommand{\qq}{\mathbf{q}}

\DeclareMathOperator{\Rep}{Rep}
\DeclareMathOperator{\RO}{RO}
\DeclareMathOperator{\RU}{RU}
\DeclareMathOperator{\RSp}{RSp}

\newcommand{\gr}{\textup{gr}}
\newcommand{\pt}{\textup{pt}}
\newcommand{\spin}{\textup{spin}}
\newcommand{\std}{\textup{std}}
\newcommand{\triv}{\textup{triv}}
\newcommand{\vol}{\textup{vol}}

\DeclareMathOperator{\Ad}{Ad}
\DeclareMathOperator{\Diag}{Diag}
\DeclareMathOperator{\Grass}{Grass}
\DeclareMathOperator{\id}{id}
\DeclareMathOperator{\Lie}{Lie}
\DeclareMathOperator{\pr}{pr}
\DeclareMathOperator{\Taut}{Taut}
\DeclareMathOperator{\Tr}{Tr}
\DeclareMathOperator{\rk}{rk}
\newcommand{\Sphst}{\operatorname{Sph}^\textup{st}}
\newcommand{\Sphunst}{\operatorname{Sph}^\textup{unst}}

\DeclareMathOperator{\Vect}{Vect}

\newcommand{\bs}{\backslash}
\newcommand{\br}[1]{\lbrack #1 \rbrack}

\newcommand{\ie}{i.e.\ }
\newcommand{\eg}{e.g.\ }
\newcommand{\resp}{resp.\ }

\begin{document}

\title{Compact quotients of homogeneous spaces and homotopy theory of sphere bundles}

\author{Fanny Kassel}
\address{CNRS and Laboratoire Alexander Grothendieck, 
Institut des Hautes \'Etudes Scientifiques, Universit\'e Paris-Saclay, 
35 route de Chartres, 91440 Bures-sur-Yvette, France}
\email{kassel@ihes.fr}

\author{Yosuke Morita}
\address{Faculty of Mathematics, Kyushu University,
744 Motooka, Nishi-ku, Fukuoka 819-0395, Japan}
\email{y-morita@math.kyushu-u.ac.jp}

\author{Nicolas Tholozan}
\address{Laboratoire de Math\'ematiques d'Orsay, Universit\'e Paris-Saclay, B\^atiment 307, rue Michel Magat, 91400 Orsay, France}
\email{nicolas.tholozan@cnrs.fr}

\begin{abstract}
A reductive homogeneous space $G/H$ is always diffeomorphic to 
the normal bundle of an orbit of a maximal compact subgroup of~$G$.
We prove that if $G/H$ admits compact quotients, 
then the sphere bundle associated to this normal bundle is 
fiber-homotopically trivial.
We deduce that many reductive homogeneous spaces do not admit 
compact quotients, such as the complex~spheres 
$\OO(n+ \nolinebreak 1, \C)/\OO(n,\C)$ for all $n \notin \{1,3,7\}$, 
or $\SL(n,\R)/\SL(m,\R)$ for all $n>m>1$, 
which solves conjectures of T. Kobayashi from the early 1990s.
We also prove that if the pseudo-Riemannian hyperbolic space 
$\Hyp^{p,q}_{\R}$ of signature $(p,q)$ admits compact quotients, 
then $p$ must be divisible by at least $2^{\lfloor q/2\rfloor}$.
\end{abstract}

\maketitle
\tableofcontents

\section{Introduction}

\subsection{Problem and main theorem}

Let $G$ be a real linear reductive Lie group and 
$H$ a reductive subgroup of~$G$.
The quotient space $G/H$, equipped with the left action of $G$, 
is called a \emph{homogeneous space of reductive type}, 
or \emph{reductive homogeneous space} for short.
When $H$ is an open subgroup of the set of fixed points of some involutive automorphism of~$G$, 
the space $G/H$ is called a \emph{reductive symmetric space}.

A \emph{Clifford--Klein form} of $G/H$ is a quotient $\Gamma \bs G/H$ 
of $G/H$ by a discrete subgroup $\Gamma$ of~$G$
acting properly discontinuously on $G/H$.
The proper discontinuity of the action ensures that 
$\Gamma \bs G/H$ is Hausdorff, and in fact an orbifold 
(or even a manifold if $\Gamma$ is torsion-free).
We define a \emph{compact quotient} of $G/H$ 
to be a Clifford--Klein form $\Gamma \bs G/H$ which is compact.
This paper is devoted to the following question, which was first 
considered in full generality by Toshiyuki Kobayashi in the late 1980s.

\begin{ques} \label{ques:CompactQuotient}
Does the reductive homogeneous space $G/H$ admit compact quotients?
\end{ques}

When $H$ is compact, this question is equivalent to the existence of 
uniform lattices in $G$ and has been answered positively 
in the early 1960s by Borel \cite{Bor63}. 
On the other hand, when $H$ is not compact, 
Question~\ref{ques:CompactQuotient} is more subtle\footnote{
Note that lattices in $G$ never act properly on $G/H$ when 
$H$ is noncompact; in fact, when $G$ is simple, every lattice of $G$ 
acts ergodically on $G/H$ by Moore's ergodicity theorem.}
and remains wide open, 
despite substantial progress in the past thirty-five years. 

Various obstructions to the existence of compact quotients of 
reductive homogeneous spaces have been developed by various authors, 
including Kobayashi \cite{Kob89,Kob92coh}, 
Labourie--Mozes--Zimmer \cite{Zim94,LMZ95,LZ95}, 
Benoist \cite{Ben96}, Margulis \cite{Mar97}, Shalom \cite{Sha00}, 
and the present authors \cite{Mor15,Tho15+,Mor17,Mor24,KT24+}.

On the other hand, 
some reductive homogeneous spaces do admit compact quotients. 
In particular, when there exists a reductive Lie subgroup $L$ of $G$ 
acting transitively on $G/H$ with compact stabilizer, 
taking $\Gamma$ to be a uniform lattice in~$L$ gives a compact quotient 
$\Gamma \bs G/H$ of $G/H$; such quotients are called \emph{standard}.
A list of such triples $(G,H,L)$ where $G/H$ 
is an irreducible reductive symmetric space was given by 
Kobayashi--Yoshino \cite{KY05}.
This list was proved to be complete by Tojo \cite{Toj19}, 
and Boche\'nski--Tralle \cite{BT24+} showed that, under certain assumptions, the list does not get larger by allowing reductive homogeneous spaces that are not symmetric.
Sometimes standard compact quotients can be deformed 
\cite{Gol85,Ghy95,Kob98,Kas12}, 
and sometimes there even exist ``exotic'' compact quotients 
\cite{Sal00,LL17,MST23+} which are not deformations of standard ones.
However, the following conjecture of Kobayashi 
\cite[\S\,6, (vi)]{Kob97}, \cite[Conj.\,4.3]{Kob01} remains open:

\begin{conj}[Kobayashi] \label{conj:ExistenceStandardQuotients}
The reductive homogeneous space $G/H$ 
admits compact quotients if and only if it admits standard ones.
\end{conj}

In this paper, we present a new and rather powerful 
homotopy-theoretic obstruction to the existence of compact quotients. 

To state our result, let us introduce some notation. 
As above, let $G/H$ be a reductive homogeneous space, 
which we now assume to be connected.
Let $K$ be a maximal compact subgroup of $G$ such that 
$K_H = K \cap H$ is a maximal compact subgroup of $H$. 
We identify $X = K/K_H$ with the $K$-orbit of $o = 1 \cdot H$ in $G/H$, 
and let $N$ be the normal bundle of $X$ in $G/H$.
Finally, let $S(N)$ be the sphere bundle associated to~$N$.
It is called \emph{fiber-homotopically trivial} 
if it admits a bundle map to the trivial sphere bundle 
that is a homotopy equivalence in each fiber 
(see Definition~\ref{defn:FiberwiseHomotopicallyTrivial} 
and Fact \ref{fact:Dol63}).
Our main theorem is the following.

\begin{thm} \label{thm:MainTheorem}
If the reductive homogeneous space $G/H$ admits compact quotients, 
then $S(N)$ is fiber-homotopically trivial.
\end{thm}

\subsection{Topological \texorpdfstring{$\Ktheory$}{K}-Theory 
and applications}

The normal bundle $N$ admits an explicit description 
as a vector bundle associated to the principal bundle $K\to K/K_H$.
The question of its fiber-homotopical triviality is not obvious, 
but fortunately it has already been thoroughly investigated by topologists.

Among characteristic classes of a vector bundle, 
some detect fiber-homotopical triviality 
(namely, the Stiefel--Whitney classes and the Euler class), 
while some do not (namely, the Chern and the Pontryagin classes). 
However, our most powerful obstructions will come from topological $\Ktheory$-theory, 
more precisely from work of Adams, completed by Quillen.

In Sections \ref{s:IndefiniteGrassmannians} and 
\ref{s:ComputationsKTheory}, 
we will systematically investigate and discuss 
the fiber-homotopical triviality of the normal bundle to $K/K_H$ 
for all classical symmetric spaces for which 
Question~\ref{ques:CompactQuotient} was not previously answered, 
as well as for some interesting families of 
reductive homogeneous spaces that are not symmetric. 
More precisely, we will prove the following three main results. 

\begin{thm} \label{thm:OtherSymmetric}
Let $G/H$ be one of the reductive symmetric spaces 
in the following table, where $p \geq q \geq 1$ 
are integers satisfying the specified conditions. 
Then $G/H$ and its associated symmetric space $G/H^a$ 
do not admit compact quotients.
\begin{center}
\begin{longtable}{c|c|c|c|c}
& $G$ & $H$ & $H^a$ & Conditions \\ \hline
\textup{(1)} & 
$\OO(p+q, \C)$ & $\OO(p, \C) \times \OO(q, \C)$ & $\OO(p,q)$ & 
$(p,q) \neq (1,1), (3,1), (7,1)$ \\ 
\textup{(2)} & 
$\OO^\ast(2p+2q)$ & $\OO^\ast(2p) \times \OO^\ast(2q)$ & $\U(p,q)$ & 
$(p,q) \neq (1,1), (3,1)$ \\ 
\textup{(3)} &
$\Sp(2p+2q, \R)$ & $\Sp(2p, \R) \times \Sp(2q, \R)$ & $\U(p,q)$ & 
$(p,q) \neq (3,1)$ \\
\textup{(4)} & 
$\SL(2p, \C)$ & $\Sp(2p, \C)$ & $\SL(p, \Ha)$ &
$p\geq 2$ \\ 
\textup{(5)} & 
$\SO_0(2p, 2q)$ & $\U(p,q)$ & $\U(p,q)$ &
$q \geq 2$ \\ 
\textup{(6)} &
$\SU(2p,2q)$ & $\Sp(p,q)$ & $\Sp(p,q)$ &
$q \geq 2$
\end{longtable}
\end{center}
\end{thm}

\begin{thm} \label{thm:NonSymmetric}
Let $G/H$ be one of the reductive homogeneous spaces 
in the following table, where $p,q \geq 1$ 
are integers satisfying the specified conditions.
Then $G/H$ does not admit compact quotients.
\begin{center}
\begin{longtable}{c|c|c|c}
& $G$ & $H$ & Conditions \\ \hline
\textup{(1)} & 
$\SL(p+q, \R)$ & $\SL(p, \R)$ & 
$p \geq 2,\ (p,q) \neq (3,1), (7,1)$ \\
\textup{(2)} & 
$\OO(p+q, \C)$ & $\OO(p, \C)$ & 
$p \geq 2,\ (p,q) \neq (3,1), (7,1)$ \\
\textup{(3)} & 
$\SL(p+q, \C)$ & $\SL(p, \C)$ & 
$p \geq 2$ \\
\textup{(4)} & 
$\OO^\ast(2p+2q)$ & $\OO^\ast(2p)$ & 
$p \geq 2,\ (p,q) \neq (3,1)$ \\
\textup{(5)} & 
$\Sp(2p+2q, \R)$ & $\Sp(2p, \R)$ & 
$(p,q) \neq (1,1)$ \\
\textup{(6)} & 
$\Sp(2p+2q, \C)$ & $\Sp(2p, \C)$ & 
--- \\
\textup{(7)} & 
$\SL(p+q, \Ha)$ & $\SL(p, \Ha)$ & 
$p \geq 2$
\end{longtable}
\end{center}
\end{thm}

\begin{thm} \label{thm:IndefiniteGrassmannianRCH}
Let $p,q,q' \geq 1$ be integers, and let $n = \max \{ q,q' \}$. 
\begin{enumerate}[label = {\upshape (\arabic*)}]
  \item If $\OO(p,q+q')/(\OO(p,q) \times \OO(q'))$ 
admits compact quotients, then $p$ is divisible by $2^{\nu(n)}$. 
  \item If $\U(p,q+q')/ (\U(p,q) \times \U(q'))$ 
admits compact quotients, then $p$ is divisible by 
\[
\prod_{p\text{ prime } \leq n+1} p^{\xi_p(n)}.
\]
  \item If $\Sp(p,q+q')/ (\Sp(p,q) \times \Sp(q'))$ 
admits compact quotients, then $p$ is divisible by
\[
\prod_{p \text{ prime } \leq 2n+1} p^{\xi_p(2n)} \qquad \left( \text{\resp} \frac{1}{2} \prod_{p \text{ prime } \leq 2n+1} p^{\xi_p(2n)} \right)
\]
whenever $\xi_2(2n) = 2n+1$ (\resp $\geq 2n+2$).
\end{enumerate}
Here, for $n \in \N$, we set 
\begin{equation} \label{eqn:nu}
\nu(n) = 
\begin{cases}
\lfloor n/2 \rfloor & (n \equiv 0,6,7 \mod 8), \\
\lfloor n/2 \rfloor + 1 & (n \equiv 1,2,3,4,5 \mod 8), \\
\end{cases}
\end{equation}
and for each prime~$p$, denoting by $v_p(r) \in \N$ the $p$-adic valuation of~$r$, we set
\begin{equation} \label{eqn:xi_p}
\xi_p(n) = \max \left\{ r + v_p(r) \ \middle| 
\ 1 \leq r \leq \left\lfloor \frac{n}{p-1} \right\rfloor \right\}.
\end{equation}
\end{thm}

The homogeneous spaces considered in 
Theorem~\ref{thm:IndefiniteGrassmannianRCH} 
are indefinite Grassmannians over $\R$, $\C$, 
or the ring $\Ha$ of quaternions: see Section~\ref{ss:IndefGrassm}.

\subsection{Some important special cases}

We now discuss some cases in more details.

\subsubsection{Pseudo-Riemannian hyperbolic spaces} \label{sss:Hpq}

The (real) \emph{pseudo-Riemannian hyperbolic space} 
of signature $(p,q)$ is the reductive symmetric space 
$\OO(p,q+1) / (\OO(p,q) \times \OO(1))$, 
which can be realized as the open set
\[
\Hyp^{p,q}_\R = \{[\mathbf x] \in \RP^{p+q} \mid 
x_1^2 + \ldots + x_p^2 - x_{p+1}^2 - \ldots - x_{p+q+1}^2 < 0 \}
\]
in projective space. 
It carries an $\OO(p,q+1)$-invariant pseudo-Riemannian metric 
of signature $(p,q)$ and constant negative sectional curvature. 
In this case, the maximal compact subspace $X$ is the real projective subspace $\RP^q$
and its normal bundle $N$ is $\Taut_{\RP^q}^{\oplus p}$, 
the direct sum of $p$ copies of the tautological line bundle 
$\Taut_{\RP^q}$.
The (homotopical) triviality of $\Taut_{\RP^q}^{\oplus p}$ 
was investigated by Adams in his celebrated paper \cite{Ada62}, 
where he characterized it by the \emph{Hurwitz--Radon condition}.
Combining Theorem~\ref{thm:MainTheorem} with Adams's characterization, 
we obtain the following, which is a special case of 
Theorem~\ref{thm:IndefiniteGrassmannianRCH} 
(see \eqref{eqn:nu} for the notation $\nu(q)$).

\begin{thm} \label{thm:Hpq}
Let $p,q \geq 1$ be integers.
If the pseudo-Riemannian hyperbolic space $\Hyp^{p,q}_\R$ 
admits compact quotients, then $p$ is divisible by $2^{\nu(q)}$.
\end{thm}

In other words, $\Hyp^{p,q}_{\R}$ 
may admit compact quotients only for the following values of $(p,q)$:
\begin{center}
\begin{longtable}{|l||c|c|c|c|c|c|c|c|c|c|} \hline
$p$ & $2n$ & $4n$ & $8n$ & $16n$ & $32n$ & $64n$ & $128n$ & $256n$ & \dots \\ \hline
$q$ & $1$ & $2, 3$ & $4, 5, 6, 7$ & $8$ & $9$ & $10, 11$ & $12, 13, 14, 15$ & $16$ & \dots \\ \hline
\end{longtable}
\end{center}

Previously, the non-existence of compact quotients of $\Hyp^{p,q}_{\R}$ 
was known in the following two cases:
\begin{itemize}
\item $1 \leq p \leq q$ 
(Calabi--Markus \cite[Th.\,1]{CM62}, Wolf \cite[Th.\,1]{Wol62}, 
Kobayashi \cite[Cor.\,4.4]{Kob89}). 
\item $p$ odd 
(Tholozan \cite[Th.\,5]{Tho15+}, Morita \cite[Cor.\,6.5]{Mor17}, 
improving earlier results by Kulkarni \cite[Cor.\,2.10]{Kul81} 
and Benoist \cite[\S\,1.2, Cor.\,1]{Ben96}).
\end{itemize}
Theorem~\ref{thm:Hpq} recovers these non-existence results 
and is significantly better for large $q$. 
The smallest signatures where our result is new are: 
\[
(p,q) = (6,2), (6,3), (6,4), (6,5), (10,2), (10,3), \dots. 
\]

On the other hand, standard compact quotients of $\Hyp^{p,q}_\R$ 
are known to exist in the following cases:
\begin{itemize}
\item $p=0$ 
(as $\Hyp^{0,q}_\R \simeq \RP^q$ is itself compact),
\item $q=0$ 
(as $\Hyp^{p,0}_{\R}$ is the classical Riemannian hyperbolic space),
\item $q=1$, $p$ even 
(as $L=\U(k,1)$ acts transitively on $\Hyp^{2k,1}_\R$ 
with compact stabilizer \cite[Th.\,6.1]{Kul81}),
\item $q=3$, $p$ divisible by~$4$ 
(as $L=\Sp(k,1)$ acts transitively on $\Hyp^{4k,3}_\R$ 
with compact stabilizer \cite[Th.\,6.1]{Kul81}),
\item $q=7$, $p=8$ 
(as $\OO(8,8)$ contains a subgroup 
$L \simeq \Spin(1,8)$ acting transitively on $\Hyp^{8,7}_\R$ 
with compact stabilizer \cite[Cor.\,5.6]{Kob97}).
\end{itemize}
As a special case of Conjecture~\ref{conj:ExistenceStandardQuotients}, 
Kobayashi \cite{Kob01} conjectured that these are the only values of 
$(p,q)$ for which $\Hyp^{p,q}_{\R}$ admits compact quotients.
After Theorem~\ref{thm:Hpq}, 
the smallest signatures for which this conjecture is still open are: 
\[
(p,q) = (4,2), (8,2), (8,4), (8,5), (8,6), (12,2), \dots. 
\]

\subsection{Complex spheres}

For $n \geq 1$, the \emph{complex sphere} of dimension~$n$ 
is the reductive symmetric space $\OO(n+\nolinebreak 1,\C)/\OO(n,\C)$, 
which can be realized as the affine hypersurface
\[
\Sph^n_{\C} = \{\mathbf z \in \C^{n+1} \mid 
z_1^2 + \ldots + z_{n+1}^2 = 1\}
\]
in~$\C^{n+1}$.
It is known to admit standard compact quotients for $n=1$ 
(as $\OO(1,\C)$ is compact), 
for $n=3$ (as $\Sph^3_{\C}$ is locally isomorphic to the group manifold 
$(G'\times G')/\Diag(G')$ for $G' = \SL(2,\C)$), 
and for $n=7$ (as the subgroup $\Spin(1,7)$ of $\OO(8,\C)$ 
acts transitively on $\Sph^7_{\C}$ 
with compact stabilizer \cite[Th.\,4.2.1]{KY05}).

Here the maximal compact subspace $X$ 
is the $n$-dimensional real sphere $\Sph^n$, 
and its normal bundle $N$ 
is isomorphic to the tangent bundle of~$\Sph^n$, 
which is (homotopically) nontrivial for $n \notin \{ 1,3,7 \}$, 
as proven by Milnor--Spanier \cite{MS60}.
Therefore Theorem~\ref{thm:MainTheorem} yields the following, 
which solves a conjecture of Kobayashi \cite[Conj.\,2.4.2]{KY05}.

\begin{thm} \label{thm:ComplexSpheres}
The complex sphere $\Sph^n_{\C} = \OO(n+1,\C)/\OO(n,\C)$ 
admits compact quotients if and only if $n \in \{ 1,3,7 \}$.
\end{thm}

\begin{rmk}
The complex sphere $\Sph^n_{\C}$ admits compact quotients if and only if its $\Z/2$-quotient 
$\OO(n+\nolinebreak 1,\C) / (\OO(n,\C) \times \OO(1,\C))$ does.
Hence, Theorem \ref{thm:ComplexSpheres} is equivalent to Theorem \ref{thm:OtherSymmetric}~(1) in the case $q=1$. 
\end{rmk}

The non-existence of compact quotients of $\Sph^n_{\C}$ 
was previously proved for $n$ even by Kobayashi 
\cite[Cor.\,4.4]{Kob89}, 
and for $n \equiv 1$ modulo~$4$ by Benoist 
\cite[\S 1.2, Cor.\,1]{Ben96}.

The symmetric spaces $\Sph^n_{\C} = \OO(n+1,\C)/\OO(n,\C)$ 
were the last complex reductive symmetric spaces 
for which the existence of a compact quotient was not known.
Thus, combining the present result with the older results of 
Kobayashi and Benoist (see \cite[Prop.\,2.4.5]{KY05}), 
we obtain the following.

\begin{coro}
Let $G$ be a complex simple Lie group and 
$H$ the subgroup fixed by a nontrivial holomorphic involution of $G$. 
Then $G/H$ admits compact quotients if and only if 
$G$ and $H$ are locally isomorphic to 
$\SO(8, \C)$ and $\SO(7, \C)$, respectively.
\end{coro}

\subsection{The homogeneous space 
\texorpdfstring{$\SL(p+q,\R)/\SL(p,\R)$}{SL(p+q,R)/SL(p,R)}}

Given integers $p \geq 2$ and $q \geq 1$, consider the embedding of 
$\SL(p,\R)$ into $\SL(p+q,\R)$ as a diagonal block.
Kobayashi conjectured in 1990 that $\SL(p+q,\R)/\SL(p,\R)$ 
never admits compact quotients for any $p \geq 2$ and $q \geq 1$, 
and this conjecture has attracted a lot of interest since then.
Here, we solve it for all but two values of $(p,q)$, 
and also over the complex numbers and the quaternions:

\begin{thm} \label{thm:SLn/SLm}
Let $\K$ be $\R$, $\C$, or the ring $\Ha$ of quaternions.
For $p \geq 2$ and $q \geq 1$, the homogeneous space 
$\SL(p+q,\K)/\SL(p,\K)$ does not admit compact quotients, 
except possibly for $\K=\R$ and $(p,q) = (3,1)$ or $(7,1)$.
\end{thm}

On the other hand, the first and third authors proved 
in the independent paper \cite[Cor.\,1.16]{KT24+} that 
$\SL(2k,\R)/\SL(2k-1,\R)$ does not admit compact quotients for 
$k \geq 2$.
We therefore obtain a full answer to Kobayashi's conjecture, 
which works also over the complex numbers and the quaternions.

\begin{coro}
Let $\K$ be $\R$, $\C$, or the ring $\Ha$ of quaternions.
For $p \geq 2$ and $q \geq 1$, the homogeneous space 
$\SL(p+q, \K)/\SL(p, \K)$ never admits compact quotients.
\end{coro}

Prior to the present paper and to \cite{KT24+}, 
the non-existence of compact quotients of 
$\SL(p+\nolinebreak q,\R)/\SL(p,\R)$ 
was known in the following two cases:
\begin{itemize}
\item $q \geq 3$ (Labourie--Zimmer \cite{LZ95}, 
improving prior results of Zimmer \cite{Zim94}, 
Labourie--Mozes--Zimmer \cite{LMZ95}, and Kobayashi \cite{Kob92coh})
\item $p$ even (Tholozan \cite{Tho15+}, Morita \cite{Mor17}, 
improving a prior result of Benoist \cite{Ben96}).
\end{itemize}
See Remark \ref{rmk:OlderResultsSL/SL}
for more details about previous results, including the complex and quaternionic cases.

\subsection{Motivation: the Geometric fibration conjecture} 
\label{ss:MotivationGeometricFibration}

Our main Theorem~\ref{thm:MainTheorem} 
was initially motivated by the following conjecture, 
formulated by the third author in \cite[\S 8]{Tho15+}:

\begin{FibrationConj}
Let $\Gamma$ be a torsion-free subgroup of~$G$ 
acting properly discontinuously and cocompactly on $G/H$. Then
\begin{itemize}
  \item $\Gamma$ is the fundamental group of a closed aspherical manifold~$M$,
  \item there exists a smooth $\Gamma$-equivariant fiber bundle
$G/H \to \widetilde{M}$ whose fibers are $G$-translates of $X = K/K_H$,
where $\widetilde{M}$ is the universal cover of $M$.
\end{itemize}
\end{FibrationConj}

This property was proved to hold for $G/H = (G' \times G')/\Diag(G')$ 
with $G' = \SO(n,1)$ by Gu\'eritaud--Kassel \cite{GK17}, 
which motivated the general conjecture.
Since recently, the conjecture is also known to hold for 
$G/H = \SO(2n,2)/\U(n,1)$ by combining work of 
Barbot--M\'erigot \cite{BM12}, 
Monclair--Schlenker--Tholozan \cite[Th.\,1.10]{MST23+}, 
and Kassel--Tholozan \cite[Th.\,1.11]{KT24+}.
In the special case where $\Gamma \bs G/H$ is a \emph{standard} compact quotient of $G/H$, 
the group $\Gamma$ always satisfies the conclusion of 
the Geometric fibration conjecture (see Remark~\ref{rmk:standard-fibr}).

We shall call 
\emph{local geometric fibration} 
a smooth fiber bundle whose total space is an open subset of $G/H$ 
and whose fibers are translates of $K/K_H$ 
(see Definition~\ref{defn:LocalGeomFib}).
It is easy to see that the fiber-homotopical triviality of $S(N)$ is 
a necessary condition for the existence of local geometric fibrations 
(see Lemma~\ref{lem:ImplicationsTriviality}), 
and Theorem~\ref{thm:MainTheorem} shows that it is also 
a necessary condition for the existence of compact quotients. 
It is consistent with the Geometric fibration conjecture, 
which predicts in particular that any reductive homogeneous space 
admitting compact quotients must have local geometric fibrations.

This is strongly related to work of Kobayashi and Yoshino in \cite{KY05}: 
they introduced the \emph{tangential homogeneous space} 
associated to $G/H$ as the homogeneous space 
$(K \ltimes \p) / (K_H \ltimes \p_H)$, 
where $\p$ (\resp $\p_H$) is the orthogonal complement of 
$\kk = \Lie(K)$ in $\g = \Lie(G)$ 
(\resp of $\kk_H = \Lie(K_H)$ in $\h = \Lie(H)$).
They found obstructions to the existence of compact quotients of 
these homogeneous spaces. Here we prove the following:

\begin{prop} \label{prop:LocalFibrTangential}
For a reductive homogeneous space $G/H$, the following are equivalent: 
\begin{enumerate}[label = {\upshape (\roman*)}]
\item $G/H$ admits local geometric fibrations,
\item the tangential homogeneous space 
$(K\ltimes \p) / (K_H \ltimes \p_H)$ admits compact quotients.
\end{enumerate}
\end{prop}

In Section \ref{s:GeometricFibrations}, 
we will argue in favor of the following audacious conjecture, which is 
already stated in an unpublished draft of Kobayashi--Yoshino \cite{KY}:

\begin{conj} \label{conj:CompactQuotients<=>Fibrations}
Let $G/H$ be a reductive homogeneous space. 
Then the following are equivalent: 
\begin{enumerate}[label = {\upshape (\roman*)}]
\item $G/H$ admits compact quotients,
\item $G/H$ admits local geometric fibrations.
\end{enumerate}
\end{conj}

Note that there is an interesting incompatibility between 
Kobayashi's Conjecture \ref{conj:ExistenceStandardQuotients} and 
Conjecture \ref{conj:CompactQuotients<=>Fibrations} in 
the case of the pseudo-Riemannian hyperbolic spaces $\Hyp^{p,q}_{\R}$.
Indeed, when $p$ is divisible by $2^{\nu(q)}$ (see \eqref{eqn:nu}), 
the space $\Hyp_\R^{p,q}$ admits local geometric fibrations by 
Proposition~\ref{prop:LocalFibrTangential} and \cite[Th.\,1.1.6]{KY05}; 
thus Conjecture~\ref{conj:CompactQuotients<=>Fibrations} 
predicts that $\Hyp_\R^{p,q}$ should admit compact quotients.
On the other hand, as mentioned in Section~\ref{sss:Hpq}, 
the space $\Hyp_\R^{p,q}$ only admits standard compact quotients when 
$(p,q) \in \{(2k,1), (4k,3), (8,7)\}$ 
(see \cite{Toj19} for an outline of proof), 
hence Conjecture~\ref{conj:ExistenceStandardQuotients} predicts that 
$\Hyp_\R^{p,q}$ should not admit compact quotients in all other cases.
The smallest case that would discriminate between the two conjectures 
is the following:

\begin{ques} \label{question:H42}
Does $\Hyp_\R^{4,2}$ admit compact quotients?
\end{ques}

\subsection{Structure of the paper}

In Section~\ref{s:ReductiveHomogeneousSpaces}, 
we review some classical facts about the geometry of 
reductive homogeneous spaces $G/H$, and recall in particular that 
they are diffeomorphic to the normal bundle to 
their maximal compact subspace $K/K_H$.
In Section~\ref{s:FiberwiseHomotopy}, 
we recall the definition of fiber-homotopical triviality 
for sphere bundles, and prove some general sufficient condition for it 
(Lemma~\ref{lem:CriterionFiberwiseHomotopy}).
In Section~\ref{s:ProofMain}, 
we prove our main Theorem~\ref{thm:MainTheorem}.

In Section~\ref{s:RemindersKTheory}, 
we quickly recall some background on topological $\Ktheory$-theory and 
the (reduced) $J$-group. 
In Section~\ref{s:MainStable}, 
we state a consequence of Theorem~\ref{thm:MainTheorem} 
in terms of the reduced $J$-group (Corollary~\ref{cor:MainStable}), 
from which we deduce 
Theorems \ref{thm:OtherSymmetric}, \ref{thm:NonSymmetric}, 
and~\ref{thm:IndefiniteGrassmannianRCH} in the next two sections. 
In Section~\ref{s:IndefiniteGrassmannians}, 
we prove Theorem~\ref{thm:IndefiniteGrassmannianRCH} 
(which contains Theorem~\ref{thm:Hpq} as a special case) 
as an immediate application of Corollary~\ref{cor:MainStable} 
and work on the reduced $J$-group from the 1960-70s.
In Section~\ref{s:ComputationsKTheory}, 
we investigate further the fiber-homotopical triviality of 
the normal bundle to $K/K_H$, proving 
Theorems \ref{thm:OtherSymmetric} and~\ref{thm:NonSymmetric} 
(which contain Theorems \ref{thm:ComplexSpheres} 
and~\ref{thm:SLn/SLm} as special cases).
The examples we treat are ranked by increasing order of 
complexity of the required $\Ktheory$-theoretic arguments. 

In Section~\ref{s:GeometricFibrations}, 
we discuss the existence of local geometric fibrations and 
prove Proposition~\ref{prop:LocalFibrTangential}.

We finally include two appendices.
Appendix~\ref{s:AppendixKTheory} recalls more advanced facts about 
topological $\Ktheory$-theory which are required in the depths of 
Section \ref{s:ComputationsKTheory}, 
while Appendix~\ref{s:AppendixSymmetricSpaces} gives 
an explicit and synthetic description of 
some families of reductive symmetric spaces, 
clarifying how $H$ and its associated symmetric subgroup $H^a$ 
are embedded in~$G$.

\subsection{Acknowledgements}

We warmly thank Olivier Benoist for pointing us to 
the results of Adams at the early stages of this work.
We are grateful to Toshiyuki Kobayashi 
for sharing his unpublished draft \cite{KY} with Taro Yoshino. 
YM would like to thank Taro Yoshino for telling him about 
a possible relation between the triviality of the normal bundle and 
the existence of compact quotients, as well as Conjecture~\ref{conj:CompactQuotients<=>Fibrations}, 
when he was a graduate student. 

Part of this work was completed during YM's visits at IHES 
in Bures-sur-Yvette in the winter 2020, fall 2023, and spring 2025, 
and during the trimester \emph{Higher rank geometric structures} 
at IHP in Paris in the spring 2025, 
as well as during FK's visit to the University of Tokyo in the fall 2025.
We thank the IHES, the IHP (UAR 839 CNRS-Sorbonne Universit\'e), 
the LabEx CARMIN (ANR-10-LABX-59-01), the ENS in Paris, 
and the French--Japanese Laboratory of Mathematics (IRL2025 CNRS) 
at the University of Tokyo for their hospitality and support.
YM was supported by JSPS KAKENHI Grant Numbers 19K14529 and 24K16922.

\section{Reminders: reductive homogeneous spaces} \label{s:ReductiveHomogeneousSpaces}

\subsection{Reductive Lie groups and reductive homogeneous spaces} \label{ss:Reductive}

Let $G$ be a real linear reductive Lie group with Cartan involution $\theta$. 
Let $H$ be a reductive subgroup of $G$, which we assume to be $\theta$-invariant without loss of generality. 
In this situation, the homogeneous space $G/H$ is called \emph{reductive} (or \emph{of reductive type}). In this paper, we assume all reductive homogeneous spaces to be connected. It is equivalent to assuming that $H$ meets all connected components of~$G$.
The fixed-point subgroup
\[
K = G^\theta = \{ g \in G \mid \theta(g) = g \}
\]
is a maximal compact subgroup of $G$, and $K_H = K \cap H$ is a maximal compact subgroup of $H$. 

Let $\g$ and $\h$ be the Lie algebras of $G$ and $H$, respectively. We use the same symbol $\theta$ for the involution on $\g$ induced by the Cartan involution. 
We have $\theta(\h) = \h$. 
We fix a $G$-invariant and $\theta$-invariant 
nondegenerate symmetric bilinear form $B$ on $\g$ for which the symmetric bilinear form
\[
B_\theta \colon \g \times \g \to \R, \qquad (X, Y) \mapsto -B(\theta(X), Y)
\]
is positive definite 
(if $G$ is semisimple, the Killing form satisfies these requirements). 
We write $\q$ for the orthogonal complement of $\h$ in $\g$ with respect to $B$ 
(or equivalently, with respect to~$B_\theta$). 

Let $\g = \kk \oplus \p$ be the eigenspace decomposition of $\g$ with respect to $\theta$: 
\[
\kk = \{ v \in \g \mid \theta(v) = v \}, \qquad \p = \{ v \in \g \mid \theta(v) = -v \}. 
\]
Then $\kk$ and $\kk \cap \h$ are the Lie algebras of $K$ and $K_H$, respectively, and we have the orthogonal decompositions
\[
\h = (\kk \cap \h) \oplus (\p \cap \h), \qquad \q = (\kk \cap \q) \oplus (\p \cap \q).
\]
Note that the adjoint action of $K_H$ on $\g$ preserves the orthogonal decomposition
\[
\g = (\kk \cap \h) \oplus (\kk \cap \q) \oplus (\p \cap \h) \oplus (\p \cap \q).
\]

We identify the compact homogeneous space $X = K/K_H$ with the $K$-orbit of the base point $o = 1 \cdot H$ in $G/H$. 
We denote by $N$ the normal bundle of $X$ in $G/H$. It is $K$-equivariantly isomorphic to 
$K \times_{K_H} (\p \cap \q)$, the vector bundle associated to the principal $K_H$-bundle $K \to K/K_H$ and the adjoint action of $K_H$ on $\p \cap \q$.

In the proof of Theorem~\ref{thm:MainTheorem}, we shall use the following fact, which is classical when $G/H$ is a reductive symmetric space, and was proved by Kobayashi \cite[Lem.\,2.7]{Kob89} in general.

\begin{fact} \label{fact:Kob89}
The map 
\[
\Phi \colon N = K \times_{K_H} (\p \cap \q) \to G/H, \qquad [k, v] \mapsto k \exp(v) \cdot o
\]
is a $K$-equivariant diffeomorphism. 
\end{fact}

In particular, $X = K/K_H$ is connected since $G/H$ is. \\

Recall that by our choice of the symmetric bilinear form~$B$, it is negative definite on $\kk$ and positive definite on $\p$. We will sometimes write its signature as
\[
(\dim_+(G),\dim_-(G)) = (\dim(\p), \dim(\kk)).
\]
The restriction of $B$ to $\q \simeq \g/\h$ extends to a $G$-invariant pseudo-Riemannian metric on $G/H$, of signature
\begin{align} \label{eqn:dim_+}
(\dim_+(G/H),\dim_-(G/H)) & = (\dim (\p \cap \q), \dim (\kk \cap \q))\\ 
& = (\dim_+(G)-\dim_+(H), \dim(K/K_H)). \nonumber 
\end{align}

We note that $\p \cap \q$ and $\kk \cap \q$ are respectively identified with $N_o X$ and $T_o X$. In fact, one can check that $X$ is a totally negative totally geodesic submanifold of $G/H$ of maximal dimension, and that the pseudo-Riemannian metric on $G/H$ restricts to a positive definite metric on the normal bundle $N$ of~$X$ in $G/H$.

\subsection{Reductive symmetric spaces and associated pairs}
\label{ss:ReductiveSymmetric}

Let $G$ be a real linear reductive Lie group with Cartan involution $\theta$, and let $H$ be a closed subgroup of~$G$.
The pair $(G,H)$ is called \emph{symmetric} if there exists an involution $\sigma$ of~$G$ commuting with~$\theta$ such that $H$ is an open subgroup of the fixed-point subgroup 
\[
G^\sigma = \{ g \in G \mid \sigma(g) = g \}. 
\]
Note that such an $H$ is automatically a reductive subgroup of~$G$. 
In this situation, the reductive homogeneous space $G/H$ is called a \emph{reductive symmetric space}. 
For simplicity, we only consider the case where $H = G^\sigma$ in this paper. This does not cause a loss of generality in our context, as properness and cocompactness are unaffected by taking finite-index subgroups (or overgroups).

Let us recall the corresponding concepts for Lie algebras. Let $\g$ be a real reductive Lie algebra with Cartan involution $\theta$, and let $\h$ be a Lie subalgebra of~$\g$.
We say that $(\g,\h)$ is a \emph{reductive symmetric pair} if there exists an involution $\sigma$ on $\g$ commuting with $\theta$ such that $\h$ is the fixed-point subalgebra 
\[
\g^\sigma = \{ u \in \g \mid \sigma(u) = u \}.
\]

Let $(\g, \h)$ be a reductive symmetric pair, and consider the orthogonal complement $\q$ of $\h$ in~$\g$ as in Section~\ref{ss:Reductive} just above. Then $\g = \h \oplus \q$ is the eigenspace decomposition of $\g$ with respect to $\sigma$: 
\[
\h = \{ u \in \g \mid \sigma(u) = u \}, \qquad \q = \{ u \in \g \mid \sigma(u) = -u \}. 
\]

Since $\sigma$ and $\theta$ commute, the composition $\sigma \theta$ is again an involution on $\g$. 
Let $\h^a$ denote the corresponding symmetric subalgebra of $\g$: 
\[
\h^a = \g^{\sigma \theta} = \{ u \in \g \mid \sigma(\theta(X)) = u \}. 
\]
In other words, 
\[
\h^a = (\kk \cap \h) \oplus (\p \cap \q).
\]
The pair $(\g, \h^a)$ is called the \emph{associated symmetric pair} to $(\g, \h)$.
We similarly define the \emph{associated symmetric space} $G/H^a$ to $G/H$: 
\[
H^a = G^{\sigma \theta} = \{ g \in G \mid \sigma(\theta(g)) = g \}. 
\]
We have $K \cap H^a = K \cap H$, hence $K_H = K \cap H$ is also a maximal compact subgroup of $H^a$. It follows that $G/H$ and $G/H^a$ have the same maximal compact subspace $X = K/K_H$. Finally, the normal bundle to $K/K_H$ in $G/H^a$ is
\[
K \times_{K_H} (\p\cap \h).
\]

\begin{rmk} \label{rmk:H^aOrbit}
The $H^a$-orbit of the base point $o = 1 \cdot H$ in $G/H$ is a totally positive totally geodesic submanifold of $G/H$ of maximal dimension, orthogonal to $X$ at $o$, and isometric to the Riemannian symmetric space $H^a/K_H$. 
Under the identification $G/H \simeq N$ given by Fact~\ref{fact:Kob89}, 
the fibers of the vector bundle $N \to X$ are the $K$-translates of this submanifold. 
\end{rmk}

A reductive symmetric pair $(\g, \h)$ is called an \emph{irreducible} if either $\g$ is simple or $(\g, \h) \simeq\linebreak (\g' \oplus \g', \Diag(\g'))$ for some simple~$\g'$, where $\Diag(\g')$ denotes the diagonal of $\g' \oplus \g'$. 
Any reductive symmetric pair is decomposed into a direct sum of $(\R, 0)$, of $(i\R, i\R)$, and of irreducible ones. By an \emph{irreducible symmetric space}, we mean a reductive symmetric space $G/H$ such that $(\g, \h)$ is irreducible.

\section{A sufficient condition for fiber-homotopy equivalence} \label{s:FiberwiseHomotopy}

Let us start by recalling some basic definitions.

\begin{defn} \label{defn:FiberwiseHomotopyEquivalence}
Let $\pi \colon E \to X$ and $\pi' \colon E' \to X$ be two fiber bundles with common base space~$X$. 
\begin{enumerate}
\item A \emph{continuous map from $E$ to $E'$ over $X$} is a continuous map 
$f \colon E \to E'$ such that $\pi' \circ f = \pi$.
\item Let $f_0, f_1 \colon E \to E'$ be two continuous maps over~$X$. 
A \emph{fiber-homotopy from $f_0$ to $f_1$} is a continuous map $F \colon E \times [0,1] \to E'$ over $X$ such that $F(e,0) = f_0(e)$ and $F(e,1) = f_1(e)$ for any $e \in E$.
We say that \emph{$f_0$ is fiber-homotopic to $f_1$} if there exists a fiber-homotopy from $f_0$ to~$f_1$.
\item A continuous map $f \colon E \to E'$ over $X$ 
is called a \emph{fiber-homotopy equivalence} 
if there exists a continuous map $g \colon E' \to E$ over $X$ 
such that $g \circ f$ and $f \circ g$ 
are respectively fiber-homotopic to $\id_E$ and $\id_{E'}$.
We say that \emph{$E$ is fiber-homotopy equivalent to $E'$} 
if there exists a fiber-homotopy equivalence from $E$ to $E'$. 
\end{enumerate}
\end{defn}

\begin{defn} \label{defn:FiberwiseHomotopicallyTrivial}
Let $\pi \colon E \to X$ be a fiber bundle with typical fiber $F$. 
We say that $E$ is \emph{fiber-homotopically trivial} if 
it is fiber-homotopy equivalent to the trivial bundle $X \times F$. 
\end{defn}

Dold proved that a bundle map is a fiber-homotopy equivalence if it is a homotopy equivalence in each fiber.

\begin{fact}[{Dold \cite[Th.\,6.3]{Dol63}}] \label{fact:Dol63}
Let $X$ be a locally contratible paracompact Hausdorff space. 
Let $\pi \colon E \to X$ and $\pi' \colon E' \to X$ be two fiber bundles over $X$, and let $f \colon E \to E'$ be a continuous map over $X$. 
Then $f$ is a fiber-homotopy equivalence if and only if $f_x \colon E_x \to E'_x$ is a homotopy equivalence for every $x \in X$. 
\end{fact}

\begin{rmk} \label{rmk:Dol63Connected}
Let $X$, $E$, and $E'$ be as in Fact~\ref{fact:Dol63}. 
If $x, y \in X$ are joined by some continuous path, 
then $f_x \colon E_x \to E'_x$ is a homotopy equivalence if and only if 
$f_y \colon E_y \to E'_y$ is. 
Therefore, if $X$ is in addition connected, 
in order to see that $f$ is a fiber-homotopy equivalence, 
it suffices to verify that 
$f_x \colon E_x \to E'_x$ is a homotopy equivalence for \emph{some} $x \in X$. 
\end{rmk}

In the proof of Theorem~\ref{thm:MainTheorem}, we will use the following sufficient condition for two sphere bundles to be fiber-homotopy equivalent: 

\begin{lem} \label{lem:CriterionFiberwiseHomotopy}
Let $X = (X, x_0)$ be a connected and locally contractible based compact Hausdorff space, 
let $\pi \colon E \to X$ and $\pi' \colon E' \to X$ be two real vector bundles over $X$, and let $0_E \colon X \to E$ and $0_{E'} \colon X \to E'$ be the zero sections of $E$ and $E'$, respectively. 
Assume that there exists a continuous map $\Psi \colon E \to E'$ satisfying the following four conditions:
\begin{enumerate}[label = {\upshape (\roman*)}]
  \item $\Psi \circ 0_E = 0_{E'}$,
  \item $\Psi(E_{x_0}) \subset E'_{x_0}$, 
  \item The induced map $\Psi \colon E_{x_0} \smallsetminus \{ 0_E(x_0) \} \to E'_{x_0} \smallsetminus \{ 0_{E'}(x_0) \}$ is a homotopy equivalence,
  \item $\Psi^{-1}(0_{E'}(X))$ is a compact subset of~$E$. 
\end{enumerate}
Then the sphere bundles $S(E)$ and $S(E')$ are fiber-homotopy equivalent.
\end{lem}

\begin{rmk}
Importantly here, $\Psi$ is not assumed to be a map over $X$. 
\end{rmk}

\begin{proof}[Proof of Lemma \ref{lem:CriterionFiberwiseHomotopy}]
We fix metrics $\| {-} \|$ on $E$ and $E'$, and regard $S(E)$ and $S(E')$ as the subbundles of $E$ and $E'$ consisting of the norm-one vectors, respectively. 
By condition~(iv), rescaling the metric on $E$ if necessary, we may assume that $S(E) \cap \Psi^{-1}(0_{E'}(X)) = \varnothing$. 

Define a map $\psi \colon S(E) \to S(E')$ by 
\[
\psi(v) = \frac{\Psi(v)}{\| \Psi(v) \|}, 
\]
and define a continuous map
\[
\Phi = (\Phi_t)_{t \in [0,1]} \colon S(E) \times [0,1] \to X
\]
by 
$\Phi_t(v) = \pi'\circ\Psi(tv)$. One sees from condition~(ii) that the diagram 
\[
\begin{tikzcd}
(S(E) \times \{ 1 \}) \cup (S(E)_{x_0} \times [0, 1]) \ar[d, hook] \ar[r, "\pr_1"] & S(E) \ar[r, "\psi"] & S(E') \ar[d, "\pi'"] \\
S(E) \times [0, 1] \ar[rr, "\Phi"'] & & X
\end{tikzcd}
\]
commutes, where $\pr_1 \colon (S(E) \times \{ 1 \}) \cup (S(E)_{x_0} \times [0, 1]) \to S(E)$ is the projection onto the first factor. 
By the homotopy lifting property of fibrations (see \eg \cite[Ch.\,7]{May99}), there exists a continuous map
\[
\widehat{\Phi} = (\widehat{\Phi}_t)_{t \in [0, 1]} \colon S(E) \times [0, 1] \to S(E') 
\]
which makes the following diagram commute: 
\[
\begin{tikzcd}
(S(E) \times \{ 1 \}) \cup (S(E)_{x_0} \times [0, 1]) \ar[d, hook] \ar[r, "\pr_1"] & S(E) \ar[r, "\psi"] & S(E') \ar[d, "\pi'"] \\
S(E) \times [0, 1] \ar[rr, "\Phi"'] \ar[rru, "\widehat{\Phi}", shift right = 1] & & X.
\end{tikzcd}
\]

We claim that $\widehat{\Phi}_0 \colon S(E) \to S(E')$ is a fiber-homotopy equivalence. 
Indeed, it follows from condition~(i) that $\pi' \circ \widehat{\Phi}_0 = \Phi_0 = \pi$,
\ie $\widehat{\Phi}_0$ is a continuous map over $X$. Furthermore, we have 
\[
(\widehat{\Phi}_0)_{|S(E)_{x_0}} = \psi_{|S(E)_{x_0}},
\]
hence 
\[
(\widehat{\Phi}_0)_{|S(E)_{x_0}} \colon S(E)_{x_0} \to S(F)_{x_0}
\]
is a homotopy equivalence by condition~(iii). 
We thus conclude from Fact~\ref{fact:Dol63} and Remark~\ref{rmk:Dol63Connected} that $\widehat{\Phi}_0$ is a fiber-homotopy equivalence. 
\end{proof}

\section{Fiber-homotopical triviality of the normal sphere bundle} \label{s:ProofMain}

In this section we prove Theorem~\ref{thm:MainTheorem}.
We use the notation and terminology of Sections \ref{s:ReductiveHomogeneousSpaces} and~\ref{s:FiberwiseHomotopy}.

Suppose that a reductive homogeneous space $G/H$ admits a compact quotient $\Gamma \bs G/H$.
By Selberg's lemma \cite[Lem.\,8]{Sel60}, up to replacing $\Gamma$ by a finite-index subgroup, we may assume that $\Gamma$ is torsion-free, hence acts freely on $G/H$.

\subsection{Preliminary lemmas}

Let $o = 1 \cdot H$ and $\widetilde{o} = 1 \cdot K_H$ denote the base points of $G/H$ and $G/K_H$, respectively.
The following basic observation already appears in \cite[\S\,3.2]{KT24+}.

\begin{lem} \label{lem:EquivariantSection}
The $G$-equivariant fiber bundle 
\[
\pi \colon G/K_H \to G/H, \qquad g \cdot \widetilde{o} \mapsto g \cdot o
\]
admits a $\Gamma$-equivariant section $s \colon G/H \to G/K_H$ such that $s(o) = \widetilde{o}$. 
\end{lem}

\begin{proof}
The fiber bundle $\pi \colon G/K_H \to G/H$ induces a fiber bundle $\pi_\Gamma \colon \Gamma \bs G/K_H \to \Gamma \bs G/H$.
Its typical fiber $H/K_H$ is diffeomorphic to $\p \cap \h$, hence contractible. 
Therefore there exists a section $s_\Gamma \colon \Gamma \bs G/H \to \Gamma \bs G/K_H$ of $\pi_\Gamma$
such that $s_\Gamma(\Gamma \cdot o) = \Gamma \cdot \widetilde{o}$. 
The pullback of the section $s_\Gamma$ under the covering map $G/H \to \Gamma \bs G/H$ gives a $\Gamma$-equivariant section $s \colon G/H \to G/K_H$ with $s(o) = \widetilde{o}$. 
\end{proof}

As in previous sections, we identify $X = K/K_H$ with the $K$-orbit of $o$ in $G/H$. 
For each $x \in G/H$, the subset 
$s(x) X$ of $G/H$ is well defined since $\ell X = X$ for any $\ell \in K_H$. 
The following is contained in \cite[Lem.\,3.6]{KT24+}; we give a short proof for the reader's convenience.

\begin{lem} \label{lem:TranslateIntersection}
Let $s \colon G/H \to G/K_H$ be as in Lemma~\ref{lem:EquivariantSection}.
For any compact subset $C$ of $G/H$, the set
\[
Z_C = \{ x \in G/H \mid s(x) X \cap C \neq \varnothing \}
\]
is compact. 
\end{lem}

\begin{proof}
By compactness of $X$ and~$C$, the set $Z_C$ is closed in $G/H$.
Thus, it suffices to check that $Z_C$ is relatively compact in $G/H$. 

Since $\Gamma$ acts cocompactly on $G/H$, there is a compact subset $C'$ of $G/H$ such that $\Gamma \cdot C' = G/H$.
The set
\[
F = \{ \gamma \in \Gamma \mid \gamma \cdot s(C') X \cap C \neq \varnothing \}
\]
is finite since $s(C') X$ and $C$ are compact and the action of $\Gamma$ on $G/H$ is properly discontinuous.

We claim that $Z_C \subset F \cdot C'$.
Indeed, for any $x \in Z_C$, there exist $\gamma \in \Gamma$ and $x' \in C'$ such that $x = \gamma \cdot x'$.
By $\Gamma$-equivariance of $s$, we have $s(x)X = \gamma \cdot s(x') X \subset \gamma \cdot s(C') X$.
Therefore $\gamma \cdot s(C') X \cap C$ contains the set $s(x)X \cap C$, which is nonempty since $x\in Z_C$, and so $\gamma \in F$.
This proves the claim.

Thus the closed subset $Z_C$ of $G/H$ is relatively compact, hence compact. 
\end{proof}

\subsection{Proof of Theorem~\ref{thm:MainTheorem}}

Let $\Phi \colon N \to G/H$ be the diffeomorphism in Fact~\ref{fact:Kob89}. Since the quotient map $G \to G/K_H$ is a fiber bundle and since $\p \cap \q$ is contractible, there exists a smooth map $\widetilde{s} \colon N_o = \p \cap \q \to G$ 
such that the diagram 
\begin{equation}
\label{eqn:SigmaTildeDiagram}
\begin{tikzcd}
&[-0.5cm] & & G \arrow[d] \\
\p \cap \q \arrow[r, equal] \arrow[rrru, "\widetilde{s}", bend left = 12] & N_o \arrow[r, "{\Phi_{\vert N_{o}}}"'] & G/H \arrow[r, "s"'] & G/K_H.
\end{tikzcd}
\end{equation}
commutes and $\widetilde{s}(0) = 1$. 
Define the map 
\[
\Psi \colon X \times (\p \cap \q) \to N
\]
by $\Psi(x,v) = \Phi^{-1}(\widetilde{s}(v) x)$. 
We claim that $\Psi$ satisfies conditions (i)--(iv) of Lemma~\ref{lem:CriterionFiberwiseHomotopy}. 

(i): Since $\widetilde{s}(0) = 1$, we have $\Psi(x, 0) = [x, 0]$, \ie 
the map $\Psi$ restricts to the identity map between the zero sections. 

(ii), (iii): It follows from the commutativity of the diagram (\ref{eqn:SigmaTildeDiagram}) that 
$\Psi(o, v) = [o, v]$.
Hence, $\Psi(o, {-}) \colon \p \cap \q \to N_o$ is a linear isomorphism.

(iv): We prove the stronger assertion that $\Psi$ is a proper map.  
Since $\Phi \colon N \to G/H$ is a diffeomorphism, it is equivalent to saying that $\Phi \circ \Psi$ is a proper map.
For any compact subset $C$ of $G/H$, we have
\begin{align*}
(\Phi \circ \Psi)^{-1}(C) &= \{ (x, v) \in X \times (\p \cap \q) \mid \widetilde{s}(v) x \in C \} \\
&\subset X \times \{ v \in \p \cap \q \mid \widetilde{s}(v) X \cap C \neq \varnothing \} \\
&= X \times (\Phi_{|N_o})^{-1}(Z_C), 
\end{align*}
where $Z_C$ is the compact subset of $G/H$ defined in Lemma~\ref{lem:TranslateIntersection}. Since 
$\Phi_{|N_o} \colon \p \cap \q \to G/H$
is a closed embedding, $(\Phi \circ \Psi)^{-1}(C)$ is a compact subset of $X \times (\p \cap \q)$. 

We can thus apply Lemma~\ref{lem:CriterionFiberwiseHomotopy} and deduce that $S(N)$ is fiber-homotopically trivial, which completes the proof of Theorem~\ref{thm:MainTheorem}. 
\hfill\qedsymbol

\begin{rmk}
Since the proof might seem a bit technical, let us try to rephrase it in an informal but more visual way. 

The section $s$ assigns to every point $x \in G/H$ a translate $s(x) X$ of the compact submanifold $X$ passing through $x$, in a smooth and $\Gamma$-equivariant way. 
When $G/H$ is a reductive symmetric space, the fibers of the projection map $N \to X$ are totally positive totally geodesic submanifolds of $G/H$ of maximal dimension, as mentioned in Remark~\ref{rmk:H^aOrbit}. 
On the other hand, the base space $X$ is a totally negative totally geodesic submanifold of maximal dimension. Thus, each $s(x) X$ is transverse to the fibers, and one can see that it defines a section of~$N$.
One can define a continuous map over~$X$ 
\[
\Psi' \colon X \times N_o \to N \ (= G/H)
\]
by assigning to each $(x, v) \in X \times N_o$ the value of the section $s(v) X$ at $x \in X$. Here, we have regarded $N_o$ as a submanifold of $G/H$. 
Finally, the properness and cocompactness of the $\Gamma$-action imply that, for $v \in N_o$ with sufficiently large norm, 
$s(v) X$ does not intersect the zero section, because it is the image of $s(x) X$ by $\gamma$ for some $x$ in a compact fundamental domain and a very large $\gamma$ in $\Gamma$. Thus, the restriction of $\Psi'$ to $X \times S$, with $S$ a very large sphere in $N_o$, gives the required homotopical trivialization of $S(N)$.

When $G/H$ is not symmetric, $s(x) X$ might not be a section of the vector bundle $N \to X$. However, it is always homotopic to a section, and Lemma~\ref{lem:CriterionFiberwiseHomotopy} is designed to deal with this more general case.
\end{rmk}

\section{Reminders: topological \texorpdfstring{$\Ktheory$}{K}-theory and the \texorpdfstring{$J$}{J}-group} \label{s:RemindersKTheory}

In this section, we briefly recall the basic definitions of topological $\Ktheory$-theory and the $J$-group, which will be heavily used in the rest of this paper. 
See Appendix~\ref{s:AppendixKTheory} for more properties of them.

\subsection{The (reduced) $0$-th $\Ktheory$-theory}

Recall the following classical notion.

\begin{defn} \label{def:group-compl}
Let $A$ be a commutative monoid.
The \emph{group completion} (or \emph{Grothendieck group}) of~$A$~is 
\[
\Ktheory(A) = (A \times A) / {\sim},
\]
where the equivalence relation $\sim$ is defined by
\[
(a_0, a_1) \sim (b_0, b_1) \iff \exists c \in A,\, a_0 + b_1 + c = a_1 + b_0 + c.
\]
\end{defn}

The monoid structure on $A$ induces an abelian group structure on $\Ktheory(A)$; the inverse of $\br{a_0, a_1}$ is $\br{a_1, a_0}$. We will write $\br{a}$ instead of $\br{a, 0}$, 
and $\br{a_0} - \br{a_1}$ instead of $\br{a_0, a_1}$. 
One can easily see that $\Ktheory(A)$ has the following universal property: 
\begin{itemize}
\item \emph{For any abelian group $B$
and any monoid homomorphism $\varphi \colon A \to B$, there exists a unique group homomorphism $\varphi' \colon \Ktheory(A) \to B$ such that $\varphi' \circ \gamma = \varphi$, where $\gamma \colon A \to \Ktheory(A)$ is the obvious map.}
\end{itemize}
Intuitively, $\Ktheory(A)$ is the abelian group obtained by ``formally adding the inverses'' to $A$. 

\begin{defn}
Let $X$ be a compact Hausdorff space. 
\begin{enumerate}
\item For $\K = \R, \C$, or the ring $\Ha$ of quaternions, we define $\Vect_\K(X)$ to be the set of isomorphism classes of $\K$-vector bundles over $X$. We regard $\Vect_\K(X)$ as a commutative monoid for the direct sum $\oplus$. 
\item We define the \emph{$0$-th $\KO$-theory} of $X$, 
denoted as $\KO^0(X)$ (or simply $\KO(X)$), to be the group completion of the commutative monoid $\Vect_\R(X)$: 
\[
\KO^0(X) = \Ktheory(\Vect_\R(X)). 
\]
We similarly define
\[
\KU^0(X) = \Ktheory(\Vect_\C(X)), \qquad 
\KSp^0(X) = \Ktheory(\Vect_\Ha(X)). 
\]
\end{enumerate}
\end{defn}

\begin{rmk}
In the literature, the $0$-th $\KU$-theory is often simply called the \emph{$\Ktheory$-theory} (or \emph{topological $\Ktheory$-theory}) and written as $\Ktheory^0(X)$ (or $\Ktheory(X)$).
\end{rmk}

\begin{rmk}
For the definition of the nonzero degree part, see 
Sections~\ref{ss:NegativeDegrees} and \ref{ss:PositiveDegrees}. 
\end{rmk}

A continuous map $f \colon X \to Y$ between compact Hausdorff spaces induces a monoid homomorphism $f^\ast \colon \Vect_\R(Y) \to \Vect_\R(X)$, hence a group homomorphism $f^\ast \colon \KO^0(Y) \to \KO^0(X)$. If two continuous maps $f, g \colon X \to Y$ are homotopic, then we have $f^\ast = g^\ast$ (see \eg \cite[Lem.\,1.4.3]{Ati67}). Similar results hold for $\KU$ and $\KSp$. 

For $\K = \R, \C$, or $\Ha$, we denote by $\underline{\K} \in \Vect_\K(X)$ the trivial line bundle $X \times \K$.

The tensor product of vector bundles induces a commutative ring structure on $\KO^0(X)$ (resp.\ $\KU^0(X)$). We will simply write $1$ for its multiplicative unit $[\underline{\R}]$ (resp.\ $[\underline{\C}]$). 
The tensor product similarly induces a $\KO^0(X)$-module structure on $\KSp^0(X)$ and a symmetric bilinear map 
\[
\KSp^0(X) \times \KSp^0(X) \to \KO^0(X), \qquad (\br{E}, \br{F}) \mapsto \br{E \otimes_\Ha \overline{F}},
\]
making the direct sum $\KO^0(X) \oplus \KSp^0(X)$ a commutative ring (see \cite[\S\,1.1.1]{All73}). 

\begin{defn}
For a based compact Hausdorff space $X = (X, x_0)$, we define the \emph{$0$-th reduced $\KO$-theory}, denoted as $\widetilde{\KO}^0(X)$ (or simply $\widetilde{\KO}(X)$), to be the kernel of the group homomorphism
\[
\rk_{x_0} \colon \KO^0(X) \to \Z
\]
defined by setting 
$\rk_{x_0}([E]) = \dim_{\R}(E_{x_0})$
for $E \in \Vect_\R(X)$.
We similarly define $\widetilde{\KU}^0(X)$ and $\widetilde{\KSp}^0(X)$. 
\end{defn}

We have a canonical direct sum decomposition $\KO^0(X) = \widetilde{\KO}^0(X) \oplus \Z$, where the $\Z$-factor is spanned by trivial vector bundles. The same holds for $\KU$ and $\KSp$. 

\begin{rmk}
The definition of $\widetilde{\KO}^0(X)$ depends only on the connected component of $X$ containing~$x_0$. In particular, if $X$ is connected, then $\widetilde{\KO}^0(X)$ does not depend on the choice of $x_0$. 
\end{rmk}

As in the unbased case, a based continuous map $f \colon X \to Y$ between based compact Hausdorff spaces induces a group homomorphism $f^\ast \colon \widetilde{\KO}^0(Y) \to \widetilde{\KO}^0(X)$. If two based continuous maps $f, g \colon X \to Y$ are homotopic in the based sense, then $f^\ast = g^\ast$.
Similar results hold for $\widetilde{\KU}^0$ and~$\widetilde{\KSp}^0$. 

The complexification, the quaternionification, and forgetting the complex or the quaternionic structures respectively induce the monoid homomorphisms 
\[
\begin{tikzcd}
\Vect_\R(X) \ar[r, shift left = 1, "\cc "] & \Vect_\C(X) \ar[l, shift left = 1, "\rr "] \ar[r, shift left = 1, "\qq "] & \Vect_\Ha(X) \ar[l, shift left = 1, "\cc' "]. 
\end{tikzcd}
\]
We use the same symbols for the induced group homomorphisms: 
\[
\begin{tikzcd}
\KO^0(X) \ar[r, shift left = 1, "\cc "] & \KU^0(X) \ar[l, shift left = 1, "\rr "] \ar[r, shift left = 1, "\qq "] & \KSp^0(X) \ar[l, shift left = 1, "\cc' "]. 
\end{tikzcd}
\]
Similarly, we write $\overline{(-)} \colon \KU(X) \to \KU(X)$ for the group homomorphism induced from the complex conjugation $\overline{(-)} \colon \Vect_\C(X) \to \Vect_\C(X)$.
One easily checks the following identities: 
\begin{itemize}
\item For any $x \in \KO^0(X)$, we have 
$\rr \cc x = 2x$ and $\overline{\cc x} = \cc x$. 
\item For any $y \in \KU^0(X)$, we have 
$\overline{\overline{y}} = y$, 
$\rr \overline{y} = \rr y$, 
$\qq  \overline{y} = \qq y$, and 
$\cc \rr y = \cc' \qq y = y + \overline{y}$.
\item For any $z \in \KSp^0(X)$, we have 
$\qq \cc' z = 2z$ and $\overline{\cc' z} = \cc' z$.
\end{itemize}

\begin{rmk}
The map $\cc$ is a ring homomorphism, whereas $\rr$ is not. 
\end{rmk}

\begin{rmk}\label{rmk:KSpProduct}
For $x,y \in \KSp^0(X)$, we have 
$(\cc' x)(\cc' y) = \cc (xy)$ in $\KU^0(X)$. 
\end{rmk}

\subsection{The (reduced) $J$-group} \label{ss:J}

By a \emph{bundle of based spheres} over a compact Hausdorff space $X$, we mean 
a sphere bundle $\pi \colon S \to X$ together with a specified section $\sigma \colon X \to S$.
Similarly to Definition~\ref{defn:FiberwiseHomotopyEquivalence}, one can define a notion of fiber-homotopy equivalence for bundles of based spheres.
We write $\Sphunst(X)$ for the set of all fiber-homotopy equivalence classes of bundles of based spheres over~$X$.
This is a commutative monoid under the fiberwise smash product. 

\begin{rmk}
If $\pi \colon S \to X$ is a sphere bundle, 
its fiberwise suspension is naturally regarded as a bundle of based spheres. If $S = S(E)$ is the sphere bundle of a real vector bundle~$E$, the resulting bundle of based spheres is the fiberwise one-point compactification $S^E$ of~$E$. 

In particular, if the sphere bundles $S(E)$ and $S(E')$ of two real vector bundles $E$ and $E'$ are fiber-homotopy equivalent, then $\br{S^E} = \br{S^{E'}}$ in $\Sphunst(X)$.
\end{rmk}

We define the group $\Sphst(X)$ to be the group completion (Definition~\ref{def:group-compl}) of $\Sphunst(X)$: 
\[
\Sphst(X) = \Ktheory(\Sphunst(X)). 
\]
The fiberwise one-point compactification 
\[
\Vect_\R(X) \to \Sphunst(X), \qquad E \mapsto \br{S^E}
\]
induces a group homomorphism $J \colon \KO^0(X) \to \Sphst(X)$, called the \emph{$J$-homomorphism}.
For $E \in \Vect_\R(X)$, we shall write $J(E)$ instead of $J(\br{E})$.

\begin{defn}
For a compact Hausdorff space~$X$, the \emph{$J$-group} of~$X$ is the image $J(X)$ of the $J$-homomorphism.
\end{defn}

For a based compact Hausdorff space $X = (X, x_0)$, consider the following composite:
\[
\begin{tikzcd}
\widetilde{J} \colon \KO^0(X) \ar[r, "\pi"] & \widetilde{\KO}^0(X) \ar[r, "J"] & J(X), 
\end{tikzcd}
\]
where the first map $\pi$ is the projection with respect to the decomposition 
$\KO(X) = \widetilde{\KO}(X) \oplus \Z$.

\begin{defn} \label{def:red-J-group}
The \emph{reduced $J$-group} of~$X$ is the image $\widetilde{J}(X)$ of~$\widetilde{J}$, or in other words the image of $\widetilde{\KO}^0(X)$ under~$J$.
\end{defn}

Note that $\widetilde{J} \colon \KO^0(X) \to \widetilde{J}(X)$ is also the composite
\[
\begin{tikzcd}
\KO^0(X) \ar[r, "J"] & J(X) \ar[r, "\varpi"] & \widetilde{J}(X), 
\end{tikzcd}
\]
where the second map $\varpi$ is the projection with respect to the decomposition 
$J(X) = \widetilde{J}(X) \oplus \Z$, 
which depends only on the connected component of $X$ containing~$x_0$. 

As in the unbased case, we write $\widetilde{J}(E)$ instead of $\widetilde{J}([E])$ for $E \in \Vect_\R(X)$. 
Note that
\[
J(E \oplus E') = J(E) + J(E'), \qquad 
\widetilde{J}(E \oplus E') = \widetilde{J}(E) + \widetilde{J}(E') \qquad (E, E' \in \Vect_\R(X)). 
\]

In concrete terms, given two vector bundles $E$ and $E'$, we have
\[
\widetilde{J}(E) = \widetilde{J}(E')
\]
if and only if there exist integers $k,k' \geq 0$ such that the sphere bundles $S(E \oplus \underline{\R}^{\oplus k})$ and $S(E' \oplus \underline{\R}^{\oplus k'})$ are fiber-homotopically equivalent.

\begin{defn}
The sphere bundle $S(E)$ associated to a vector bundle~$E$ is \emph{stably fiber-homotopically trivial} if $\widetilde{J}(E) = 0$ in the reduced $J$-group $\widetilde{J}(X)$, or equivalently if the sphere bundle $S(E \oplus \underline{\R}^{\oplus k})$ is fiber-homotopically trivial for some integer $k \geq 0$.
\end{defn}

Fiber-homotopical triviality implies stable fiber-homotopical triviality.

\section{Non-vanishing in the reduced \texorpdfstring{$J$}{J}-group} \label{s:MainStable}

Let $G/H$ be a reductive homogeneous space.
We use the notation and terminology of Sections \ref{s:ReductiveHomogeneousSpaces} and~\ref{s:FiberwiseHomotopy}.
In particular, $N$ denotes the normal bundle of $X = K/K_H$ in $G/H$.

Our main Theorem~\ref{thm:MainTheorem} has the following consequence.

\begin{coro} \label{cor:MainStable}
If the reductive homogeneous space $G/H$ admits compact quotients, then $S(N)$ is stably fiber-homotopically trivial, \ie $\widetilde{J}(N) = 0$.
\end{coro}

Therefore, in order to prove that $G/H$ does not have compact quotients, we shall check that $\widetilde{J}(N) \neq\nolinebreak 0$ for all cases in Theorems \ref{thm:OtherSymmetric}, \ref{thm:NonSymmetric}, and~\ref{thm:IndefiniteGrassmannianRCH}.

In general, determining whether $S(N)$ is fiber-homotopically trivial is a subtle problem, whereas checking the non-vanishing of $\widetilde{J}(N)$ in $\widetilde{J}(X)$ is a more tractable problem, which was thoroughly investigated by topologists in the 1960-70s, most notably Adams (see Sections~\ref{ss:J} and \ref{ss:AdamsOperations}).

\subsection{A small improvement}

By contraposition, Corollary~\ref{cor:MainStable} states that if $\widetilde{J}(N) \neq 0$, then $G/H$ does not admit compact quotients.
In fact, the following slightly stronger statement holds.

\begin{thm} \label{thm:MainStableProd}
Let $G/H$ be a reductive homogeneous space such that $\widetilde{J}(N) \neq 0$. 
Then for any reductive homogeneous space $G'/H'$, the product $(G \times G') / (H \times H')$ does not admit compact quotients. 
\end{thm}

\begin{proof}
Let $K'$ be a maximal compact subgroup of~$G'$ such that $K'_{H'} = K' \cap H'$ is a maximal compact subgroup of $H'$.
Let $X' = K' / K'_{H'}$ and let $N'$ be the normal bundle of $X'$ in $G'/H'$. 
The normal bundle of $X \times X'$ in $(G \times G') / (H \times H')$ is the outer direct sum $N \boxplus N'$, \ie the bundle over $X\times X'$ with total space $N\times N'$.
By Theorem~\ref{thm:MainTheorem}, it suffices to see that $S(N \boxplus N')$ is not fiber-homotopically trivial.

Fixing a point $x' \in X'$ defines an inclusion $X \subset X \times X'$ given by $x \mapsto (x,x')$. 
In particular, any bundle over $X\times X'$ restricts to a bundle over~$X$, which induces a group homomorphism
\[
\widetilde{J}(X \times X') \to \widetilde{J}(X), 
\]
The restriction of $N \boxplus N'$ to~$X$ is identified with $N \oplus \underline{\R}^{\oplus d'}$, where $\underline{\R} = X \times \R$ is the trivial real line bundle over~$X$ and $d' = \dim(X')$.
Since
\[
\widetilde{J}(N \oplus \underline{\R}^{\oplus d'}) = \widetilde{J}(N) \neq 0,
\]
its preimage under the group homomorphism above must satisfy $\widetilde{J}(N \boxplus N') \neq 0$. In particular, $S(N \boxplus\nolinebreak N')$ is not fiber-homotopically trivial.
\end{proof}

As mentioned above, we shall prove in Sections \ref{s:IndefiniteGrassmannians} and~\ref{s:ComputationsKTheory} that $\widetilde{J}(N) \neq 0$ for all cases in Theorems \ref{thm:OtherSymmetric}, \ref{thm:NonSymmetric}, and~\ref{thm:IndefiniteGrassmannianRCH}.
Therefore Theorem~\ref{thm:MainStableProd} implies that in these cases $G/H$, not only has no compact quotients, but also cannot appear as a factor of a reductive homogeneous space admitting compact quotients.

\subsection{Passing to the associated symmetric space}

In the proof of Theorem~\ref{thm:OtherSymmetric}, we shall use the following property.

\begin{prop} \label{prop:AssociatedJTilde}
Let $G/H$ be a reductive symmetric space and $G/H^a$ its associated symmetric space. 
Let $N$ and $N^a$ be the normal bundles of $X$ in $G/H$ and $G/H^a$, respectively. 
Then $\widetilde{J}(N) = 0$ if and only if $\widetilde{J}(N^a) = 0$. 
\end{prop}

The proof of Proposition~\ref{prop:AssociatedJTilde} is based on the following classical fact. 

\begin{lem}\label{lem:HomogeneousBundleTrivial}
Let $L/L'$ be a homogeneous space and $V$ a linear representation of~$L$. Then the associated vector bundle 
$L \times_{L'} V$ over $L/L'$ is trivial. 
\end{lem}

\begin{proof}[Proof of Lemma~\ref{lem:HomogeneousBundleTrivial}]
The following map is well-defined and an isomorphism of vector bundles: 
\[
(L/L') \times V \to L \times_{L'} V, \qquad (\ell L', v) \mapsto [\ell, \ell^{-1}v]. \qedhere
\]
\end{proof}

\begin{proof}[Proof of Proposition \ref{prop:AssociatedJTilde}]
The vector bundles $N$ and $N^a$ are isomorphic to $K \times_{K_H} (\p \cap \q)$ and $K \times_{K_H} (\p \cap \h)$, respectively. 
We thus have
\[
N \oplus N^a = K \times_{K_H} ((\p \cap \q) \oplus (\p \cap \h)) = K \times_{K_H} \p.
\]
The adjoint action of $K_H$ on~$\p$ extends to the adjoint action of $K$ on~$\p$, and so Lemma~\ref{lem:HomogeneousBundleTrivial} applies, showing that $N \oplus N^a$ is trivial. We thus obtain
$\widetilde{J}(N \oplus N^a) = \widetilde{J}(N) + \widetilde{J}(N^a) = 0$. 
In particular, $\widetilde{J}(N)$ vanishes if and only if $\widetilde{J}(N^a)$ does. 
\end{proof}

\section{Compact quotients of indefinite Grassmannians} \label{s:IndefiniteGrassmannians}

In this section we prove Theorem~\ref{thm:IndefiniteGrassmannianRCH} (which contains Theorem~\ref{thm:Hpq}).

\subsection{Tautological vector bundles over Grassmannians}

Let $\K$ be $\R$, $\C$, or the ring $\Ha$ of quaternions. 
Given integers $p,q \geq 1$, we denote by $\Grass_\K(p,q)$ the Grassmannian of $q$-dimensional $\K$-linear subspaces of $\K^{p+q}$.
The \emph{tautological vector bundle} $\Taut_\K(p,q)$ is the $\K$-vector bundle of rank~$q$ over $\Grass_\K(p,q)$ whose fiber over a point $x$ is the linear subspace in $\K^{p+q}$ defining~$x$: 
\[
\Taut_\K(p,q) = \{ (x,v) \in \Grass_\K(p,q) \times \K^{p+q} \mid v \in x \}. 
\]
We equip $\K^{p+q}$ with the standard Hermitian scalar product, and write $\Taut^\perp_\K(p,q)$ for the fiberwise orthogonal complement of $\Taut_\K(p,q)$ in the trivial bundle $\underline{\K}^{\oplus (p+q)} = \Grass_\K(p,q) \times \K^{p+q}$: 
\[
\Taut^\perp_\K(p,q) = \{ (x,v) \in \Grass_\K(p,q) \times \K^{p+q} \mid v \perp x \}.
\]

As in Section~\ref{ss:J}, we consider the homomorphism $\widetilde{J} \colon \KO^0(\Grass_\K(p,q)) \to \widetilde{J}(\Grass_\K(p,q))$, where $\widetilde{J}(\Grass_\K(p,q))$ is the reduced $J$-group of $\Grass_\K(p,q)$.
We denote by $j_\K(p,q)$ the order of $\widetilde{J}(\Taut_\K(p,q))$ in $\widetilde{J}(\Grass_\K(p,q))$; it is known to be finite by work of Atiyah \cite[Prop.\,1.5]{Ati61}.

For $(p,q) = (n,1)$, the Grassmannian $\Grass_\K(n,1)$ is the projective space $\KP^n$; we will write $\Taut_{\KP^n}$ and $j_{\KP^n}$ for $\Taut_\K(n,1)$ and $j_\K(n,1)$, respectively. 
The values of $j_{\KP^n}$ were computed in the 1960--70s by Adams \cite{Ada62} for $\K=\R$, by Adams--Walker \cite{AW65} for $\K=\C$, and by Sigrist--Suter \cite{SS73} for $\K=\Ha$.

\begin{fact}[{\cite{Ada63,AW65,SS73}}] \label{fact:OrderTaut}
With the notation \eqref{eqn:nu} and \eqref{eqn:xi_p}, for any integer $n\geq 1$, we have
\[
j_{\RP^n} = 2^{\nu(n)}, \qquad 
j_{\CP^n} = \prod_{p\text{ prime }\leq n+1} p^{\xi_p(n)}, 
\]
and
\[
j_{\HP^n} = 
\begin{cases} 
j_{\CP^{2n}} & (\xi_2(2n) = 2n+1), \\ 
j_{\CP^{2n}}/2 & (\xi_2(2n) \geq 2n+2).
\end{cases}
\]
\end{fact}

For small $n \geq 1$, the values of $j_{\RP^n}$, $j_{\CP^n}$, and $j_{\HP^n}$ are as follows: 

\begin{center}
\begin{longtable}{|c||c|c|c|c|c|c|c|c|c|} \hline
$n$ & $1$ & $2, 3$ & $4, 5, 6, 7$ & $8$ & $9$ & $10, 11$ & $12, 13, 14, 15$ & $16$ & \dots \\ \hline
$j_{\RP^n}$ & $2$ & $4$ & $8$ & $16$ & $32$ & $64$ & $128$ & $256$ & \dots \\ \hline
\end{longtable}
\begin{longtable}{|c||c|c|c|c|c|c|c|c|} \hline
$n$ & $1$ & $2, 3$ & $4, 5$ & $6, 7$ & $8, 9$ & $10,11$ & \dots \\ \hline
$j_{\CP^n}$ & $2$ & $24$ & $2880$ & $362880$ & $29030400$ & $958003200$ & \dots \\ \hline
\end{longtable}
\begin{longtable}{|c||c|c|c|c|c|c|c|c|} \hline
$n$ & $1$ & $2$ & $3$ & $4$ & $5$ & \dots \\ \hline
$j_{\HP^n}$ & $24$ & $1440$ & $362880$ & $14515200$ & $958003200$ & \dots \\ \hline
\end{longtable}
\end{center}

As far as we know, the orders $j_\K(p,q)$ have not been determined for general $p,q$. However, the following partial results are easy to prove: 

\begin{lem}\label{lem:jEstimate}
Let $\K$ be $\R$, $\C$, or the ring $\Ha$ of quaternions, and let $p,q \geq 1$ be integers.
Then 
\begin{enumerate}[label = {\upshape (\arabic*)}]
  \item $j_\K(p,q) = j_\K(q,p)$,
  \item $j_\K(p,q)$ is divisible by $j_\K(p',q)$ for any $1 \leq p' \leq p$,
  \item $j_\K(p,q)$ is divisible by $j_{\KP^n}$, where $n = \max \{ p,q \}$.
\end{enumerate}
\end{lem}

\begin{proof}
(1) Taking orthogonal complements in $\K^{p+q}$ with respect to the standard Hermitian form gives a diffeomorphism 
\[
\Grass_\K(p,q) \simeq \Grass_\K(q,p). 
\]
Under this diffeomorphism, the tautological vector bundle $\Taut_\K(q,p)$ over $\Grass_\K(q,p)$ 
corresponds to the fiberwise orthogonal complement $\Taut_\K^\perp(p,q)$ of $\Taut_\K(p,q)$ over $\Grass_\K(p,q)$.
Since $\Taut_\K(p,q)\oplus \Taut_\K(p,q)^\perp = \underline{\K}^{\oplus (p+q)}$ is trivial, we have 
\[
\widetilde{J}(\Taut_\K(q,p)) = \widetilde{J}(\Taut^\perp_\K(p,q)) = -\widetilde{J}(\Taut_\K(p,q)).
\]
In particular, $j_\K(p,q) = j_\K(q,p)$.

(2) Let $1\leq p'\leq p$.
The inclusion $\K^{p'+q} \subset \K^{p+q}$ induces an inclusion between the Grassmannians
\[
\Grass_\K(p',q) \subset \Grass_\K(p,q).
\]
In particular, any bundle over $\Grass_\K(p,q)$ restricts to a bundle over $\Grass_\K(p',q)$, which induces a group homomorphism
\[
\widetilde{J}(\Grass_\K(p,q)) \to \widetilde{J}(\Grass_\K(p',q)).
\]
The restriction of $\Taut_\K(p,q)$ to $\Grass_\K(p',q)$ is identified with $\Taut_\K(p',q)$, and so the order $j_\K(p',q)$ of $\widetilde{J}(\Taut_\K(p',q))$ must divide the order $j_\K(p,q)$ of $\widetilde{J}(\Taut_\K(p,q))$.

(3) This is immediate from (1) and (2).
\end{proof}

\subsection{Indefinite Grassmannians} \label{ss:IndefGrassm}

For $p,p',q,q' \in \N$, 
let $\R^{p+p', q+q'}$ be the real vector space $\R^{p+p'+q+q'}$ equipped with the standard quadratic form of signature $(p+p',q+q')$.
The space of nondegenerate linear subspaces of signature $(p',q')$ in $\R^{p+p', q+q'}$, called the \emph{real indefinite Grassmannian}, is the reductive symmetric space
\[
G/H = \OO(p+p', q+q')/(\OO(p,q) \times \OO(p',q')).
\]
The most important example among them is the \emph{pseudo-Riemannian hyperbolic space} $\Hyp^{p,q}_\R$, corresponding to the case $(p',q') = (0,1)$:
\[
\Hyp^{p,q}_\R = \OO(p,q+1) / (\OO(p,q) \times \OO(1)). 
\]

One can similarly define the \emph{complex} and the \emph{quaternionic indefinite Grassmannians}, which are respectively the reductive symmetric spaces
\[
\U(p+p', q+q') / (\U(p,q) \times \U(p',q')), \qquad 
\Sp(p+p', q+q') / (\Sp(p,q) \times \Sp(p',q')). 
\]
When $(p',q') = (0,1)$, they are respectively the \emph{complex} and the \emph{quaternionic pseudo-Riemannian hyperbolic spaces}
\[
\Hyp^{p,q}_\C = \U(p,q+1) / (\U(p,q) \times \U(1)), \qquad 
\Hyp^{p,q}_\Ha = \Sp(p,q+1) / (\Sp(p,q) \times \Sp(1)).
\]

\subsection{Compact quotients of indefinite Grassmannians}

We now apply Corollary~\ref{cor:MainStable} to the reductive symmetric space $G/H = \clubsuit(p+p',q+q') / (\clubsuit(p,q) \times \clubsuit(p',q))$, where $\clubsuit = \OO$, $\U$, or $\Sp$, corresponding to an indefinite Grassmannian over $\R$, $\C$, or~$\Ha$ as in Section~\ref{ss:IndefGrassm}.

Kobayashi \cite[Ex.\,1.7]{Kob92coh} proved that $G/H$ does not admit compact quotients unless at least one of $p,p',q,q'$ is zero, so we now assume that one of them is zero.
By symmetry, and up to switching the sign of the Hermitian form defining~$G$, we may assume that $p' = 0$.
If one of $p,q,q'$ is also zero, then $G/H$ trivially admits compact quotients.
We thus assume $p,q,q' \geq 1$.  

In this setting, the maximal compact subspace $K/K_H$ of $G/H$ is
\[
K/K_H = \clubsuit(q+q') / (\clubsuit(q) \times \clubsuit(q'))
\]
where $\clubsuit = \OO$, $\U$, or $\Sp$, which is the (ordinary) Grassmannian of $q'$-planes in $\K^{q+q'}$ where $\K=\R$, $\C$, or~$\Ha$.
The normal bundle $N$ is associated to the representation of $\clubsuit(q) \times \clubsuit(q')$ on 
\[
\p \cap \q \simeq (\K^{q'}_\std)^{\oplus p},
\]
where $\K^{q'}_\std$ is the standard representation of $\clubsuit(q')$ on $\K^{q'}$.
Hence, $N$ is the sum of $p$ copies of the tautological vector bundle $\Taut_\K(q,q')$:
\[
N \simeq \Taut_\K(q,q')^{\oplus p}.
\]
Corollary~\ref{cor:MainStable} then implies the following.

\begin{thm} \label{thm:IndefiniteGrassmannianR}
Let $p,q,q' \geq 1$. 
If the reductive symmetric space $\OO(p,q+q')/(\OO(p,q) \times \OO(q'))$ (\resp $\U(p, q+q')/(\U(p,q) \times \U(q'))$, \resp $\Sp(p, q+q')/(\Sp(p,q) \times \Sp(q'))$) admits compact quotients, then $p$ is divisible by $j_\R(q,q')$ (\resp $j_\C(q,q')$, \resp $j_\Ha(q,q')$).
\end{thm}

For the pseudo-Riemannian hyperbolic space $\Hyp^{p,q}_\R = \OO(p,q+1)/(\OO(p,q) \times \OO(1))$, Theorem~\ref{thm:IndefiniteGrassmannianR} with $q'=1$ and Adams's computation of $j_\R(q,1) = j_{\RP^q}$ (Fact~\ref{fact:OrderTaut}) yield Theorem~\ref{thm:Hpq}.
Similarly, Theorem~\ref{thm:IndefiniteGrassmannianR} with $q'=1$ and the computation of $j_\C(q,1) = j_{\CP^q}$ (\resp $j_\Ha(q,1) = j_{\HP^q}$) by Adams--Walker (\resp Sigrist--Suter) (Fact~\ref{fact:OrderTaut}) yield that if $\Hyp^{p,q}_{\C}$ (\resp $\Hyp^{p,q}_{\Ha}$) admits compact quotients, then $p$ is divisible by
\[
\prod_{p\text{ prime }\leq n+1} p^{\xi_p(n)} \quad \text{\bigg(\resp } 2^{2n+1} \prod_{p\text{ odd prime }\leq (n+1)/2} p^{\xi_p(2n)}\text{\bigg)} .
\]
For general indefinite Grassmannian $G/H = \OO(p,q+q') / (\OO(p,q) \times \OO(q'))$, although we do not know the exact value of $j_\R(q,q')$, the constraint on $(p,q)$ in Theorem~\ref{thm:IndefiniteGrassmannianR} is at least as strong as for $\Hyp^{p,q}_{\K}$ by Lemma~\ref{lem:jEstimate}, which proves Theorem~\ref{thm:IndefiniteGrassmannianRCH}.

\begin{rmk}
It was previously known from work of Kobayashi \cite[Ex.\,1.7]{Kob92coh} that for $\clubsuit = \OO, \U$, or $\Sp$, the homogeneous space $\clubsuit(p,q+q')/(\clubsuit(p,q) \times \clubsuit(q'))$ does not admit compact quotients when $p,q,q'\geq 1$ satisfy $p < q+q'$.
Furthermore, it was known from work of the second and third authors \cite[Th.\,7.2~(1)]{Tho15+}, \cite[Cor.\,6.5]{Mor17} that $\OO(p,q+q')/(\OO(p,q) \times \OO(q'))$ does not admit compact quotients for odd~$p$; this generalized earlier results by Kobayashi--Ono \cite[Cor.\,5]{KO90} and \cite[Cor.\,1.4~(4)]{Mor15}.
Theorem~\ref{thm:IndefiniteGrassmannianR} recovers these non-existence results and is significantly better for large~$q$. 

On the other hand, in addition to $\Hyp^{2k,1}_\R$, $\Hyp^{4k,3}_\R$, and $\Hyp^{8,7}_\R$ which we mentioned in Section~\ref{sss:Hpq}, it is known that
$\OO(4,3)/(\OO(4,1) \times \OO(2))$ and $\OO(4,4)/(\OO(4,1) \times \OO(3))$ admit compact quotients \cite[Cor.\,4.7]{Kob97}, and similarly for $\Hyp^{2k,1}_\C = \U(2k,2) / (\U(2k,1) \times \U(1))$ \cite[Prop.\,4.9]{Kob89}. 
\end{rmk}

\begin{rmk}
Interestingly, while the sphere bundle of $\Taut_{\CP^q}^{\oplus p}$ is stably fiber-homotopically trivial when $j_{\CP^q}$ divides $p$, it is never trivial as a topological fiber bundle for $q \geq 2$, as seen from computing its rational Pontryagin classes.
In a forthcoming paper \cite{KMT}, we will use this and the Anosov property from \cite{KT24+} to prove that $\Hyp^{p,q}_\C$ admits compact quotients if and only if $p$ is even and $q = 1$.
We will similarly prove that $\Hyp^{p,q}_\Ha$ never admits compact quotients for any $p,q \geq 1$.
\end{rmk}

\section{Other computations and applications} \label{s:ComputationsKTheory}

In this section, we carry further computations of the reduced $J$-groups (see Section~\ref{ss:J}) for vector bundles over Grassmannians and other compact symmetric spaces.
Using Corollary~\ref{cor:MainStable}, we deduce Theorems \ref{thm:OtherSymmetric} and~\ref{thm:NonSymmetric} on the non-existence of compact quotients.

We will treat each case of Theorems \ref{thm:OtherSymmetric} and~\ref{thm:NonSymmetric} independently, ordering the proofs by increasing degree of complexity of the computations. The proofs will freely use the notation and results in $\Ktheory$-theory from Section~\ref{s:RemindersKTheory} and Appendix~\ref{s:AppendixKTheory}.

\subsection{Comments on Theorems \ref{thm:OtherSymmetric} and~\ref{thm:NonSymmetric}}

Before starting computations, here are a few remarks.

\begin{rmk}
Let $p \geq q \geq 1$. 
Prior to Theorem~\ref{thm:OtherSymmetric}, it was known that the following reductive homogeneous spaces do not admit compact quotients: 
\begin{itemize}
  \item $\OO(p+q, \C) / (\OO(p, \C) \times \OO(q, \C))$ with $(p,q) \neq (4k-1,1), (1,1)$ (Kobayashi \cite[Ex.\,1.9]{Kob92coh}, Benoist \cite[\S\,1.2, Cor.\,2]{Ben96});
  \item $\OO(p+q, \C) / \OO(p,q)$ with $(p,q) \neq (2k-1,1)$ (Tholozan \cite[Th.\,7.3~(3)]{Tho15+}, Morita \cite[Cor.\,6.1]{Mor17});
  \item $\OO^\ast(2p+2q) / (\OO^\ast(2p) \times \OO^\ast(2q))$ with $(p, q) \neq (2k-1, 1)$ (Kobayashi \cite[Th.\,1.4]{Kob92coh});
  \item $\OO^\ast(2p+2q) / \U(p,q)$ with $p \leq 2q$ (Kobayashi \cite[Ex.\,1.8]{Kob92coh});
  \item $\Sp(2p+2q, \R) / (\Sp(2p, \R) \times \Sp(2q, \R))$ for any $p \geq q \geq 1$, including the case $(p,q) = (3,1)$ (Kobayashi \cite[Cor.\,4.4]{Kob89});
  \item $\Sp(2p+2q, \R) / \U(p,q)$ with $p=q$ (Kobayashi \cite[Th.\,1.4]{Kob92coh});
  \item $\SL(2p, \C) / \Sp(2p, \C)$ for any $p \geq 2$ (Benoist \cite[\S\,1.1, Cor.\,2]{Ben96}). 
\end{itemize}

On the other hand, it is known that the following reductive homogeneous spaces admit standard compact quotients: 
\begin{itemize}
  \item $\OO(4, \C) / (\OO(3, \C) \times \OO(1, \C))$ (locally isomorphic to $(\OO(3,\C) \times \OO(3,\C)) / \Diag(\OO(3,\C))$);
  \item $\OO(8, \C) / (\OO(7, \C) \times \OO(1, \C))$ (Kobayashi--Yoshino \cite[Th.\,4.2.1]{KY05});
  \item $\OO(4, \C) / \OO(3,1)$ (locally isomorphic to $(\OO(3,\C) \times \OO(3,\C)) / \Diag(\OO(3,\C))$);
  \item $\OO(8, \C) / \OO(7,1)$ (Kobayashi--Yoshino \cite[Th.\,4.10.1]{KY05});
  \item $\OO^\ast(8) / (\OO^\ast(6) \times \OO^\ast(2))$ (locally isomorphic to $\SO_0(6,2) / \U(3,1)$ below);
  \item $\OO^\ast(8) / \U(3,1)$ (locally isomorphic to $\SO_0(6,2) / \U(3,1)$ below);
  \item $\SO_0(2p,2) / \U(p,1)$ for any~$p$ (Kobayashi \cite[Prop.\,4.9~(2)]{Kob89});
  \item $\SU(2p,2) / \Sp(p,1)$ for any~$p$ (Kobayashi \cite[Prop.\,4.9~(1)]{Kob89}).
\end{itemize}

Our method does not apply to $\Sp(8,\R) / \U(3,1)$: it is still an open question whether this homogeneous space admits compact quotients. 
\end{rmk}

\begin{rmk} \label{rmk:OlderResultsSL/SL}
Let $p \geq 2$ and $q \geq 1$.
The non-existence of compact quotients of $\SL(p+q,\K)/\SL(p,\K)$ for $\K = \R$ was conjectured by Kobayashi in 1990.
Prior to Theorem~\ref{thm:NonSymmetric}, it was known for $\K = \R, \C$, or~$\Ha$ in the following cases: 
\begin{itemize}
  \item $q \geq \max \{ p, 3 \}$ (Labourie--Mozes--Zimmer \cite[Th.\,1.1]{LMZ95}); this is a generalization of earlier work by Zimmer \cite{Zim94};
  \item $q \geq 3$ and $\K = \R$ or $\C$ (Labourie--Zimmer \cite[Th.\,1]{LZ95}; they only discussed the case $\K = \R$, but their proof equally works over $\C$); this is a further improvement of \cite{LMZ95}, but it does not apply to $\K = \Ha$;
  \item $(p,q) = (2k, 1)$ (Benoist \cite[\S\,1.2, Cor.\,1]{Ben96});
  \item $p = 2$ and $q \geq 2$ (Shalom \cite[Cor.\ to Th.\,1.7]{Sha00}; he only discussed the case where $\K = \R$ or~$\C$, but his proof equally works over $\Ha$); 
  \item $p$ is even and $\K = \R$ (Tholozan \cite[Th.\,7.2~(2)]{Tho15+}, Morita \cite[Cor.\,6.7]{Mor17}); this result does not apply to $\K = \C$ or $\Ha$;
  \item $q \geq (p + \varepsilon(p, \K))/4$ (Morita \cite[Th.\,1.4]{Mor24}), where $\varepsilon(p, \K)$ is as in the Table~\ref{table:eps-p-K} below; this is an improvement of earlier work by Kobayashi \cite[Ex.\,5.19]{Kob97}, based on \cite[Th.\,1.5]{Kob92coh};
  \item $(p,q) = (2k+1, 1)$ (Kassel--Tholozan \cite[Cor.\,1.16]{KT24+}). 
\end{itemize}

Our method does not apply to $\SL(4, \R) / \SL(3, \R)$, $\SL(8, \R) / \SL(7, \R)$, and $\Sp(4, \R) / \Sp(2, \R)$. Fortunately, these three cases were recently settled by the first and third authors \cite[Cor.\,1.16]{KT24+}.
\end{rmk}

\begin{center}
\begin{longtable}{|c||c|c|c|c|}\hline
$p \bmod 4$ & $0$ & $1$ & $2$ & $3$ \\ \hline
$\K = \R$ & $4$ & $7$ & $2$ & $9$ \\ \hline
$\K = \C$ & $0$ & $7$ & $2$ & $5$ \\ \hline
$\K = \Ha$ & $0$ & $3$ & $2$ & $5$  \\ \hline
\caption{The value of $\varepsilon(p,\K)$}
\label{table:eps-p-K}
\end{longtable}
\end{center}

The question of the existence of compact quotients of $\SL(p+q,\R)/\SL(p,\R)$ (and more generally, of closed manifolds locally modelled on this homogeneous space) has attracted a lot of interest in the past 35 years.
This is an instance of a homogeneous space $G/H$ where $H$ has a noncompact centralizer $J$ (namely, the complementary block-diagonal subgroup $\GL(q,\R)$). This centralizer $J$ would act on any closed manifold locally modelled on $G/H$, making this space a potential model for ``geometric'' dynamical systems. 

For instance, an important step of Benoist--Foulon--Labourie's classification of contact Anosov flows with smooth stable distributions \cite{BFL92} is that on a hypothetical closed manifold locally modelled on $\SL(p+1,\R)/\SL(p,\R)$, the flow induced by $J$ could not be Anosov. Another occurrence of such homogeneous spaces is in the theory of translation surfaces (see \cite{Fil24} for a survey on this topic). The strata of the moduli space of translation surfaces of genus $g$ carry locally homogeneous geometric actions of $\SL(2,\R)$. In particular, the minimal stratum is locally modelled on $\Sp(2g,\R)/\Sp(2g-2,\R)$.

Unfortunately, unlike the Labourie--Mozes--Zimmer approach, our method does not say anything about manifolds locally modelled on $G/H$ which are not a priori quotients of $G/H$. It is conjectured, however, that any closed manifold locally modelled on a reductive homogeneous space is a compact quotient of that space.

\subsection{Computations in real Grassmannians} \label{ss:GrassR}

The purpose of this Section~\ref{ss:GrassR} is to prove Theorem~\ref{thm:OtherSymmetric}~(1) and Theorem~\ref{thm:NonSymmetric}~(1)--(2). 

For $p,q \geq 1$, we regard the real projective space $\RP^p = \Grass_\R(p,1)$ as a submanifold of the real Grassmannian $\Grass_\R(p,q)$ in an obvious way. We set $t = [ \Taut_{\RP^p} ] - 1 \in \widetilde{\KO}^0(\RP^p)$. 

\begin{lem} \label{lem:RestGrassR}
For $p,q\geq 1$, the following identities hold in $\KO^0(\RP^p)$: 
\begin{enumerate}[label = {\upshape (\arabic*)}]
\item $[ \Taut_\R(p,q) \otimes_\R \Taut^\perp_\R(p,q) ]_{|\RP^p} = (p-q+2) t + pq$. 
\item $[ \Sym^2 \Taut_\R(p,q) ]_{|\RP^p} = (q-1) t + q(q+1)/2$. 
\item $[ \Lambda^2 \Taut_\R(p,q) ]_{|\RP^p} = (q-1) t + q(q-1)/2$. 
\end{enumerate}
\end{lem}

\begin{proof}
(1) We have
\begin{align*}
[ \Taut_\R(p,q) ]_{|\RP^p} &= [ \Taut_{\RP^p} ] + q-1, \\
[ \Taut^\perp_{\R}(p,q) ]_{|\RP^p} &= p+q - [ \Taut_\R(p,q) ]_{|\RP^p}
= -[ \Taut_{\RP^p} ] + p+1.
\end{align*}
Since $\Taut_{\RP^p}$ is a real line bundle, we have 
$\Taut_{\RP^p} \otimes \Taut_{\RP^p} \simeq \underline{\R}$. 
We thus obtain 
\begin{align*}
[ \Taut_\R(p,q) \otimes_\R \Taut^\perp_\R(p,q) ]_{|\RP^p}
&= ([ \Taut_{\RP^p} ] + q-1)(-[ \Taut_{\RP^p} ] + p+1) \\ 
&= -1 + (p-q+2) [ \Taut_{\RP^p} ] + (q-1)(p+1) \\
&= (p-q+2)t + pq. 
\end{align*}

(2) Since $\Taut_{\RP^p}$ is a real line bundle, we have 
$\Sym^2 \Taut_{\RP^p} \simeq \underline{\R}$. 
Hence, 
\begin{align*}
\Sym^2 \Taut_\R(p,q)_{|\RP^p} &\simeq \Sym^2 (\Taut_{\RP^p} \oplus \underline{\R}^{\oplus (q-1)}) \\
&\simeq (\Sym^2 \Taut_{\RP^p}) \oplus (\Taut_{\RP^p} \otimes \underline{\R}^{\oplus (q-1)}) \oplus (\Sym^2 (\underline{\R}^{\oplus (q-1)})) \\
&\simeq (\Taut_{\RP^p})^{\oplus (q-1)} \oplus \underline{\R}^{\oplus (q^2-q+2)/2}.
\end{align*}
Passing to $\KO^0$ yields the desired identity.

(3) Since $\Taut_{\RP^p}$ is a real line bundle, its second exterior power 
$\Lambda^2 \Taut_{\RP^p}$ vanishes.  
Hence, 
\begin{align*}
\Lambda^2 \Taut_\R(p,q)_{|\RP^p} &\simeq \Lambda^2 (\Taut_{\RP^p} \oplus \underline{\R}^{\oplus (q-1)}) \\
&\simeq (\Lambda^2 \Taut_{\RP^p}) \oplus (\Taut_{\RP^p} \otimes \underline{\R}^{\oplus (q-1)}) \oplus (\Lambda^2 (\underline{\R}^{\oplus (q-1)})) \\
&\simeq (\Taut_{\RP^p})^{\oplus (q-1)} \oplus \underline{\R}^{\oplus (q-1)(q-2)/2}.
\end{align*}
Passing to $\KO^0$ yields the desired identity.
\end{proof}

\subsubsection{Proof of Theorem~\ref{thm:OtherSymmetric}~(1)} \label{sss:ComplexSpheres}

For $p \geq q \geq 1$, let
\[
G/H = \OO(p+q, \C) / (\OO(p, \C) \times \OO(q, \C)). 
\]
The compact symmetric space $X = K/K_H$ is the real Grassmannian: 
\[
X = \OO(p+q) / (\OO(p) \times \OO(q)) = \Grass_\R(p,q). 
\]
Furthermore, the normal bundle $N$ is the tangent bundle of~$X$ by the following lemma: 

\begin{lem} \label{lem:Normal=Tangent}
Suppose that $G/H$ is a complex reductive symmetric space. Then, as a real vector bundle, the normal bundle $N$ is $K$-equivariantly isomorphic to the tangent bundle $TX$. 
\end{lem}

\begin{proof}
Recall that 
\[
N = K \times_{K_H} (\p \cap \q), \qquad 
TX = K \times_{K_H} (\kk \cap \q).
\]
Thus, it suffices to find a $K_H$-equivariant $\R$-linear isomorphism from $\p \cap \q$ to $\kk \cap \q$. 
Since $G$ and $H$ are complex, $\q$ is a $\C$-linear subspace of $\g$, and $\p = i \kk$. Hence
\[
i (\p \cap \q) = i (i \kk \cap \q) = \kk \cap \q. 
\]
Therefore, the multiplication by $i$ gives the desired isomorphism between $\p \cap \q$ and $\kk \cap \q$. 
\end{proof}

Theorem~\ref{thm:OtherSymmetric}~(1) is an immediate consequence of Corollary~\ref{cor:MainStable}, Proposition~\ref{prop:AssociatedJTilde}, and the following.

\begin{prop} \label{prop:TangentGrassR}
Let $p\geq q\geq 1$. Then
$\widetilde{J}(T\Grass_\R(p,q)) \neq 0$
except if $(p,q) = (1,1), (3,1)$, or $(7,1)$. 
\end{prop}

\begin{proof}
We have an isomorphism
\[
N = T\Grass_\R(p,q) \simeq \Taut_\R^{\perp}(p,q) \otimes_\R \Taut_\R(p,q).
\]
Let us consider the restriction of $N$ to the subspace 
$\RP^p = \Grass_\R(p,1) \subset \Grass_\R(p,q)$. 
It follows from Lemma~\ref{lem:RestGrassR}~(1) that 
\[
\widetilde{J}(N)_{|\RP^p} = (p-q+2) \cdot \widetilde{J}(\Taut_{\RP^p}). 
\]
Thus, if $\widetilde{J}(N) = 0$, then 
$j_\R(p,1)$ must divide $p-q+2$. By Adams's computation (Fact~\ref{fact:OrderTaut}), this happens only when $(p,q) = (1,1), (3,1), (7,1)$. 
\end{proof}

\begin{rmk}
For $(p,q) = (1,1), (3,1)$, or $(7,1)$, the normal bundle $N$ is the tangent bundle to $\RP^1, \RP^3$, or $\RP^7$, which is known to be trivial.
\end{rmk}

\subsubsection{Proof of Theorem~\ref{thm:NonSymmetric}~(1)} \label{sss:SLR/SLR}

Let $p\geq 2$ and $q\geq 1$.
If a discrete subgroup $\Gamma$ of $\SL(p+q, \R)$ acts properly discontinuously and cocompactly on $\SL(p+q, \R)/\SL(p, \R)$, then it also acts properly discontinuously and cocompactly on $\GL(p+q, \R) / (\GL(p, \R) \times \OO(q))$. 
Thus, it suffices to show that 
\[
G/H = \GL(p+q, \R) / (\GL(p, \R) \times \OO(q))
\]
does not admit compact quotients except possibly when $(p,q) = (3,1), (7,1)$. 
The compact symmetric space $X = K/K_H$ is then the real Grassmannian
\[
X = \OO(p+q) / (\OO(p) \times \OO(q)) = \Grass_\R(p,q).
\]
The normal bundle $N$ comes from the representation of $\OO(p) \times \OO(q)$ on 
\[
\p \cap \q \simeq (\R^p_\std \otimes_\R \R^q_\std) \oplus \Sym^2 (\R^q_\std),
\]
where $\R^p_\std$ and $\R^q_\std$ are the standard representations of $\OO(p)$ and $\OO(q)$, respectively. 
Hence, 
\[
N = (\Taut^\perp_\R(p,q) \otimes_\R \Taut_\R(p,q)) \oplus \Sym^2 \Taut_\R(p,q). 
\]
Therefore Theorem~\ref{thm:NonSymmetric}~(1) is an immediate consequence of Corollary~\ref{cor:MainStable} and of the following.

\begin{prop} \label{prop:J-GLR/GLR}
Let $p \geq 2$ and $q \geq 1$. Then 
\[
\widetilde{J}(\Taut^\perp_\R(p,q) \otimes_\R \Taut_\R(p,q)) + \widetilde{J}(\Sym^2 \Taut_\R(p,q)) \neq 0
\]
except if $(p,q) = (3,1), (7,1)$. 
\end{prop}

\begin{proof}
Let us first consider the restriction of $N$ to the subspace $\RP^p = \Grass_\R(p, 1)$. 
It follows from Lemma~\ref{lem:RestGrassR}~(1) and (2) that
\[
[N]_{|\RP^p} = (p+1) t + pq + \frac{q(q+1)}{2},
\]
where $t = [ \Taut_{\RP^p} ] - 1 \in \KO^0(\RP^p)$. 
In particular, we have $\widetilde{J}(N)_{|\RP^p} = (p+1) \cdot \widetilde{J}(t)$. 
Thus, if $\widetilde{J}(N) = 0$, then $j_{\RP^p}$ must divide $p+1$. By Fact~\ref{fact:OrderTaut}, this happens only when $p = 3,7$. 

We now restrict $N$ to another subspace $\RP^q = \Grass_\R(1, q)$. 
Observe that the action of $K \cap H = \OO(p) \times \OO(q)$ on
\[
\Sym^2 (\R^p_\std) \oplus (\p \cap \q) \simeq \Sym^2 (\R^p_\std \oplus \R^q_\std)
\]
extends to the action of $K = \OO(p+q)$. Thus, it follows from Lemma~\ref{lem:HomogeneousBundleTrivial} that 
\[
[N] = \frac{(p+q)(p+q+1)}{2} - [ K \times_{K_H} \Sym^2(\R^p_\std) ]
\]
in $\KO^0(\Grass_\R(p,q))$. By Lemma~\ref{lem:RestGrassR}~(2), we have 
\begin{align*}
[N]_{|\RP^q} &= 
\frac{(p+q)(p+q+1)}{2} -
[\Sym^2(\Taut_\R(p,q))]_{|\RP^q} \\
&= - (p-1) t' + \frac{(p+q)(p+q+1)}{2} - \frac{p(p+1)}{2},
\end{align*}
where $t' = [ \Taut_{\RP^q} ] - 1 \in \KO^0(\RP^q)$. 
In particular, we have $\widetilde{J}(N)_{|\RP^q} = -(p-1) \cdot \widetilde{J}(t')$. 
Thus, if $\widetilde{J}(N) = 0$, then $j_{\RP^q}$ must divide $p-1$. 
Again by Fact~\ref{fact:OrderTaut}, 
for $p = 3,7$, this happens only when $q = 1$.
\end{proof}

\begin{rmk}
For $(p,q) = (3,1)$ or $(7,1)$, the normal bundle $N$ is a direct sum of the trivial bundle $\underline{\R}$ and of the tangent bundle to $\RP^3$ or $\RP^7$, hence it is trivial. 
\end{rmk}

\subsubsection{Proof of Theorem~\ref{thm:NonSymmetric}~(2)}

As in Section~\ref{sss:SLR/SLR}, we may replace $\OO(p+q, \C) / \OO(p, \C)$ with
\[
G/H = \OO(p+q, \C) / (\OO(p, \C) \times \OO(q)) \qquad (p \geq 2,\, q \geq 1).
\]
The compact symmetric space $X = K/K_H$ is then the real Grassmannian 
\[
X = \OO(p+q) / (\OO(p) \times \OO(q)) = \Grass_\R(p,q).
\]
The normal bundle $N$ comes from the representation of $\OO(p) \times \OO(q)$ on 
\[
\p \cap \q \simeq (\R^p_\std \otimes_\R \R^q_\std) \oplus \Lambda^2 (\R^q_\std),
\]
hence 
\[
N = (\Taut^\perp_\R(p,q) \otimes_\R \Taut_\R(p,q)) \oplus \Lambda^2 \Taut_\R(p,q). 
\]
Therefore Theorem~\ref{thm:NonSymmetric}~(2) is an immediate consequence of Corollary~\ref{cor:MainStable} and of the following.

\begin{prop}
Let $p \geq 2$ and $q \geq 1$. Then 
\[
\widetilde{J}(\Taut^\perp_\R(p,q) \otimes_\R \Taut_\R(p,q)) + \widetilde{J}(\Lambda^2 \Taut_\R(p,q)) \neq 0
\]
except when $(p,q) = (3,1), (7,1)$. 
\end{prop}

\begin{proof}
It follows from Lemma~\ref{lem:RestGrassR}~(1) and (3) that
\[
[N]_{|\RP^p} = (p+1) t + pq + \frac{q(q-1)}{2},
\]
where $t = [ \Taut_{\RP^p} ] - 1 \in \KO^0(\RP^p)$. 
In particular, we have $\widetilde{J}(N)_{|\RP^p} = (p+1) \cdot \widetilde{J}(t)$. 
As we saw in the proof of Proposition~\ref{prop:J-GLR/GLR}, this vanishes only when $p = 3, 7$. 

Observe that the action of $K \cap H = \OO(p) \times \OO(q)$ on 
\[
\Lambda^2 (\R^p_\std) \oplus (\p \cap \q) \simeq \Lambda^2 (\R^p_\std \oplus \R^q_\std)
\]
extends to the action of $K = \OO(p+q)$. It follows from Lemma~\ref{lem:HomogeneousBundleTrivial} that 
\[
[N] = \frac{(p+q)(p+q-1)}{2} - [\Lambda^2 (\Taut_\R(p,q)) ]
\]
in $\KO^0(\Grass_\R(p,q))$. By Lemma~\ref{lem:RestGrassR}~(3), we have 
\[
[N]_{|\RP^q} = - (p-1) t' + \frac{(p+q)(p+q-1)}{2} - \frac{p(p-1)}{2},
\]
where $t' = [ \Taut_{\RP^q} ] - 1 \in \KO^0(\RP^q)$. 
In particular, we have $\widetilde{J}(N)_{|\RP^q} = -(p-1) \cdot \widetilde{J}(t')$. As in the proof of Proposition~\ref{prop:J-GLR/GLR}, we can conclude that $(p,q) = (3,1), (7,1)$. 
\end{proof}

\begin{rmk}
For $(p,q) = (3,1)$ or $(7,1)$, the normal bundle $N$ is the tangent bundle to $\RP^3$ or $\RP^7$, hence it is trivial. 
\end{rmk}

\subsection{Computations in complex Grassmannians} \label{ss:GrassC}

The purpose of this Section~\ref{ss:GrassC} is to prove Theorem~\ref{thm:OtherSymmetric}~(2)--(3) and Theorem~\ref{thm:NonSymmetric}~(3)--(5). 
We set $u = [ \Taut_{\CP^p} ] - 1 \in \widetilde{\KU}^0(\CP^p)$. 

We will use the following lemma:

\begin{lem} \label{lem:RestGrassC}
The following identities hold in $\KU^0(\CP^p)$: 
\begin{enumerate}[label = {\upshape (\arabic*)}]
\item $[ \Taut^\perp_\C(p,q) \otimes_\C \Taut_\C(p,q) ]_{|\CP^p} = -u^2 + (p-q) u + pq$. 
\item $[ \Taut^\perp_\C(p,q) \otimes_\C \overline{\Taut_\C(p,q)} ]_{|\CP^p} = -(q-1) u + (p+1) \overline{u} + pq$. 
\item $[ \Sym^2 \Taut_\C(p,q) ]_{|\CP^p} = u^2 + (q+1) u + q(q+1)/2$.
\item $[ \Lambda^2 \Taut_\C(p,q) ]_{|\CP^p} = (q-1) u + q(q-1)/2$.
\end{enumerate}
Also, the following identity holds in $\KO^0(\CP^p)$: 
\begin{enumerate}
\item[{\upshape (5)}] $[ \Herm (\Taut_\C(p,q)) ]_{|\CP^p} = (q-1) \cdot \rr u + q^2$.
\end{enumerate}
\end{lem}

\begin{proof}
We only give a proof of (1) and (5), for (2)--(4) can be proved in a similar way as Lemma~\ref{lem:RestGrassR}.

(1) We have 
\begin{align*}
\br{\Taut_\C(p,q)}_{|\CP^p} &= \br{\Taut_{\CP^p}} + q-1 = u + q, \\
\br{\Taut_\C(p,q)}_{|\CP^p} 
&= p+q - \br{\Taut_\C(p,q)}_{|\CP^p} 
= -u + p.
\end{align*}
We thus see that 
\[
[\Taut^\perp_\C(p,q) \otimes_\C \Taut_\C(p,q) ]_{|\CP^p}
= (u+q)(-u+p) = -u^2 + (p-q)u + pq. 
\]

(5) Since $\Taut_{\CP^p}$ is a complex line bundle, we have $\Herm(\Taut_{\CP^p}) \simeq \underline{\R}$. Hence, 
\begin{align*}
\Herm(\Taut_\C(p,q))_{|\CP^p}
&\simeq \Herm(\Taut_{\CP^p} \oplus\,\underline{\C}^{\oplus (q-1)}) \\
&\simeq \Herm(\Taut_{\CP^p}) \oplus \rr (\Taut_{\CP^p} \otimes_\C\, \overline{\underline{\C}^{\oplus (q-1)}}) \oplus \Herm(\underline{\C}^{\oplus (q-1)}) \\
&\simeq (\rr \Taut_{\CP^p})^{\oplus (q-1)} \oplus \underline{\R}^{\oplus(q^2 - 2q + 2)}.
\end{align*}
Passing to $\KO^0$ yields the desired identity. 
\end{proof}

\subsubsection{Proof of Theorem~\ref{thm:NonSymmetric}~(3)} \label{sss:GLC/GLC}

As in Section~\ref{sss:SLR/SLR}, we may replace $\SL(p+q, \C) / \SL(p, \C)$ with 
\[
G/H = \GL(p+q, \C) / (\GL(p, \C) \times \U(q)) \qquad (p \geq 2,\, q \geq 1).
\]
The compact symmetric space $X = K/K_H$ is then the complex Grassmannian 
\[
X = \U(p+q) / (\U(p) \times \U(q)) = \Grass_\C(p,q).
\]
The normal bundle $N$ is associated to the representation of $\U(p) \times \U(q)$ on 
\[
\p \cap \q \simeq (\C^p_\std \otimes_\C \overline{\C^q_\std}) \oplus \Herm(\C^q_\std),
\]
where $\C^p_\std$ and $\C^q_\std$ are the standard representations of $\U(p)$ and $\U(q)$, respectively. 
Hence, 
\[
N = (\Taut^\perp_\C(p,q) \otimes_\C \overline{\Taut_\C(p,q)})) \oplus \Herm(\Taut_\C(p,q)). 
\]
Therefore Theorem~\ref{thm:NonSymmetric}~(3) is an immediate consequence of Corollary~\ref{cor:MainStable} and of the following.

\begin{prop} \label{prop:J-GLC/GLC}
Let $p \geq 2$ and $q \geq 1$. Then 
\[
\widetilde{J}(\Taut^\perp_\C(p,q) \otimes_\C \overline{\Taut_\C(p,q)}) + \widetilde{J}(\Herm(\Taut_\C(p,q))) \neq 0.
\]
\end{prop}

\begin{proof}
It follows from Lemma~\ref{lem:RestGrassC}~(2) and (5) that
\[
[N]_{|\CP^p} = (p+1) \cdot \rr u + 2pq + q^2\]
in $\KO^0(\CP^p)$. 
In particular, we have $\widetilde{J}(N)_{|\CP^p} = (p+1) \cdot \widetilde{J}(\rr u)$. 
Thus, if $\widetilde{J}(N) = 0$, then $j_{\CP^p}$ must divide $p+1$. Adams--Walker's computation (Fact~\ref{fact:OrderTaut}) tells us that this is impossible for $p \geq 2$. 
\end{proof}

\subsubsection{Proof of Theorem~\ref{thm:OtherSymmetric}~(2)--(3)} \label{sss:OtherSymmetric-2-3}

We first remark that one may assume $(p,q) \neq (1,1)$ in Theorem~\ref{thm:OtherSymmetric}~(3). 
Indeed, if $(p,q) = (1,1)$, then $G$, $H$, and $H^a$ are locally isomorphic to 
$\OO(2,3)$, $\OO(2,2) \times \OO(1)$, and $\OO(2,1) \times \OO(2)$, respectively, 
hence Theorem~\ref{thm:IndefiniteGrassmannianRCH} (which is already proved in Section~\ref{s:IndefiniteGrassmannians}) applies. 

Let
\[
G/H = \OO^\ast(2p+2q) / (\OO^\ast(2p) \times \OO^\ast(2q)) \qquad (p \geq q \geq 1) 
\]
or 
\[
G/H = \Sp(2p+2q, \R) / (\Sp(2p, \R) \times \Sp(2q, \R)) \qquad (p \geq q \geq 1).
\]
One sees from Sections~\ref{ss:O*/O*O*} and \ref{ss:Sp/SpSp} that 
the compact symmetric space $X = K/K_H$ is the complex Grassmannian 
\[
X = \U(p+q) / (\U(p) \times \U(q)) = \Grass_\C(p,q),
\]
and that the normal bundle $N$ comes from the representation of $\U(p) \times \U(q)$ on 
\[
\p \cap \q = \C^p_\std \otimes_\C \C^q_\std. 
\]
Hence,
\[
N = \Taut_\C^\perp(p,q) \otimes_\C \Taut_\C(p,q).
\]
Therefore Theorem~\ref{thm:OtherSymmetric}~(2)--(3) is an immediate consequence of Corollary~\ref{cor:MainStable}, Proposition~\ref{prop:AssociatedJTilde}, and the following.

\begin{prop} \label{prop:TautCTensorTautCPerp}
Let $p \geq q \geq 1$. Then 
\[
\widetilde{J}(\Taut_\C^\perp(p,q) \otimes_\C \Taut_\C(p,q)) \neq 0
\]
except when $(p,q) = (1,1), (3,1)$. 
\end{prop}

\begin{proof}
It follows from Lemma~\ref{lem:RestGrassC}~(1) that 
\[
[N]_{|\CP^p} = \rr (-u^2 + (p-q) u + pq)
\]
in $\KO^0(\CP^p)$. 
By Lemma~\ref{lem:SecondAdams}~(1) and Fact~\ref{fact:AdamsProperties}~(2), we have 
\begin{align*}
\rr (-u^2 + (p-q) u + pq) 
&= \rr (-\psi^2(u) + (p-q+2) u + pq) \\
&= -(\psi^2 (\rr u) - \rr u) + (p-q+1) \rr u + 2pq,
\end{align*}
where $\psi^2$ is the second Adams operation.
Thus, if $\widetilde{J}( \Taut_\C(p,q) \otimes_\C \Taut_\C^\perp(p,q) ) = 0$, then we must have 
\[
-\widetilde{J}(\psi^2(\rr u) - \rr u) + (p-q+1)\cdot \widetilde{J}(\rr u) = 0. 
\]
By the Adams conjecture (Fact~\ref{fact:AdamsConjecture1}), the order of $\widetilde{J}(\psi^2(\rr u) - \rr u)$ is a power of $2$. Hence, for a sufficiently large $e \in \N$, we have 
\[
2^e (p-q+1) \cdot \widetilde{J}(\rr u) = 0,
\]
that is, $j_{\CP^p}$ must divide $2^e (p-q+1)$. 
The Adams--Walker computation (Fact~\ref{fact:OrderTaut}) shows that this happens only when $(p,q) = (1,1), (3,1)$. 
\end{proof}

\begin{rmk} \label{rmk:TrivialProjC}
If $(p,q) = (1,1)$, the normal bundle 
\[
N = \Taut^\perp_{\CP^1} \otimes_\C \Taut_{\CP^1}
= \mathscr{O}(1) \otimes_\C \mathscr{O}(-1)
\]
is trivial, even as a holomorphic vector bundle on $\CP^1$. 

Let us see that $N$ is trivial also for $(p,q) = (3,1)$. Observe that the closed subgroup $\Sp(2) \subset \U(4)$ acts transitively on $\Sph^7 = \Sp(2) / \Sp(1)$, hence also on $\CP^3 = \Sph^7 / \U(1)$. This yields a diffeomorphism 
\[
\Sp(2) / (\Sp(1) \times \U(1)) \simeq \CP^3. 
\]
The normal bundle $N$ comes from the representation of $\Sp(1) \times \U(1)$ on 
\[
(\C^3_\std \otimes_\C \C_\std)_{|\Sp(1) \times \U(1)}
= (\Ha_\std \otimes_\C \C_\std) \oplus \C_\triv.
\]
As a real representation, it coincides with the restriction of the $\Sp(2)$-representation $\Herm(\Ha^2_\std)$: 
\[
\Herm(\Ha^2_\std)_{|\Sp(1) \times \U(1)} = 
\rr (\Ha_\std \otimes_\C \C_\std) \oplus \R_\triv^{\oplus 2}.
\]
Hence, it follows from Lemma~\ref{lem:HomogeneousBundleTrivial} that $N$ is trivial. 
\end{rmk}

\subsubsection{Proof of Theorem~\ref{thm:NonSymmetric}~(4)}

As in Section~\ref{sss:SLR/SLR}, we may replace $\OO^\ast(p+q) / \OO^\ast(p)$ with 
\[
G/H = \OO^\ast(p+q) / (\OO^\ast(p) \times \U(q)) \qquad (p \geq 2,\, q \geq 1).
\]
One sees from Section~\ref{ss:O*/O*O*} that 
the compact symmetric space $X = K/K_H$ is then the complex Grassmannian 
\[
X = \U(p+q) / (\U(p) \times \U(q)) = \Grass_\C(p,q)
\]
and that the normal bundle $N$ comes from the representation of $\U(p) \times \U(q)$ on 
\[
\p \cap \q \simeq (\C^p_\std \otimes_\C \C^q_\std) \oplus \Lambda^2 (\C^q_\std).
\]
Hence, we have 
\[
N = (\Taut^\perp_\C(p,q) \otimes_\C \Taut_\C(p,q)) \oplus \Lambda^2 \Taut_\C(p,q). 
\]
Therefore Theorem~\ref{thm:NonSymmetric}~(4) is an immediate consequence of Corollary~\ref{cor:MainStable} and of the following.

\begin{prop}
Let $p \geq 2$ and $q \geq 1$. Then 
\[
\widetilde{J}(\Taut^\perp_\C(p,q) \otimes_\C \Taut_\C(p,q)) + \widetilde{J}(\Lambda^2 \Taut_\C(p,q)) \neq 0
\]
except when $(p,q) = (3,1)$. 
\end{prop}

\begin{proof}
By Lemma~\ref{lem:RestGrassC}~(1) and (4), we have 
\[
[N]_{|\CP^p} = \rr \left( -u^2 + (p-1)u + pq + \frac{q(q-1)}{2} \right) = - \rr (u^2) + (p-1) \rr u + 2pq + q(q-1)\]
in $\KO^0(\CP^p)$.
As in the proof of Proposition~\ref{prop:TautCTensorTautCPerp}, this can be rewritten as 
\[
[N]_{|\CP^p} = - (\psi^2(\rr u) - \rr u) + p \cdot \rr u + 2pq + q(q-1).
\]
By the Adams conjecture (Fact~\ref{fact:AdamsConjecture1}),
if $\widetilde{J}(N) = 0$, then 
$2^e p \cdot \widetilde{J}(\rr u) = 0$ 
for a sufficiently large $e \in \N$. 
This means that $j_{\CP^p}$ must divide $2^e p$.
By Adams--Walker's computation (Fact~\ref{fact:OrderTaut}), this happens only when $p=3$. 

Let us thus assume $p=3$. 
Arguing as in the proof of Proposition~\ref{prop:J-GLR/GLR} and applying Lemma~\ref{lem:RestGrassC}~(4), we obtain 
\[
[ N ]_{|\CP^q} = \rr \left( \frac{(q+3)(q+2)}{2} - [\Lambda^2 (\Taut_\C(3,q)) ]_{|\CP^q} \right) = - 2\rr u' + (q+3)(q+2) - 6,
\]
where $u' = [ \Taut_{\CP^q} ] - 1 \in \KU^0(\CP^q)$. In particular, we have 
$\widetilde{J}(N)_{|\CP^q} = -2 \cdot \widetilde{J}(\rr u')$. 
Thus, if $\widetilde{J}(N) = 0$, then $j_{\CP^q}$ must divide $2$. Again by Fact~\ref{fact:OrderTaut}, we have $q=1$. 
\end{proof}

\begin{rmk}
For $q = 1$, we have $\Lambda^2 \Taut_\C(p,q) = 0$.
Hence, it follows from Remark~\ref{rmk:TrivialProjC} that the normal bundle $N$ is trivial for $(p,q) = (3,1)$. 
\end{rmk}

\subsubsection{Proof of Theorem~\ref{thm:NonSymmetric}~(5)}

As in Section~\ref{sss:SLR/SLR}, we may replace $\Sp(2p+2q, \R) / \Sp(2p, \R)$ with 
\[
G/H = \Sp(2p+2q, \R) / (\Sp(2p, \R) \times \U(q)) \qquad (p,q \geq 1).
\]
The compact symmetric space $X = K/K_H$ is then the complex Grassmannian 
\[
X = \U(p+q) / (\U(p) \times \U(q)) = \Grass_\C(p,q).
\]
One sees from Section~\ref{ss:Sp/SpSp} that the normal bundle $N$ comes from the representation of $\U(p) \times \U(q)$ on 
\[
\p \cap \q \simeq (\C^p_\std \otimes_\C \C^q_\std) \oplus \Sym^2 (\C^q_\std),
\]
hence 
\[
N = (\Taut^\perp_\C(p,q) \otimes_\C \Taut_\C(p,q)) \oplus \Sym^2 \Taut_\C(p,q). 
\]
Therefore Theorem~\ref{thm:NonSymmetric}~(5) is an immediate consequence of Corollary~\ref{cor:MainStable} and of the following.

\begin{prop}
Let $p,q \geq 1$. Then 
\[
\widetilde{J}(\Taut^\perp_\C(p,q) \otimes_\C \Taut_\C(p,q)) + \widetilde{J}(\Sym^2 \Taut_\C(p,q)) \neq 0
\]
except when $(p,q) = (1,1)$. 
\end{prop}

\begin{proof}
By Lemma~\ref{lem:RestGrassC}~(1) and (3), we have 
\[
[N]_{|\CP^p} = \rr \left( (p+1) u + pq + \frac{q(q+1)}{2} \right) = (p+1) \rr u + 2pq + q(q+1),
\]
where $u = [ \Taut_{\CP^p} ] - 1 \in \KU^0(\CP^p)$.
In particular, we have 
$\widetilde{J}(N)_{|\RP^p} = (p+1) \cdot \widetilde{J}(u)$. 
As we saw in the proof of Proposition~\ref{prop:J-GLC/GLC}, this vanishes only when $p = 1$.

Let us thus assume $p=1$ (hence $X = \CP^q$). 
Arguing as in the proof of Proposition~\ref{prop:J-GLR/GLR} and applying Lemma~\ref{lem:RestGrassC}~(3), we obtain 
\[
[N] = \rr \left( \frac{(q+1)(q+2)}{2} - [\Sym^2 (\Taut_{\CP^q}) ] \right) 
= - \rr ( {u'}^2) + 2\rr u' + (q+1)(q+2) + 1,
\]
where $u' = [ \Taut_{\CP^q} ] - 1 \in \KU^0(\CP^q)$.
As in the proof of Proposition~\ref{prop:TautCTensorTautCPerp}, this can be rewitten as 
\[
[N] = -( \psi^2(\rr u') - \rr u') - \rr u' + (q+1)(q+2).
\]
By the Adams conjecture (Fact~\ref{fact:AdamsConjecture1}), 
if $\widetilde{J}(N) = 0$, then 
$-2^e \cdot \widetilde{J}(u') = 0$ for a sufficiently large $e \in \N$, \ie $j_{\CP^q}$ must be a power of $2$. We conclude from Fact~\ref{fact:OrderTaut} that $q = 1$. 
\end{proof}

\begin{rmk}
The normal bundle $N$ is trivial for $(p,q) = (1,1)$. Indeed, recall from Remark~\ref{rmk:TrivialProjC} that 
\[
\Taut_{\CP^1}^\perp \otimes_\C \Taut_{\CP^1} \simeq \underline{\C}.
\]
On the other hand, 
\[
\Sym^2 \Taut_{\CP^1} = \mathscr{O}(-2)
\]
is the tangent bundle to $X = \CP^1 = \Sph^2$. Since $T\Sph^2 \oplus \underline{\R}$ is a trivial bundle, so is $N$. 
\end{rmk}

\subsection{Computations in quaternionic Grassmannians} \label{ss:GrassH}

The purpose of this Section~\ref{ss:GrassH} is to prove Theorem~\ref{thm:NonSymmetric}~(6)--(7). 
We set $v = [ \Taut_{\HP^p} ] - 1 \in \widetilde{\KSp}^0(\HP^p)$. 

We will use the following lemma: 

\begin{lem} \label{lem:RestGrassH}
The following identities hold in $\KO^0(\HP^p)$:
\begin{enumerate}[label = {\upshape (\arabic*)}]
\item $[ \Taut^\perp_\Ha(p,q) \otimes_\Ha \overline{\Taut_\Ha(p,q)} ]_{|\HP^p} = - v^2 + (p-q) \cdot \rr \cc' v + 4pq$. 
\item $[ \Herm (\Taut_\Ha(p,q)) ]_{|\HP^p} = (q-1) \cdot \rr \cc' v + q(2q-1)$. 
\item $[ \SkewHerm (\Taut_\Ha(p,q)) ]_{|\HP^p} = v^2 + (q+1) \cdot \rr \cc' v + q(2q+1)$.
\end{enumerate} 
\end{lem}

\begin{proof}
We only give a proof of (3), for (1)--(2) can be proved in a similar way as Lemmas~\ref{lem:RestGrassR} and \ref{lem:RestGrassC}.

(3) Since $\Taut_{\HP^p}$ is a quaternionic line bundle, we have 
\[
\SkewHerm(\Taut_{\HP^p}) \oplus \underline{\R} \simeq \Taut_{\HP^p} \otimes_\Ha \overline{\Taut_{\HP^p}}.
\]
We thus see that 
\begin{align*}
&\SkewHerm(\Taut_\Ha(p,q))_{|\HP^p} \oplus \underline{\R}
\simeq \SkewHerm(\Taut_{\HP^p} \oplus \underline{\Ha}^{\oplus (q-1)}) \oplus \underline{\R} \\
&\qquad\simeq \SkewHerm(\Taut_{\HP^p}) \oplus \rr \cc' (\Taut_{\HP^p} \otimes_\Ha \overline{\underline{\Ha}^{\oplus (q-1)}}) \oplus \SkewHerm(\underline{\Ha}^{\oplus (q-1)}) \oplus \underline{\R} \\
&\qquad\simeq (\Taut_{\HP^p} \otimes_\Ha \overline{\Taut_{\HP^p}}) \oplus (\rr\cc' \Taut_{\HP^p})^{\oplus (q-1)} \oplus \underline{\R}^{\oplus(2q-1)(q-1)}.
\end{align*}
Passing to $\KO^0$ yields the desired identity. 
\end{proof}

\subsubsection{Proof of Theorem~\ref{thm:NonSymmetric}~(6)} \label{sss:SpC/SpC}

As in Section~\ref{sss:SLR/SLR}, we may replace $\Sp(2p+2q, \C) / \Sp(2p, \C)$ with
\[
G/H = \Sp(2p+2q, \C) / (\Sp(2p, \C) \times \Sp(q)) \qquad (p,q \geq 1).
\]
The compact symmetric space $X = K/K_H$ is then the quaternionic Grassmannian 
\[
X = \Sp(p+q) / (\Sp(p) \times \Sp(q)) = \Grass_\Ha(p,q).
\]
The normal bundle $N$ comes from the representation of $\Sp(p) \times \Sp(q)$ on 
\[
\p \cap \q \simeq (\Ha^p_\std \otimes_\Ha \overline{\Ha^q_\std}) \oplus \SkewHerm (\Ha^q_\std),
\]
hence 
\[
N = (\Taut^\perp_\Ha(p,q) \otimes_\Ha \overline{\Taut_\Ha(p,q)}) \oplus \SkewHerm (\Taut_\Ha(p,q)). 
\]
Therefore Theorem~\ref{thm:NonSymmetric}~(6) is an immediate consequence of Corollary~\ref{cor:MainStable} and of the following.

\begin{prop}
Let $p,q \geq 1$. Then 
\[
\widetilde{J}(\Taut^\perp_\Ha(p,q) \otimes_\Ha \overline{\Taut_\Ha(p,q)}) + \widetilde{J}(\SkewHerm (\Taut_\Ha(p,q)) \neq 0.
\]
\end{prop}

\begin{proof}
By Lemma~\ref{lem:RestGrassH}~(1) and (3), we have 
\[
[N]_{|\HP^p} = (p+1) \cdot \rr \cc' v + 4pq + q(2q+1),
\]
in $\KO^0(\HP^p)$.
Thus, if $\widetilde{J}(N) = 0$, then $(p+1) \cdot \widetilde{J}(\rr \cc' v) = 0$. In other words, $j_{\HP^p}$ must divide $p+1$. This is impossible by Sigrist--Suter's computation (Fact~\ref{fact:OrderTaut}).
\end{proof}

\subsubsection{Proof of Theorem~\ref{thm:NonSymmetric}~(7)}

As in Section~\ref{sss:SLR/SLR}, we may replace $\SL(p+q, \Ha) / \SL(p, \Ha)$ with
\[
G/H = \GL(p+q, \Ha) / (\GL(p, \Ha) \times \Sp(q)) \qquad (p \geq 2,\, q \geq 1).
\]
The compact symmetric space $X = K/K_H$ is then the quaternionic Grassmannian 
\[
X = \Sp(p+q) / (\Sp(p) \times \Sp(q)) = \Grass_\Ha(p,q).
\]
The normal bundle $N$ comes from the representation of $\Sp(p) \times \Sp(q)$ on 
\[
\p \cap \q \simeq (\Ha^p_\std \otimes_\Ha \overline{\Ha^q_\std}) \oplus \Herm (\Ha^q_\std),
\]
hence 
\[
N = (\Taut^\perp_\Ha(p,q) \otimes_\Ha \overline{\Taut_\Ha(p,q)}) \oplus \Herm (\Taut_\Ha(p,q)). 
\]
Therefore Theorem~\ref{thm:NonSymmetric}~(7) is an immediate consequence of Corollary~\ref{cor:MainStable} and of the following.

\begin{prop}
Let $p \geq 2$ and $q \geq 1$. Then 
\[
\widetilde{J}(\Taut^\perp_\Ha(p,q) \otimes_\Ha \overline{\Taut_\Ha(p,q)}) + \widetilde{J}(\Herm (\Taut_\Ha(p,q)) \neq 0.
\]
\end{prop}

\begin{proof}
By Lemma~\ref{lem:RestGrassH}~(1) and (2), we have 
\[
[N]_{|\HP^p} = -v^2 + (p-1) \cdot \rr \cc' v + 4pq + q(2q-1),
\]
in $\KO^0(\HP^p)$.
By Lemma~\ref{lem:SecondAdams}~(2), we have 
\[
2[N]_{|\HP^p} = -(\psi^2(\rr \cc' v) - \rr \cc' v) + (2p+1) \cdot \rr \cc' v + 8pq + 2q(2q-1).
\]
By the Adams conjecture (Fact~\ref{fact:AdamsConjecture1}), if $\widetilde{J}(N) = 0$, then $2^e (2p+1) \cdot \widetilde{J}(\rr \cc' v) = 0$ for a sufficiently large $e \in \N$. It means that $j_{\HP^p}$ must divide $2^e (2p+1)$. 
Sigrist--Suter's computation (Fact~\ref{fact:OrderTaut}) tells us that this cannot happen when $p \geq 2$. 
\end{proof}

\subsection{Computations in other compact symmetric spaces} \label{ss:OtherSymmetric}

The purpose of this Section~\ref{ss:OtherSymmetric} is to prove Theorem~\ref{thm:OtherSymmetric}~(4)--(6). 

\subsubsection{Proof of Theorem~\ref{thm:OtherSymmetric}~(4)}

We first remark that one may assume $p \geq 3$ in Theorem~\ref{thm:OtherSymmetric}~(4). 
Indeed, if $p=2$, then 
then $G$, $H$, and $H^a$ are locally isomorphic to 
$\OO(6, \C)$, $\OO(5, \C) \times \OO(1, \C)$, and $\OO(5,1)$, respectively, 
hence Theorem~\ref{thm:OtherSymmetric}~(1) (which is already proved in Section~\ref{sss:ComplexSpheres}) applies. 

We now fix $p\geq 3$ and let
\[
G/H = \SL(2p, \C) / \Sp(2p, \C). 
\]
The compact symmetric space $X = K/K_H$ is then 
\[
X_p = \SU(2p) / \Sp(p),
\]
and the normal bundle $N$ is the tangent bundle of $X$ by Lemma~\ref{lem:Normal=Tangent}.
Therefore Theorem~\ref{thm:OtherSymmetric}~(4) is an immediate consequence of Corollary~\ref{cor:MainStable}, Proposition~\ref{prop:AssociatedJTilde}, and the following.

\begin{prop}\label{prop:SU/SpTangent}
$\widetilde{J}(TX_p) \neq 0$ for $p \geq 3$. 
\end{prop}

\begin{proof}
We regard $X_3 = \SU(6) / \Sp(3)$ as a submanifold of $X_p = \SU(2p) / \Sp(p)$ in an obvious way.
The vector bundle $TX_p$ is induced from the representation of 
$\Sp(p)$ on 
\[
\Herm_0(p, \Ha) = \{ A \in \M(p, \Ha) \mid {^t \overline{A}} = A,\ \Tr A = 0 \}, 
\]
which we denote by $\Herm_0(\Ha^p_\std)$. 
Its restriction to $\Sp(3) \subset \Sp(p)$ is decomposed as 
\[
\Herm_0(\Ha^p_\std)_{|\Sp(3)} = 
\Herm_0(\Ha^3_\std) \oplus (\Ha^3_\std)^{\oplus (p-3)} \oplus (\R_\triv)^{\oplus (p-3)(2p-7)}, 
\]
where $\Ha^3_\std$ and $\R_\triv$ are the standard and the trivial representations of $\Sp(3)$, respectively. 
Since the representation $\Ha^3_\std$ is a restriction of the standard representation $\C^6_\std$ of $\SU(6)$ to $\Sp(3)$, it induces a trivial vector bundle on $X_3$ by Lemma~\ref{lem:HomogeneousBundleTrivial}. Hence, 
the restriction $(TX_p)_{|X_3}$ is the direct sum of $TX_3$ and a trivial vector bundle. 
Thus, it suffices to prove the proposition for $p=3$. 

The inclusion $\SU(5) \subset \SU(6)$ induces a diffeomorphism 
\[
\SU(5) / \Sp(2) \simeq \SU(6) / \Sp(3) = X_3. 
\]
Hence, we can form the following fiber bundle over $\Sph^9$: 
\[
\left( \Sph^5 \xrightarrow{\iota} X_3 \xrightarrow{\pi} \Sph^9 \right) = \left( \SU(4) / \Sp(2) \xrightarrow{\iota} \SU(5) / \Sp(2) \xrightarrow{\pi} \SU(5) / \SU(4) \right). 
\]
Here, we have used the diffeomorphism 
\[
\Sph^5 \simeq \Spin(6) / \Spin(5) \simeq \SU(4) / \Sp(2)
\]
induced from the accidental isomorphism $\SU(4) \simeq \Spin(6)$. 
It is a sphere bundle for the spin vector bundle 
\[
\SU(5) \times_{\SU(4)} \R^6 \to \SU(5) / \SU(4) = \Sph^9, 
\]
where $\SU(4)$ acts on $\R^6$ via the double covering 
\[
\SU(4) \simeq \Spin(6) \to \SO(6). 
\]
Hence, we obtain the Thom--Gysin long exact sequence (Section~\ref{ss:Thom}): 
\[
\cdots \to \KO^{-6}(\Sph^9) \to \KO^0(\Sph^9) \xrightarrow{\pi^\ast} \KO^0(X_3) \to \KO^{-5}(\Sph^9) \to \cdots.
\]
It follows from Fact~\ref{fact:Unbased} and Table~\ref{table:ABS} that: 

\begin{align*}
\KO^{-6}(\Sph^9) &\simeq \KO^{-6}(\pt) \oplus \widetilde{\KO}^{-6}(\Sph^9) = \widetilde{\KO}^0(\Sph^6) \oplus \widetilde{\KO}^0(\Sph^{15}) \simeq 0, \\
\KO^{-5}(\Sph^9) &\simeq \KO^{-5}(\pt) \oplus \widetilde{\KO}^{-5}(\Sph^9) = \widetilde{\KO}^0(\Sph^5) \oplus \widetilde{\KO}^0(\Sph^{14}) \simeq 0.
\end{align*}
Thus, the map
$\pi^\ast \colon \KO^0(\Sph^9) \to \KO^0(X_3)$
is an isomorphism. By Lemma~\ref{lem:ConclusionOfAdams1}, 
$\pi^\ast \colon \widetilde{J}(\Sph^9) \to \widetilde{J}(X_3)$
is also an isomorphism.

The inclusion $\SU(5) \subset \Spin(10)$ induces a diffeomorphism 
\[
\Sph^9 \simeq \SU(5) / \SU(4) \simeq \Spin(10) / \Spin(9).
\]
Let $\Sigma$ be the spinor bundle on $\Sph^9$, namely, 
\[
\Sigma = \Spin(10) \times_{\Spin(9)} \R^{16}_\spin,
\]
where $\R^{16}_\spin = \Delta_9^0$ is the spinor representation of $\Spin(9)$ (Definition~\ref{defn:SpinorRepresentation2}). 
Observe that
\begin{align*}
(\R^{16}_\spin)_{|\Sp(2)} &= \Herm_0(\Ha^2_\std) \oplus \Ha^2_\std \oplus (\R_\triv)^{\oplus 3}, \\
\Herm_0(\Ha^3_\std)_{|\Sp(2)} &= \Herm_0(\Ha^2_\std) \oplus \Ha^2_\std \oplus \R_\triv.
\end{align*}
This implies that 
$\pi^\ast \Sigma = TX_3 \oplus \underline{\R}^{\oplus 2}$,
hence
$\widetilde{J}(TX_3) = \pi^\ast \widetilde{J}(\Sigma)$. 
Now, by Table~\ref{table:ABS} and Fact~\ref{fact:JTildeSphere}, $\widetilde{J}(\Sigma)$ is a nonzero element in $\widetilde{J}(\Sph^9) \simeq \Z/2$. 
Recall that $\pi^\ast \colon \widetilde{J}(\Sph^9) \to \widetilde{J}(X_3)$ is an isomorphism. Thus, we have $\pi^\ast \widetilde{J}(\Sigma) \neq 0$. 
\end{proof}

\subsubsection{Proof of Theorem~\ref{thm:OtherSymmetric}~(5)}

For $n \geq 2$, we set
\[
X_n = \SO(2n) / \U(n) \qquad \text{and} \qquad E_n = \SO(2n) \times_{\U(n)} \C^n. 
\]
If $n \geq m \geq 2$, we regard $X_m$ as a submanifold of $X_n$ in an obvious way. 

Suppose that 
\[
G/H = \SO_0(2p,2q) / \U(p,q) \qquad (p \geq q \geq 2). 
\]
One sees from Section~\ref{ss:O/U} that 
the compact symmetric space $X = K/K_H$ is the product 
\[
X = (\SO(2p) \times \SO(2q)) / (\U(p) \times \U(q)) = 
X_p \times X_q,
\]
and that the normal bundle $N$ comes from the representation of $\U(p) \times \U(q)$ on 
\[
\p \cap \q = \C^p_\std \otimes_\C \overline{\C^q_\std}.
\]
Hence, $N = \rr (E_p \boxtimes_\C \overline{E_q})$ where $\boxtimes$ is the outer tensor product.  
Theorem~\ref{thm:OtherSymmetric}~(5) is an immediate consequence of Corollary~\ref{cor:MainStable} and of the following.

\begin{prop}
$\widetilde{J}(\rr (E_p \boxtimes_\C \overline{E_q})) \neq 0$ for $p \geq q \geq 2$. 
\end{prop}

\begin{proof}
The restriction of the representation $\C^p_\std \otimes_\C \overline{\C^q_\std}$ to $\U(2) \times \U(2)$ is a sum of $\C^2_\std \otimes_\C \overline{\C^2_\std}$ and a representation that extends to $\SO(4) \times \SO(4)$. 
Hence, as in the proof of Proposition~\ref{prop:SU/SpTangent}, 
we can reduce the problem to the case of $(p,q) = (2,2)$. 

The accidental isomorphism $\SO(4) \simeq (\SU(2) \times \SU(2)) / \{ \pm (I_2, I_2) \}$ enables us to consider the embedding 
\[
\SU(2) \hookrightarrow (\SU(2) \times \SU(2)) / \{ \pm (I_2, I_2) \} \simeq \SO(4), \qquad g \mapsto \br{g, I_2}. 
\]
It induces a diffeomorphism 
\[
\Sph^2 = \SU(2) / \operatorname{S}(\U(1) \times \U(1)) \simeq \SO(4) / \U(2) = X_2,
\]
under which the complex vector bundle 
$E_2$ corresponds to $(\Taut_{\CP^1})^{\oplus 2}$.

Let $\beta = [\Taut_{\CP^1}] - 1 \in \widetilde{\KU^0}(\Sph^2)$ be the complex Bott element (Definition~\ref{defn:BottElement}).
Let $\pr_1, \pr_2 \colon \Sph^2 \times\nolinebreak\Sph^2 \to \Sph^2$ be the projection onto the first and the second factors, respectively, 
and let $\beta_1 = \pr_1^\ast(\beta)$ and $\beta_2 = \pr_2^\ast(\beta)$.
We have 
\[
[ E_2 ] = 2 [ \Taut_{\CP^1} ] = 2 \beta + 2, \qquad 
[ \overline{E_2} ] = 2 (2 - [ \Taut_{\CP^1} ]) = -2\beta + 2
\]
in $\KU^0(\Sph^2)$. Hence, 
\[
\rr [ E_2 \boxtimes_\C \overline{E_2} ] - 8 = \rr ((2 \beta_1 + 2)(-2 \beta_2 + 2)) - 8 = -4 \cdot \rr (\beta_1\beta_2).
\]
Thus, our goal is to show that $-4 \cdot \widetilde{J}(\rr (\beta_1\beta_2)) \neq 0$ in $\widetilde{J}(\Sph^2 \times \Sph^2)$. 

Let $q \colon \Sph^2 \times \Sph^2 \to \Sph^2 \wedge \Sph^2 \ (= \Sph^4)$ be the quotient map. 
By Fact~\ref{fact:KTheoryProduct} and Table~\ref{table:ABS}, the maps $\pr_1, \pr_2$, and $q$ induce an isomorphism 
\[
\widetilde{\KO}^0(\Sph^2 \times \Sph^2) = \widetilde{\KO}^0(\Sph^2) \oplus \widetilde{\KO}^0(\Sph^2) \oplus \widetilde{\KO}^0(\Sph^4)  
\simeq \Z/2 \oplus \Z/2 \oplus \Z.
\]

By construction of the Atiyah--Bott--Shapiro map $\chi \colon \RU_\gr(\Cl_k) \to \widetilde{\KU}^0(\Sph^k)$ (Section~\ref{ss:ABS}), 
the following diagram commutes: 
\[
\begin{tikzcd}
\RU_\gr(\Cl_2) \times \RU_\gr(\Cl_2) \ar[d, "\wedge'"'] \ar[r, "\chi \times \chi"] & \widetilde{\KU}^0(\Sph^2) \times \widetilde{\KU}^0(\Sph^2) \ar[d, "\wedge"] \\
\RU_\gr(\Cl_4) \ar[r, "\chi"] & \widetilde{\KU}^0(\Sph^4),
\end{tikzcd}
\]
where the left map $\wedge'$ is defined as $[M] \wedge' [N] = [M \boxtimes_\C N]$.

Let $\Delta_{2,+}$ be the positive $\Z/2$-graded spinor representation of $\Cl_2$ (Definition~\ref{defn:SpinorRepresentation1}). 
The even part of $\Delta_{2,+} \boxtimes_\C \Delta_{2,+}$ is a nontrivial complex $4$-dimensional representation of $\Cl_4^0 \simeq \M(2, \Ha)^2$ on which $\gamma = c(e_1)c(e_2)c(e_3)c(e_4)$ acts by $i^2 = -1$. Hence, by Remark~\ref{rmk:SpinorRepresentationSign}, it must be isomorphic to the even part of $\cc' \Delta_{4,+}$, where
$\Delta_{4,+}$ is the positive $\Z/2$-graded spinor representation of $\Cl_4$. By Remark~\ref{rmk:CliffordZ/2Grading}~(1) and the above commutative diagram, we have 
\[
q^\ast \chi(\cc' \Delta_{4,+}) = q^\ast \chi(\Delta_{2,+} \boxtimes_\C \Delta_{2,+}) = q^\ast(\chi(\Delta_{2,+}) \wedge \chi(\Delta_{2,+})) = \beta_1\beta_2.
\]
Since $\chi(\rr \cc' \Delta_{4,+})$ is a generator of $\widetilde{\KO}^0(\Sph^4)$, under the above isomorphism, $\rr (\beta_1\beta_2) \in \widetilde{\KO}^0(\Sph^2 \times \Sph^2)$ corresponds to 
$(0,0,1) \in \Z/2 \oplus \Z/2 \oplus \Z$. 
Finally, recall from Fact~\ref{fact:JTildeSphere} that $\widetilde{J}(\Sph^4)$, which is a quotient of $\widetilde{\KO}_0(\Sph^4)$, is isomorphic to $\Z/24$. Hence $\widetilde{J}(\rr (\beta_1\beta_2))$ must be a generator of $\widetilde J(\Sph^4)$ and we conclude from Lemma~\ref{lem:ConclusionOfAdams2} that $-4 \cdot \widetilde{J}(\rr (\beta_1\beta_2)) \neq 0$. 
\end{proof}

\subsubsection{Proof of Theorem~\ref{thm:OtherSymmetric}~(6)}

For $n \geq 2$, we set 
\[
X_n = \SU(2n) / \Sp(n) \qquad \text{and} \qquad E_n = \SU(2n) \times_{\Sp(n)} \Ha^n_\std.
\]
If $n \geq m \geq 2$, we regard $X_m$ as a submanifold of $X_n$ in an obvious way. 

Suppose that 
\[
G/H = \SU(2p,2q) / \Sp(p,q) \qquad (p \geq q \geq 2). 
\]
One sees from Section~\ref{ss:U/Sp} that 
the compact symmetric space $X = K/K_H$ is given by 
\[
X = \operatorname{S}(\U(2p) \times \U(2q)) / (\Sp(p) \times \Sp(q)),
\]
which contains 
\[
X_p \times X_q = (\SU(2p) \times \SU(2q)) / (\Sp(p) \times \Sp(q))
\]
as a submanifold, 
and that the normal bundle $N$ comes from the representation of $\Sp(p) \times \Sp(q)$ on 
\[
\p \cap \q = \Ha^p_\std \otimes_\Ha \overline{\Ha^q_\std}.
\]
Hence,
\[
N_{|X_p \times X_q} = E_p \boxtimes_\Ha \overline{E_q}
\]
where $\boxtimes$ is the outer tensor product. 
Theorem~\ref{thm:OtherSymmetric}~(6) is an immediate consequence of Corollary~\ref{cor:MainStable} and of the following.

\begin{lem}
$\widetilde{J}(E_p \boxtimes_\Ha \overline{E_q}) \neq 0$ for $p \geq q \geq 2$.
\end{lem}

\begin{proof}
The restriction of the representation 
$\Ha^p_\std \otimes_\Ha \overline{\Ha^q_\std}$ to $\Sp(2) \times \Sp(2)$ is a sum of $\Ha^2_\std \otimes_\Ha \overline{\Ha^2_\std}$ and a representation that extends to $\SU(4) \times \SU(4)$. 
Hence, as in the proof of Proposition~\ref{prop:SU/SpTangent}, we can reduce the problem to the case of $(p,q) = (2,2)$. 

The accidental isomorphism $\SU(4) \simeq \Spin(6)$
induces a diffeomorphism 
\[
\Sph^5 = \Spin(6) / \Spin(5) \simeq \SU(4) / \Sp(2) = X_2. 
\]
Under the isomorphism $\Sp(2) \simeq \Spin(5)$, the standard representation $\Ha^2_\std$ corresponds to the spinor representation $\Delta_5^0$ of $\Spin(5)$ (Definition~\ref{defn:SpinorRepresentation2}). 

Let $\pr_1, \pr_2 \colon \Sph^5 \times \Sph^5 \to \Sph^5$ be the projection onto the first and the second factors, respectively, and let $\chi_1 = \pr_1^\ast(\chi(\Delta_5))$ and $\chi_2 = \pr_2^\ast(\chi(\Delta_5))$, where $\chi \colon \RSp_\gr(\Cl_k) \to \widetilde{\KSp}^0(\Sph^k)$ is the Atiyah--Bott--Shapiro map (Section~\ref{ss:ABS}). An easy computation shows that 
\[
[E_2 \boxtimes_\Ha \overline{E_2}] - 16 = \chi_1 \chi_2.
\]
Thus, our goal is to show that $\widetilde{J}(\chi_1\chi_2) \neq 0$ in $\widetilde{J}(\Sph^5 \times \Sph^5)$. 

Let $q \colon \Sph^5 \times \Sph^5 \to \Sph^5 \wedge \Sph^5 (= \Sph^{10})$ be the quotient map.
By Fact~\ref{fact:KTheoryProduct} and Table~\ref{table:ABS}, the maps $\pr_1, \pr_2$, and $q$ induce an isomorphism  
\[
\widetilde{\KO}^0(\Sph^5 \times \Sph^5) \simeq \widetilde{\KO}^0(\Sph^5) \oplus \widetilde{\KO}^0(\Sph^5) \oplus \widetilde{\KO}^0(\Sph^{10}) \simeq 0 \oplus 0 \oplus \Z/2.
\]

By construction of the map $\chi$, the following diagram commutes: 
\[
\begin{tikzcd}
\RSp_\gr(\Cl_5) \times \RSp_\gr(\Cl_5) \ar[d, "\wedge'"'] \ar[r, "\chi \times \chi"] & \widetilde{\KO}^0(\Sph^5) \times \widetilde{\KO}^0(\Sph^5) \ar[d, "\wedge"] \\
\RO_\gr(\Cl_{10}) \ar[r, "\chi"] & \widetilde{\KO}^0(\Sph^{10}),
\end{tikzcd}
\]
where the left map $\wedge'$ is given by $[M] \wedge' [N] = [M \boxtimes_\Ha \overline{N}]$. 

The even part of $\Delta_5 \boxtimes_\Ha \overline{\Delta_5}$ is a nontrivial real $32$-dimensional representation of $\Cl^0_{10} \simeq \M(16, \C)$, hence it must be isomorphic to the even part of $\rr \Delta_{10,+}$.
By Remark~\ref{rmk:CliffordZ/2Grading}~(1) and the above commutative diagram, we have 
\[
q^\ast\chi(\rr \Delta_{10,+}) = 
q^\ast \chi(\Delta_5 \boxtimes_\Ha \overline{\Delta_5}) = 
q^\ast (\chi(\Delta_5) \wedge \chi(\Delta_5)) = \chi_1 \chi_2, 
\]
and therefore, 
\[
q^\ast \widetilde{J}(\chi(\rr \Delta_{10,+})) = \widetilde{J}(\chi_1 \chi_2).
\]
By Table~\ref{table:ABS} and Fact~\ref{fact:JTildeSphere}, $\widetilde{J}(\chi(\rr \Delta_{10,+}))$ is a nonzero element of $\widetilde{J}(\Sph^{10}) \simeq \Z/2$, whereas $q^\ast \colon \widetilde{J}(\Sph^{10}) \to \widetilde{J}(\Sph^5 \times \Sph^5)$ is an isomorphism by Lemma~\ref{lem:ConclusionOfAdams2}. We have thus proved that $\widetilde{J}(\chi_1 \chi_2) \neq 0$. 
\end{proof}

\section{Local geometric fibrations} \label{s:GeometricFibrations}

Recall from Section~\ref{ss:MotivationGeometricFibration} that the Geometric fibration conjecture, formulated in \cite{Tho15+}, predicts that compact quotients of a reductive homogeneous space $G/H$ virtually fiber over a base manifold~$M$, the fibers of which are translates of the maximal compact subspace $X = K/K_H$.
In this last section of the paper, we discuss links between the Geometric fibration conjecture, our main Theorem~\ref{thm:MainTheorem}, and Kobayashi--Yoshino's study of tangential homogeneous spaces \cite{KY05,KY}.

\subsection{Local geometric fibrations and compact quotients of tangential homogeneous spaces}

Let $G/H$ be a reductive homogeneous space $G/H$, 
and let $X=K/K_H$ and the decompositions 
$\g = \kk \oplus \p$ and $\h = (\kk \cap \h) \oplus (\p \cap \h)$ 
be as in Section~\ref{ss:Reductive}.
We shall use the following terminology.

\begin{defn} \label{defn:LocalGeomFib}
A \emph{local geometric fibration} of $G/H$ is a smooth fiber bundle 
$\pi \colon \Omega \to U$ 
whose total space $\Omega$ is an open domain of $G/H$ 
and whose fibers are $G$-translates of $X$ in $G/H$.
We call it a \emph{geometric fibration} of $G/H$ if $\Omega = G/H$.
\end{defn}

In \cite{KY05}, Kobayashi and Yoshino introduced the notion of the \emph{tangential homogeneous space} associated to the reductive homogeneous space $G/H$, 
and investigated the existence of its compact quotients.
Let us write $\p_H = \p \cap \h$, and denote by $\theta$ the Cartan involution of~$G$ fixing~$K$. 

\begin{defn}
The \emph{tangential homogeneous space} associated to the reductive homogeneous space $G/H$ is the (non-reductive) homogeneous space $G_\theta/H_\theta = (K \ltimes \p) / (K_H \ltimes \p_H)$.
\end{defn}

Consider the two projections
\[
G/K \xleftarrow{\pi_K} G/K_H \xrightarrow{\pi_H} G/H.
\]
The following is a more complete version of Proposition~\ref{prop:LocalFibrTangential}.

\begin{prop} \label{prop:EquivalentConditionsLocalFibrations}
The following four conditions are equivalent:
\begin{enumerate}[label = {\upshape (\roman*)}]
  \item the tangential homogeneous space $G_\theta/H_\theta$ admits compact quotients,
  \item there exists a linear subspace $V$ of $\p$ such that $(\p\cap \h) \oplus \Ad_k(V) = \p$
for all $k \in K$,
  \item there exists a smooth submanifold $\widetilde{M}$ of $G/K$ passing through $o' = 1 \cdot K$ such that the map
  \begin{equation} \label{eqn:pi2}
  {\pi_H}_{|\pi_K^{-1}(\widetilde{M})} \colon \pi_K^{-1}(\widetilde{M}) \to G/H
  \end{equation}
  is a local diffeomorphism at each point of $\pi_K^{-1}(o') = K/K_H$,
  \item the reductive space $G/H$ admits local geometric fibrations.
\end{enumerate}
\end{prop}

The equivalence (i)~$\Leftrightarrow$~(ii) was proved in \cite[\S\,5.3]{KY05}.

\begin{proof}[Proof of (ii)~$\Leftrightarrow$~(iii) in Proposition~\ref{prop:EquivalentConditionsLocalFibrations}]
We have natural identifications
\[
T_{o'}(G/K) \simeq \g/\kk \simeq \p, \qquad 
T_o(G/H) \simeq \g/\h \simeq \q, \qquad 
T_{\widetilde{o}}(G/K_H) \simeq \g/(\kk \cap \h) \simeq (\kk \cap \q) \oplus \p,
\]
where $o' = 1 \cdot K$, $o = 1 \cdot H$, and $\widetilde{o} = 1 \cdot K_H$ are the base points of $G/K$, $G/H$, and $G/K_H$, respectively.
With these identifications, any linear subspace $V$ of~$\p$ is the tangent space $T_{o'}\widetilde{M}$ for some smooth submanifold $\widetilde{M}$ of $G/K$ passing through~$o'$.
We have
\[
T_{\widetilde{o}} (\pi_K^{-1}(\widetilde{M})) = (\kk \cap \q) \oplus V,
\]
and the map \eqref{eqn:pi2}
is a local diffeomorphism at~$\widetilde{o}$ if and only if 
\[
(\p \cap \h) \oplus V = \p. 
\]
By equivariance of $\pi_K$ and $\pi_H$, 
for each $k \in K$, the map \eqref{eqn:pi2} is a local diffeomorphism at $k^{-1} \cdot \widetilde{o}$ if and only if 
\[
{\pi_H}_{|\pi_K^{-1}(k \cdot \widetilde{M})} \colon \pi_K^{-1}(k \cdot \widetilde{M}) \to G/H
\]
is a local diffeomorphism at $\widetilde{o}$, hence if and only if
\[
(\p \cap \h) \oplus \Ad_k(V) = \p. \hfill\qedhere
\]
\end{proof}

\begin{proof}[Proof of (iii)~$\Leftrightarrow$~(iv) in Proposition~\ref{prop:EquivalentConditionsLocalFibrations}]
Suppose there exists a smooth submanifold $\widetilde{M}$ of $G/K$ passing through $o' = 1 \cdot K$ such that \eqref{eqn:pi2} is a local diffeomorphism at each point of $\pi_K^{-1}(o') = K/K_H$.
Since $\pi_K^{-1}(o')$ is compact and ${\pi_H}_{|\pi_K^{-1}(o')}$ is injective, after replacing $\widetilde{M}$ by an open submanifold, we may assume that \eqref{eqn:pi2} is globally an open embedding.
Note that \eqref{eqn:pi2} sends $\pi_K^{-1}(gK)$ to the compact subspace $g\cdot X$ of $G/H$ for each $g\in G$.
Denoting the image of \eqref{eqn:pi2} by $\Omega$, we obtain that 
\[
\pi_K \circ \pi_H^{-1} \colon \Omega \to \widetilde{M},
\]
is a local geometric fibration by construction. 
Conversely, any local geometric fibration containing $X = K/K_H$ as a fiber arises in this way.
\end{proof}

\begin{rmk} \label{rmk:CalabiMarkusAndH^a}
As a corollary to Proposition~\ref{prop:EquivalentConditionsLocalFibrations}, we obtain the following:
\begin{enumerate}
  \item If $\rk_\R(G)= \rk_\R(H)$, then $G/H$ does not admit local geometric fibrations, by \cite[\S\,5.2, (8)]{KY05} and \cite[Fact\,3.1]{Toj21}.
  \item\label{item:local-fibr-stable-under-associate} If $G/H$ and $G/H^a$ are associated symmetric spaces, then $G/H$ admits local geometric fibrations if and only if $G/H^a$ does.
This is contained in the unpublished draft \cite{KY}, and also in \cite[Prop.\,2.10]{Toj21}.
\end{enumerate}
\end{rmk}

In view of the equivalence (iii)~$\Leftrightarrow$~(iv) of Proposition~\ref{prop:EquivalentConditionsLocalFibrations}, here is a more precise version of the Geometric fibration conjecture.

\begin{conj} \label{conj:precise-fibr-conj}
Let $\Gamma$ be a torsion-free subgroup of~$G$ acting properly discontinuously and cocompactly on $G/H$.
Then there exists a $\Gamma$-invariant contractible smooth submanifold $\widetilde{M}$ of $G/K$, of dimension $\dim_+(G/H)$ (see \eqref{eqn:dim_+}), such that
\[
G/H = \bigsqcup_{\widetilde{m}\in\widetilde{M}} \pi_H \circ \pi_K^{-1}(\widetilde{m}).
\]
\end{conj}

\subsection{Compact quotients should imply local geometric fibrations}

The Geometric fibration conjecture would imply in particular (i)~$\Rightarrow$~(ii) in Conjecture \ref{conj:CompactQuotients<=>Fibrations}, namely:
 
\begin{conj} \label{conj:NoLocalFibrationObstruction}
If a reductive homogeneous space $G/H$ admits compact quotients, then it admits local geometric fibrations. 
\end{conj}

In other words, using Proposition~\ref{prop:EquivalentConditionsLocalFibrations}, if $G/H$ admits compact quotients, then so does the corresponding tangential homogeneous space $G_{\theta}/H_{\theta}$.
We note that understanding the relation between the existence of compact quotients for $G/H$ and $G_\theta/H_\theta$ has been a long-standing open question: see \eg \cite[Rem.\,1.9]{Toj21}.

\begin{rmk} \label{rmk:standard-fibr}
If $G/H$ admits \emph{standard} compact quotients, then it always admits a geometric fibration.
Indeed, let $L$ be a reductive Lie subgroup of~$G$ acting transitively on $G/H$ with compact stabilizer.
Then the natural projection from $G/H \simeq L/(L\cap H)$ to the Riemannian symmetric space of~$L$ is an $L$-equivariant geometric fibration of $G/H$, and any uniform lattice $\Gamma$ of~$L$ satisfies the conclusion of the Geometric fibration conjecture and of the more precise Conjecture~\ref{conj:precise-fibr-conj} (with $\widetilde{M}$ equal to the Riemannian symmetric space of~$L$, seen as a smooth submanifold of $G/K$).
\end{rmk}

The following easy implications hold:

\begin{lem} \label{lem:ImplicationsTriviality}
Consider the following three conditions: 
\begin{enumerate}[label = {\upshape (\roman*)}]
  \item\label{item:fibr-1} $G/H$ admits local geometric fibrations,
  \item\label{item:fibr-2} the normal bundle $N$ to $X = K/K_H$ in $G/H$ is trivial (as a vector bundle),
  \item\label{item:fibr-3} the normal sphere bundle $S(N)$ is fiber-homotopically trivial.
\end{enumerate}
Then we have \upshape{(i)} $\Rightarrow$ \upshape{(ii)} $\Rightarrow$ \upshape{(iii)}. 
\end{lem}

\begin{proof}
The implication (ii)~$\Rightarrow$~(iii) is obvious.
For (i)~$\Rightarrow$~(ii), let $\pi \colon \Omega \to U$ be a local geometric fibration of $G/H$. Without loss of generality, we can assume that $K/K_H = \pi^{-1}(u)$ for some $u \in U$. Then the normal bundle $N$ is canonically isomorphic to the trivial vector bundle $X \times T_uU$. 
\end{proof}

As mentioned in Section~\ref{ss:MotivationGeometricFibration}, Conjecture \ref{conj:NoLocalFibrationObstruction} predicts that condition (i) is an obstruction to the existence of compact quotients of $G/H$ and our main Theorem~\ref{thm:MainTheorem}, which states that condition (iii) is an obstruction, can thus be seen as a partial confirmation of this conjecture.

Note however that the reverse implications do not hold. The following Proposition \ref{prop: Normal Bundle homotopically but not linearly trivial} (\resp \ref{prop:trivial-bundle-no-fibrations}) gives counterexamples to \ref{item:fibr-3}~$\Rightarrow$~\ref{item:fibr-2} (\resp \ref{item:fibr-2}~$\Rightarrow$~\ref{item:fibr-1}).

\begin{prop} \label{prop: Normal Bundle homotopically but not linearly trivial}
Let 
\[
G/H = (\U(p,q+1) \times \OO(k,1)) / (\U(p,q) \times \U(1) \times \OO(k) \times \OO(1))
\]
be the product of the complex pseudo-Riemannian hyperbolic space $\Hyp^{p,q}_\C$ with the real hyperbolic space $\Hyp^k_\R$. If $q\geq 2$ and $j_{\CP^q}$ divides $p$, then for $k$ large enough the normal bundle $N$ to $K/K_H = \CP^q$ is nontrivial but its sphere bundle is fiber-homotopically trivial.
\end{prop}

\begin{proof}
Recall that the normal bundle to $\CP^q$ in $\Hyp^{p,q}_\C$ is $\Taut_{\CP^q}^{\oplus p}$, where $\Taut_{\CP^q}$ is the complex tautological line bundle over $\CP^q$. Hence, for $G/H$ as above, we have $N=\Taut_{\CP^q}^{\oplus p} \oplus \R^k$. By definition of $j_{\CP^q}$, the sphere bundle $S(N)$ is fiber-homotopically trivial for $k$ large enough.\footnote{The factor $\Hyp^k_\R$ is here because the $\widetilde J$-homomorphism detects \emph{stable} spherical homotopical triviality. We do not know whether it detects spherical homotopical triviality for multiples of $\Taut_{\CP^q}$. If it is the case, Proposition \ref{prop: Normal Bundle homotopically but not linearly trivial} holds without taking a product with $\Hyp^k_\R$.}

On the other hand, the first Pontryagin class of $N$ is equal to $-p c_1^2$, where $c_1$ is the first Chern class of $\Taut_{\CP^q}$. This class is nonzero as soon as $q\geq 2$, proving that $N$ is nontrivial.
\end{proof}

\begin{prop} \label{prop:trivial-bundle-no-fibrations}
Let $G/H$ be one of the reductive symmetric spaces in the following table, where $n \geq 1$ is an integer satisfying the specified conditions. Then the normal bundles to $K/K_H$ in $G/H$ and its associated symmetric space $G/H^a$ are trivial, yet $G/H$ and $G/H^a$ do not admit local geometric fibrations.
\begin{center}
\begin{longtable}{c|c|c|c|c}
& $G$ & $H$ & $H^a$ & Conditions \\ \hline
\textup{(1)} & 
$\SO_0(n,n)$ & $\SO(n,\C)$ & $\GL^+(n,\R)$ & $n \geq 2$ \\ 
\textup{(2)} & 
$\SU(n,n)$ & $\GL(n,\C)$ & $\GL(n,\C)$ & --- \\ 
\textup{(3)} & 
$\SU(n,n)$ & $\OO^\ast(2n)$ & $\Sp(2n, \R)$ & $n \geq 2$ \\ 
\textup{(4)}
& $\Sp(n,n)$ & $\Sp(2n, \C)$ & $\GL(n,\Ha)$ & --- 
\end{longtable}
\end{center}
\end{prop}

\begin{proof}
In all these cases, the maximal compact subgroup $K$ contains a closed subgroup $L$ acting simply transitively on $K/K_H$, namely, 
\begin{itemize}
\item $L = \SO(n) \times \{ 1 \} \subset \SO(n) \times \SO(n)$ in (1); 
\item $L = \SU(n) \times \{ 1 \} \subset \operatorname{S}(\U(n) \times \U(n))$ in (2) and (3); 
\item $L = \Sp(n)\times \{1\} \subset \Sp(n) \times \Sp(n)$ in (4). 
\end{itemize}

The action of $L$ on $K/K_H$ lifts to any vector bundle $E$ associated to the principal $K_H$-bundle $K\to K/K_H$ and, since it is simply transitive, transporting a frame of $E$ at an arbitrary base point by the action of $L$ defines a trivialization of $E$. In particular, the normal bundle to $K/K_H$ in $G/H$ and $G/H^a$ are both trivial.

On the other hand, in all cases, we have $\rk_\R(H^a)= \rk_\R(G)$, hence neither $G/H$ nor $G/H^a$ admit local geometric fibrations by Remark \ref{rmk:CalabiMarkusAndH^a}.
\end{proof}

\begin{rmk}
In all cases of Proposition~\ref{prop:trivial-bundle-no-fibrations}, the space $G/H^a$ does not admit compact quotients by the Calabi--Markus phenomenon \cite{Kob89}. The same is true of $G/H$ in cases (2) and~(4). 

In case~(1), the second author previously proved that $G/H$ does not admit compact quotients \cite[Cor.\,1.4]{Mor15}. 

In case~(3), Kobayashi proved that $G/H$ does not admit compact quotients for even~$n$ \cite[Table\,4.4]{Kob92coh}. As far as we know, it is an open question whether $G/H$ admits compact quotients for odd~$n$. Conjecture \ref{conj:NoLocalFibrationObstruction} predicts that the answer should be negative in this case also. 
\end{rmk}

\subsection{A bolder conjectural implication}

The implication (ii)~$\Rightarrow$~(i) of Conjecture \ref{conj:CompactQuotients<=>Fibrations} seems more audacious:

\begin{conj} \label{conj:GeometricFibrations=>CompactQuotients}
If a reductive homogeneous space admits local geometric fibrations, then it admits compact quotients.
\end{conj}

As mentioned in Section~\ref{ss:MotivationGeometricFibration}, this conjecture is incompatible with Kobayashi's Conjecture~\ref{conj:ExistenceStandardQuotients}, the case of $\Hyp_\R^{4,2}$ being an interesting test (Question~\ref{question:H42}).
Evidence in favor of Conjecture~\ref{conj:GeometricFibrations=>CompactQuotients} is provided by the fact that many obstructions to the existence of compact quotients of $G/H$ are also obstructions to the existence of local geometric fibrations, summarized as follows:

\begin{prop} 
Let $G/H$ be a reductive homogeneous space. Assume $H$ is neither compact nor cocompact in~$G$, and does not contain an indiscrete closed normal subgroup of~$G$. Assume that one of the following conditions holds:
\begin{enumerate}[label = {\upshape (\arabic*)}]
  \item there exists a reductive subgroup $H' \subset G$ satisfying $\mu(H') \subset \mu(H)$ and $\dim_+(H') > \dim_+(H)$,
  \item the fixed point set of the opposition involution $\iota \colon \ab^+ \to \ab^+$ of $G$ is contained in $\mu(H)$,
  \item the $G$-invariant differential form $({\pi_K}_* \circ \pi_H^\ast)(\vol_{G/H})$ on $G/K$ vanishes,
  \item the normal sphere bundle to $K/K_H$ in $G/H$ is not fiber-homotopically trivial.
\end{enumerate}
Then the following conclusions both hold:
\begin{enumerate}[label = {\upshape (\alph*)}]
  \item the space $G/H$ does not admit compact quotients,
  \item the space $G/H$ does not admit local fibrations.
\end{enumerate}
Here, $\mu \colon G \to \ab^+$ denotes the Cartan projection associated to the decomposition $G = K \exp(\ab^+) K$, the maps $\pi_K \colon G/K_H \to G/K$ and $\pi_H \colon G/K_H \to G/H$ are the projections, ${\pi_K}_*$ denotes the integration along fibers, and $\vol_{G/H}$ is the $G$-invariant volume form on $G/H$ (which exists and is unique up to a nonzero multiplicative scalar). 
\end{prop}

The implication (1) $\Rightarrow$ (a) was proved by Kobayashi \cite[Th.\,1.5]{Kob92coh}, and (1) $\Rightarrow$ (b) can be proved in the same way, as remarked by Tojo~\cite[\S 5]{Toj21}. Note that the Calabi--Markus condition $\rk_\R(H)= \rk_\R(G)$ is a special case of condition~(1), with $H'=G$.

The implications (3) $\Rightarrow$ (a) and (3) $\Rightarrow$ (b) were proved by the third author in \cite{Tho15+}. This obstruction is closely related to the one developed by the second author in \cite{Mor15,Mor17} and contains Kobayashi--Ono's obstruction \cite[Cor.\,5]{KO90}, \cite[Prop.\,4.10]{Kob89} as well as Benoist--Labourie's \cite{BL92}.

The implications (4) $\Rightarrow$ (a) and (4) $\Rightarrow$ (b) are respectively Theorem~\ref{thm:MainTheorem} and Lemma~\ref{lem:ImplicationsTriviality}.

The implication (2) $\Rightarrow$ (a) was proved by Benoist \cite[\S 1.2, Cor.\,1]{Ben96}.
The remaining implication (4) $\Rightarrow$ (b) will be proved in further work.

There are other obstructions to the existence of compact quotients that we do not know how to relate to local geometric fibrations, namely, the dynamical obstructions developed by Labourie--Mozes--Zimmer \cite{Zim94,LMZ95,LZ95}, the representation-theoretic obstructions of Margulis \cite[\S 3]{Mar97} and Shalom \cite[Th.\,1.7]{Sha00}, or the geometrico-dynamical obstruction of the first and third authors \cite[\S 1.6]{KT24+}. However, we do not know any counterexample to Conjecture \ref{conj:GeometricFibrations=>CompactQuotients} produced by these obstructions. For instance, the main achievement of the approach of Labourie--Mozes--Zimmer is the fact that $\SL(p+q,\R)/\SL(p,\R)$ does not admit compact quotients for $p\geq 2$ and $q\geq 3$, a result which is also covered by our Theorem \ref{thm:SLn/SLm}.

\appendix

\section{More on topological \texorpdfstring{$\Ktheory$}{K}-theory and the \texorpdfstring{$J$}{J}-group} \label{s:AppendixKTheory}

In Section \ref{s:RemindersKTheory}, we introduced the basics of topological $\Ktheory$-theory and the $J$-group that were sufficient for the proofs of Theorem~\ref{thm:OtherSymmetric}~(1), Theorem~\ref{thm:NonSymmetric}~(1)--(3) and (6), and Theorem~\ref{thm:IndefiniteGrassmannianRCH}  in Sections \ref{s:IndefiniteGrassmannians}, \ref{ss:GrassR}, \ref{sss:GLC/GLC}, and \ref{sss:SpC/SpC}.

In this appendix, we collect further results that are used in Section \ref{s:ComputationsKTheory}. Results in Section~\ref{ss:AdamsOperations} 
are used in the remaining parts of Sections~\ref{ss:GrassC} and \ref{ss:GrassH}, and  Section~\ref{ss:OtherSymmetric}, whereas results in the other sections are needed only in Section~\ref{ss:OtherSymmetric}. We refer for instance to \cite{Ati67,Bot69,Kar78} for more details.

\subsection{Adams operations and the Adams conjecture} \label{ss:AdamsOperations}

In this Section~\ref{ss:AdamsOperations}, we briefly recall the definition and basic properties of \emph{Adams operations}, introduced in Adams \cite{Ada62}, and then state the \emph{Adams conjecture} which enables us to compute $J$-groups by using these operations. 

For $n \geq 1$, let $\sigma_{n,1}, \dots, \sigma_{n,n} \in \Z[x_1, \dots, x_n]$ be the elementary symmetric polynomials in $n$ variables, namely, 
\[
\sigma_{n,1} = \sum_{1 \leq i \leq n} x_i, \qquad \sigma_{n,2} = \sum_{1 \leq i<j \leq n} x_ix_j, \qquad \dots, \qquad \sigma_{n,n} = \prod_{1 \leq i \leq n} x_i.
\]
For each $k \geq 1$, let $Q_k \in \Z[s_1, \dots, s_k]$ be the $k$-variable polynomial characterized by the property 
\[
\sum_{1 \leq i \leq n} x_i^k = Q_k(\sigma_{n,1}, \dots, \sigma_{n,k}) \qquad (n \geq k).
\]
For example, 
\[
Q_1 = s_1, \qquad Q_2 = s_1^2 - 2s_2, \qquad Q_3 = s_1^3 - 3s_1s_2 + 3s_3, \qquad \dots. 
\]

\begin{defn}
Let $X$ be a compact Hausdorff space and $k \geq 1$. We define the \emph{$k$-th Adams operation} 
\[
\psi^k \colon \KO^0(X) \to \KO^0(X)
\]
to be the group homomorphism characterized by 
\[
\psi^k(\br{E}) = Q_k(\br{E}, \br{\Lambda^2 E}, \dots, \br{\Lambda^k E}) \qquad (E \in \Vect_\R(X)).  
\]
We similarly define $\psi^k \colon \KU^0(X) \to \KU^0(X)$. 
\end{defn}

Note that the first Adams operation $\psi^1$ is the identity map. 

\begin{rmk}
One can define the Adams operations on $\KO^0(X) \oplus \KSp^0(X)$, extending those on $\KO^0(X)$ (see Allard \cite[Ch.\,1]{All73}). We do not need them in this paper. 
\end{rmk}

The Adams operations are functorial, 
\ie for any continuous map $f \colon X \to Y$ of compact Hausdorff spaces and any $k \geq 1$, the following diagram commutes: 
\[
\begin{tikzcd}
\KU^0(Y) \ar[r, "\psi^k"] \ar[d, "f^\ast"'] & \KU^0(Y) \ar[d, "f^\ast"] \\ 
\KU^0(X) \ar[r, "\psi^k"] & \KU^0(X),
\end{tikzcd}
\qquad
\begin{tikzcd}
\KO^0(Y) \ar[r, "\psi^k"] \ar[d, "f^\ast"'] & \KO^0(Y) \ar[d, "f^\ast"] \\ 
\KO^0(X) \ar[r, "\psi^k"] & \KO^0(X).
\end{tikzcd}
\]
It follows that, if $X$ is a based compact Hausdorff space, the Adams operations restricts to: 
\[
\psi^k \colon \widetilde{\KU}^0(X) \to \widetilde{\KU}^0(X), \qquad 
\psi^k \colon \widetilde{\KO}^0(X) \to \widetilde{\KO}^0(X).  
\]

Here are some basic properties of Adams operations: 

\begin{fact}[{\cite[Th.\,5.1 and Cor.\,5.2]{Ada62}, \cite[Lem.\,A.2]{AW65}}]\label{fact:AdamsProperties}
Let $X$ be a compact Hausdorff space and $k \geq 1$. 
\begin{enumerate}[label = {\upshape (\arabic*)}]
\item $\psi^k \colon \KO^0(X) \to \KO^0(X)$ and $\psi^k \colon \KU^0(X) \to \KU^0(X)$ 
are ring endomorphisms. 
\item The following two diagrams commute: 
\[
\begin{tikzcd}
\KO^0(X) \ar[d, "\cc "'] \ar[r, "\psi^k"] & \KO^0(X) \ar[d, "\cc "] \\
\KU^0(X) \ar[r, "\psi^k"] & \KU^0(X),
\end{tikzcd}
\qquad
\begin{tikzcd}
\KU^0(X) \ar[d, "\rr "'] \ar[r, "\psi^k"]& \KU^0(X) \ar[d, "\rr "] \\
\KO^0(X) \ar[r, "\psi^k"] & \KO^0(X),
\end{tikzcd}
\]
\item If $L$ is a real (\resp complex) line bundle over $X$, then $\psi^k([L]) = [L]^k$ in $\KO^0(X)$ (\resp in $\KU^0(X)$).  
\end{enumerate}
\end{fact}

We use the following computations in Section \ref{s:ComputationsKTheory}:

\begin{lem}\label{lem:SecondAdams}
Let $X$ be a compact Hausdorff space. 
\begin{enumerate}[label = {\upshape (\arabic*)}]
\item Let $L$ be a complex line bundle over $X$, and let $u = [L]-1 \in \widetilde{\KU}^0(X)$. Then 
\[
u^2 = \psi^2(u) - 2u. 
\]
\item Let $L$ be a quaternionic line bundle over $X$, and let $v = [L]-1 \in \widetilde{\KSp}^0(X)$. Then 
\[
2 v^2 = \psi^2 (\rr \cc' v) - 4\rr \cc' v. 
\]
\end{enumerate}
\end{lem}

\begin{proof}
(1) By Fact~\ref{fact:AdamsProperties}~(3), we have 
\[
u^2 = ([L]-1)^2 = \psi^2([L]) - 2[L] + 1 = \psi^2(u) - 2u. 
\]

(2) Since $\cc' L$ is a complex $2$-plane bundle with trivial determinant, we have $\Lambda^2 (\cc' L) = \underline{\C}$. This and Remark~\ref{rmk:KSpProduct} together yield 
\[
\psi^2(\cc' [L]) = (\cc' [L])^2 - 2[\Lambda^2 (\cc' L)] = \cc ([L]^2) - 2.
\]
We thus have 
\[
\cc (v^2) = \cc ([L]^2 - 2\rr \cc' [L] + 4)
= ( \psi^2(\cc' [L]) + 2 ) - 4\cc' [L] + 4
= \psi^2(\cc' v) - 4\cc' v
\]
and, by Fact~\ref{fact:AdamsProperties}~(2),
\[
2v^2 = \rr \cc(v^2) = \psi^2(\rr \cc' v) - 4\rr \cc' v. \qedhere
\]
\end{proof}

Let us now state the \emph{Adams conjecture} \cite[Conj.\,1.2]{Ada63}, which keeps its name even though it was solved by Quillen \cite[Th.\,1.1]{Qui71} more than 50 years ago:

\begin{fact}[Quillen]\label{fact:AdamsConjecture1}
Let $X$ be a finite CW complex and $k \geq 1$. Then for any $x \in \KO^0(X)$, there exists $e \in \N$ such that
$k^e J(\psi^k(x) - x) = 0$ in $J(X)$. 
\end{fact}

\begin{rmk}
Many different proofs of the Adams conjecture are now known (\eg Quillen--Friedlander \cite{Qui68,Fri73}, Sullivan \cite{Sul74}, Becker--Gottlieb \cite{BG75}). 
\end{rmk}

Adams conjectured Fact~\ref{fact:AdamsConjecture1} since it would give a description of $\widetilde{J}(X)$ in terms of $\widetilde{\KO}^0(X)$ and the Adams operations on it: 

\begin{fact}[{Adams \cite[Prop.\,3.1 and Th.\,6.1]{Ada65a}, \cite[Th.\,1.1]{Ada65b} + Quillen \cite[Th.\,1.1]{Qui71}}]\label{fact:AdamsConjecture2}
Let $X$ be a based finite CW complex. Then 
the following two conditions on $x \in \widetilde{\KO}^0(X)$ are equivalent: 
\begin{enumerate}[label = {\upshape (\roman*)}]
\item $\widetilde{J}(x) = 0$ in $\widetilde{J}(X)$. 
\item For any map $e \colon \N_{\geq 1} \to \N$, there exist $n \in \N$, $k_1, \dots, k_n \geq 1$, and $y_1, \dots, y_n \in \widetilde{\KO}^0(X)$ such that 
\[
x = \sum_{i=1}^n k_i^{e(k_i)}(\psi^{k_i}(y_i) - y_i).
\]
\end{enumerate}
\end{fact}

We use in Section \ref{ss:OtherSymmetric} the following consequence of the Adams conjecture.

\begin{lem}\label{lem:ConclusionOfAdams1}
Let $f \colon X \to Y$ be a based continuous map of based finite CW complexes. If $f^\ast \colon \widetilde{\KO}^0(Y) \to \widetilde{\KO}^0(X)$ is an isomorphism, then $f^\ast \colon \widetilde{J}(Y) \to \widetilde{J}(X)$ is also an isomorphism. 
\end{lem}

\begin{proof}
This is because condition~(ii) in Fact~\ref{fact:AdamsConjecture2} refers only to $\widetilde{\KO}^0(X)$ and the Adams operations on it, not $X$ itself. 
\end{proof}

\subsection{Extension to negative degrees} \label{ss:NegativeDegrees}

\begin{defn} \label{def:smash-prod}
Let $X = (X, x_0)$ and $Y = (Y, y_0)$ be two based compact Hausdorff spaces. We set 
\[
X \vee Y = X \times \{ y_0 \} \cup \{ x_0 \} \times Y, \qquad 
X \wedge Y = (X \times Y) / (X \vee Y), 
\]
and call $X \wedge Y$ the \emph{smash product} of $X$ and $Y$. 
\end{defn}

\begin{rmk}\label{rmk:SmashProduct}
The one-point compactification and the removal of base point gives an equivalence of the following two categories: 
\begin{enumerate}[label = (\roman*)]
  \item the category of based compact Hausdorff spaces and based continuous maps,
  \item the category whose objects are (unbased) locally compact Hausdorff spaces and 
whose morphisms from $X$ to $Y$ are proper maps from some open subset of $X$ to $Y$. 
\end{enumerate}
One may view (ii) as a category for ``compactly supported'' version of topology and homotopy theory. 
Under this equivalence, the smash product in (i) corresponds to the Cartesian product in (ii). 
\end{rmk}

\begin{defn}
For a based compact Hausdorff space $X = (X, x_0)$ and $n \in \N$, we define the \emph{$(-n)$-th reduced $\KO$-theory}, denoted as $\widetilde{\KO}^{-n}(X)$, by 
\[
\widetilde{\KO}^{-n}(X) = \widetilde{\KO}^0(X \wedge \Sph^n), 
\]
We similarly define $\widetilde{\KU}^{-n}(X)$ and $\widetilde{\KSp}^{-n}(X)$ for $n \in \N$. 
\end{defn}

\begin{rmk}
In view of Remark~\ref{rmk:SmashProduct}, the above definition of $\widetilde{\KO}^{-n}(X)$ amounts to defining the ``$(-n)$-th compactly supported $\KO$-theory'' of a locally compact Hausdorff space $X$ 
as the ``$0$-th compactly supported $\KO$-theory'' of $X \times \R^n$. 
\end{rmk}

\begin{defn}
Let $X$ be a compact Hausdorff space and $A$ a closed subspace of $X$. For $n \in \N$, 
we define the \emph{$(-n)$-th relative $\KO$-theory} of the pair $(X, A)$, denoted as $\KO^{-n}(X, A)$, by 
\[
\KO^{-n}(X, A) = \widetilde{\KO}^{-n}(X/A). 
\]
We similarly define $\widetilde{\KU}^{-n}(X, A)$ and $\widetilde{\KSp}^{-n}(X, A)$. 
\end{defn}

If $A = \varnothing$, we will write $\KO^{-n}(X)$ instead of $\KO^{-n}(X, \varnothing)$, and similarly for $\KU$ and $\KSp$.
This notation is compatible with the previous definition for $n = 0$ since
$X / \varnothing$ is, by convention, a disjoint union of $X$ and the one-point space. 

There is an alternative description of $\KO^0(X, A)$ convenient for explicit computations: 

\begin{fact}[{see \eg \cite[Th.\,2.6.2]{Ati67}}]\label{fact:RelativeKO^0}
Let $X$ be a compact Hausdorff space and $A$ a closed subspace of $X$. 
Let $\Vect_\R(X, A)$ be the set of isomorphism classes of triples $(E^0, E^1, \varphi)$, where $E^0$ and $E^1$ are $\R$-vector bundles over $X$ and $\varphi \colon E^0_{|A} \simeq E^1_{|A}$ is an isomorphism over $A$. 
Then we have a canonical isomorphism 
\[
\Ktheory(\Vect_\R(X, A)) \simeq \KO^0(X, A). 
\]
The same holds for $\KU$ and $\KSp$. 
\end{fact}

Let $X = (X, x_0)$ be a compact Hausdorff space and $A = (A, x_0)$ be a closed subspace of $X$ containing~$x_0$. Then 
there is a canonical long exact sequence 
\begin{equation}\label{eqn:LongExactSeq}
\cdots \to \widetilde{\KO}^{-n-1}(A) \to \KO^{-n}(X, A) \to \widetilde{\KO}^{-n}(X) \to \widetilde{\KO}^{-n}(A) \to \cdots \to \widetilde{\KO}^0(A),
\end{equation}
and similarly for unbased pairs (see \eg \cite[Prop.\,2.4.4]{Ati67}). The same holds for $\KU$ and $\KSp$. 

The following facts are formal consequences of the long exact sequence (\ref{eqn:LongExactSeq}): 

\begin{fact}[{see \eg \cite[Cor.\,2.4.7]{Ati67}}] \label{fact:Unbased}
Let $X = (X, x_0)$ be a based compact Hausdorff space. For any $n \in \N$, there is a canonical direct sum decomposition 
\[
\KO^{-n}(X) \simeq \widetilde{\KO}^{-n}(X) \oplus \KO^{-n}(\pt).
\]
The same holds for $\KU$ and $\KSp$.
\end{fact}

\begin{fact}[{see \eg \cite[Cor.\,2.4.8]{Ati67}}] \label{fact:KTheoryProduct}
Let $X = (X, x_0)$ and $Y = (Y, y_0)$ be two based compact Hausdorff spaces. 
Define $i_1 \colon X \to X \times Y$ and $i_2 \colon Y \to X \times Y$ by $i_1(x) = (x, y_0)$ and $i_2(y) = (x_0, y)$, respectively, and let $q \colon X \times Y \to X \wedge Y$ be the quotient map. Then 
\[
0 \to \widetilde{\KO}^{-n}(X \wedge Y) \xrightarrow{q^\ast} \widetilde{\KO}^{-n}(X \times Y) \xrightarrow{(i_1^\ast, i_2^\ast)}
\widetilde{\KO}^{-n}(X) \oplus \widetilde{\KO}^{-n}(Y) \to 0 
\]
is a split exact sequence. The splitting is given by 
\[
(\pr_1^\ast, \pr_2^\ast) \colon \widetilde{\KO}^{-n}(X) \oplus \widetilde{\KO}^{-n}(Y) \to \widetilde{\KO}^{-n}(X \times Y),
\]
where $\pr_1 \colon X \times Y \to X$ and $\pr_2 \colon X \times Y \to Y$ are the projection maps. The same holds for $\KU$ and $\KSp$. 
\end{fact}

We use in Section \ref{ss:OtherSymmetric} the following consequence:
\begin{lem}\label{lem:ConclusionOfAdams2}
Let $X = (X, x_0)$ and $Y = (Y, y_0)$ be two based finite CW complexes. Let $\pr_1, \pr_2$, and $q$ be as in Fact~\ref{fact:KTheoryProduct}~(2). Then 
\[
(\pr_1^\ast, \pr_2^\ast, q^\ast) \colon \widetilde{J}(X) \oplus \widetilde{J}(Y) \oplus \widetilde{J}(X \wedge Y) \to \widetilde{J}(X \times Y)
\]
is an isomorphism. 
\end{lem}

\begin{proof}
By Fact~\ref{fact:KTheoryProduct}~(2), we have an isomorphism
\[
(\pr_1^\ast, \pr_2^\ast, q^\ast) \colon \widetilde{\KO}^0(X) \oplus \widetilde{\KO}^0(Y) \oplus \widetilde{\KO}^0(X \wedge Y) \xrightarrow{\simeq} \widetilde{\KO}^0(X \times Y).
\]
This isomorphism respects Adams operations. Arguing as in the proof of Lemma~\ref{lem:ConclusionOfAdams1}, one sees that the induced homomorphism between reduced $J$-groups is an isomorphism. 
\end{proof}

\subsection{Product structure} \label{ss:Product}

In this section, we extend the product structure in topological $\Ktheory$-theory to based spaces and also to negative degrees, which is needed for the Bott periodicity theorem (Fact~\ref{fact:BottPeriodicity}) in the next section.

Let $X = (X, x_0)$ and $Y = (Y, y_0)$ be two based compact Hausdorff spaces. It follows from Fact~\ref{fact:KTheoryProduct}~(2) that the ``outer'' multiplication 
\[
\KO^0(X) \times \KO^0(Y) \to \KO^0(X \times Y), \qquad (a, b) \mapsto \pr_1^\ast(a) \cdot \pr_2^\ast(b)
\]
induces a $\Z$-bilinear map 
\[
\wedge \colon \widetilde{\KO}^0(X) \times \widetilde{\KO}^0(Y) \to \widetilde{\KO}^0(X \wedge Y). 
\]
This extends, by definition, to: 
\[
\wedge \colon \widetilde{\KO}^{-m}(X) \times \widetilde{\KO}^{-n}(Y) \to \widetilde{\KO}^{-m-n}(X \wedge Y) \qquad (n,m \in \N).
\]
We similarly obtain:  
\begin{align*}
\wedge \colon \widetilde{\KU}^{-m}(X) \times \widetilde{\KU}^{-n}(Y) &\to \widetilde{\KU}^{-m-n}(X \wedge Y), \\
\wedge \colon \widetilde{\KO}^{-m}(X) \times \widetilde{\KSp}^{-n}(Y) &\to \widetilde{\KSp}^{-m-n}(X \wedge Y), \\
\wedge \colon \widetilde{\KSp}^{-m}(X) \times \widetilde{\KSp}^{-n}(Y) &\to \widetilde{\KO}^{-m-n}(X \wedge Y).
\end{align*}

We now define the ``inner'' multiplication $\widetilde{\KO}^{-m}(X) \times \widetilde{\KO}^{-n}(X) \to \widetilde{\KO}^{-m-n}(X)$ by 
\[
\widetilde{\KO}^{-m}(X) \times \widetilde{\KO}^{-n}(X) \xrightarrow{\wedge} \widetilde{\KO}^{-m-n}(X \wedge X) \xrightarrow{\Delta^\ast} \widetilde{\KO}^{-m-n}(X), 
\]
where $\Delta \colon X \to X \wedge X$ is the diagonal embedding. Replacing $X$ with its disjoint union with the one-point space, we obtain the unbased version
\[
\KO^{-m}(X) \times \KO^{-n}(X) \to \KO^{-m-n}(X),
\]
which is an extension of the multiplication $\KO^0(X) \times \KO^0(X) \to \KO^0(X)$ introduced before. We similarly define 
\begin{align*}
\widetilde{\KU}^{-m}(X) \times \widetilde{\KU}^{-n}(X) &\to \widetilde{\KU}^{-m-n}(X), \\
\widetilde{\KO}^{-m}(X) \times \widetilde{\KSp}^{-n}(X) &\to \widetilde{\KSp}^{-m-n}(X), \\
\widetilde{\KSp}^{-m}(X) \times \widetilde{\KSp}^{-n}(X) &\to \widetilde{\KO}^{-m-n}(X),
\end{align*}
as well as their unbased versions.

\subsection{Extension to positive degrees} \label{ss:PositiveDegrees}

For $n \in \N$ and $\K = \R, \C$, or $\Ha$, we write $\Taut_{\KP^n}$ for the \emph{tautological line bundle} over the projective space $\KP^n$, namely, the $\K$-line bundle whose the fiber over $x \in \KP^n$ is the line in $\K^{n+1}$ defining $x$: 
\[
\Taut_{\KP^n} = \{ (x, v) \in \KP^n \times \K^{n+1} \mid v \in x \}.
\]

\begin{defn}\label{defn:BottElement}
We set 
\begin{align*}
\beta &= 1 - [ \Taut_{\CP^1} ] \in \widetilde{\KU}^0(\CP^1) = \widetilde{\KU}^0(\Sph^2), \\
\alpha &= 1 - [ \Taut_{\HP^1} ] \in \widetilde{\KSp}^0(\HP^1) = \widetilde{\KSp}^0(\Sph^4),
\end{align*}
and
$\beta_\R = \alpha \wedge \alpha \in \widetilde{\KO}^0(\Sph^8)$. 
We call $\beta$ and $\beta_\R$ the \emph{complex} and the \emph{real Bott elements}, respectively. 
\end{defn}

Then the following \emph{Bott periodicity theorem} holds: 

\begin{fact}[{\cite[Cor.\ to Th.\,2]{Bot59a}, \cite[Th.\,1]{Bot59b}}] \label{fact:BottPeriodicity}
For any based compact Hausdorff space $X = (X, x_0)$, the following three multiplication maps are all isomorphisms: 
\begin{align*}
\widetilde{\KU}^0(X) \to \widetilde{\KU}^0(X \wedge \Sph^2) = \widetilde{\KU}^{-2}(X), &\qquad x \mapsto x \wedge \beta, \\
\widetilde{\KO}^0(X) \to \widetilde{\KSp}^0(X \wedge \Sph^4) = \widetilde{\KSp}^{-4}(X), &\qquad x \mapsto x \wedge \alpha, \\
\widetilde{\KSp}^0(X) \to \widetilde{\KO}^0(X \wedge \Sph^4) = \widetilde{\KO}^{-4}(X), &\qquad x \mapsto x \wedge \alpha.
\end{align*}
\end{fact}

Note that the latter two isomorphisms together imply that 
the following maps are isomorphisms: 
\begin{align*}
\widetilde{\KO}^0(X) \to \widetilde{\KO}^{-8}(X), &\qquad x \mapsto x \wedge \beta_\R, \\
\widetilde{\KSp}^0(X) \to \widetilde{\KSp}^{-8}(X), &\qquad x \mapsto x \wedge \beta_\R. 
\end{align*}

We thus define $\widetilde{\KO}^n(X)$, $\widetilde{\KU}^n(X)$, and $\widetilde{\KSp}^n(X)$ for general $n \in \Z$ as follows: 

\begin{defn}
For a based compact Hausdorff space $X = (X, x_0)$ and $n \in \Z$, we set 
\begin{align*}
\widetilde{\KO}^n(X) = \widetilde{\KO}^{n - 8k}(X) &\qquad (n - 8k \leq 0), \\
\widetilde{\KU}^n(X) = \widetilde{\KU}^{n - 2k}(X) &\qquad (n - 2k \leq 0), \\
\widetilde{\KSp}^n(X) = \widetilde{\KSp}^{n - 8k}(X) &\qquad (n - 8k \leq 0).
\end{align*}
\end{defn}

We can define the unbased and the relative versions of $\KU^n$, $\KO^n$, and $\KSp^n$ for general $n \in \Z$ by the same procedure as before. 
The long exact sequence (\ref{eqn:LongExactSeq}) and its unbased version extend to positive degrees, and one can see that $\KO = (\KO^n)_{n \in \Z}$ satisfies the axioms of generalized cohomology theory.
The same is true for $\KU$ and $\KSp$.
All results and definitions in Section~\ref{ss:Product} extend to positive degrees.

\subsection{Clifford algebras, spinor representations, and the spin group}\label{ss:CliffordAlgebras}

By a quadratic space over a field $\K$, we mean a finite-dimensional $\K$-vector space equipped with a nondegenerate quadratic form. 

\begin{defn}
Let $(V, Q)$ be a quadratic space over a field $\K$. We define the \emph{Clifford algebra} $\Cl(V, Q)$ by 
$\Cl(V, Q) = TV / I_Q$, 
where $TV = (TV, \otimes, 1)$ is the tensor algebra over $V$, and $I_Q$ is the two-sided ideal of $TV$ generated by 
$\{ v \otimes v + Q(v) \mid v \in V \}$.
We write $c \colon V \to \Cl(V, Q)$ for the obvious $\K$-linear map. 
\end{defn}

Our main interest is the case where $\K = \R$ and $(V, Q) = (\R^n, Q_\std)$, where $Q_\std(x_1, \dots, x_n) = x_1^2 + \dots x_n^2$ is the standard quadratic form. We often write $\Cl_n$ instead of $\Cl(\R^n)$.

\begin{rmk}
There are two sign conventions on Clifford algebras: 
some authors call $\Cl(V, Q)$ what we call $\Cl(V, -Q)$. 
\end{rmk}

The $\Z$-grading on tensor algebras induces a $\Z/2$-grading on Clifford algebras: 
\[
\Cl(V, Q) = \Cl^0(V, Q) \oplus \Cl^1(V, Q). 
\]

\begin{rmk}\label{rmk:CliffordZ/2Grading}
Here are some basic properties of $\Z/2$-grading on Clifford algebras: 

\begin{enumerate}
\item Let $(V, Q)$ and $(V', Q')$ be two quadratic spaces over $\K$. Then the following is an isomorphism of $\Z/2$-graded $\K$-algebras: 
\[
\Cl(V, Q) \otimes_\K \Cl(V', Q') \to \Cl(V \oplus V', Q \oplus Q'), \qquad c(v) \otimes c(v') \mapsto c(v,v'),
\]
where $\otimes_K$ is the \emph{$\Z/2$-graded} tensor product. 

\item For $\lambda \in \K^\times$, the following is an isomorphism of $\K$-algebras: 
\[
\Cl(V, \lambda Q) \to \Cl^0(V \oplus \K, Q \oplus \lambda), \qquad c(v) \mapsto c(v,0)c(0,1).
\]
In particular, we have 
$\Cl^0_{n+1} \simeq \Cl_n$. 

\item If $M = M^0 \oplus M^1$ is a $\Z/2$-graded $\Cl(V, Q)$-module, then $M^0$ is a $\Cl^0(V, Q)$-module. Conversely, 
if $M^0$ is a $\Cl^0(V, Q)$-module, then the tensor product $\Cl(V, Q) \otimes_{\Cl^0(V, Q)} M^0$ is a $\Z/2$-graded $\Cl(V, Q)$-module. Thus, the category of $\Z/2$-graded $\Cl(V, Q)$-modules is equivalent to that of $\Cl^0(V, Q)$-modules. 
\end{enumerate}
\end{rmk}

\begin{fact}\label{fact:CliffordComputation}
For $n \geq 0$, we have 
\[
\Cl_n \simeq 
\begin{cases}
\M(2^{n/2}, \R) & (n \equiv 0,6 \pmod 8), \\
\M(2^{(n-1)/2}, \R)^2 & (n \equiv 7 \pmod 8), \\
\M(2^{(n-1)/2}, \C) & (n \equiv 1,5 \pmod 8), \\
\M(2^{(n-2)/2}, \Ha) & (n \equiv 2,4 \pmod 8), \\
\M(2^{(n-3)/2}, \Ha)^2 & (n \equiv 3 \pmod 8).
\end{cases}
\]
\end{fact}

Using Fact~\ref{fact:CliffordComputation} and Remark~\ref{rmk:CliffordZ/2Grading}~(2)--(3), 
we define representations of $\Cl_n$ as follows:

\begin{defn} \label{defn:SpinorRepresentation1}
Let $n \geq 1$. 
\begin{enumerate}
\item If $n \equiv 1,7 \pmod 8$, 
we write $\Delta_n$ for the $\Z/2$-graded real representation of $\Cl_n$ whose even part is the standard representation $\R^{2^{(n-1)/2}}_\std$ of $\Cl_n^0 \simeq \M(2^{(n-1)/2}, \R)$. 
\item If $n \equiv 0 \pmod 8$, 
we write $\Delta_{n,+}$ and $\Delta_{n,-}$ for the $\Z/2$-graded real representations of $\Cl_n$ whose even parts are the standard representation $\R^{2^{(n-2)/2}}_\std$ of the first and the second copies of $\Cl_n^0 \simeq \M(2^{(n-2)/2}, \R)^2$, respectively. 
\item If $n \equiv 2,6 \pmod 8$, 
we write $\Delta_{n,+}$ and $\Delta_{n,-}$ for the $\Z/2$-graded complex representations of $\Cl_n$ whose even parts are the standard representation $\C^{2^{(n-2)/2}}_\std$ and its complex conjugate $\overline{\C^{2^{(n-2)/2}}_\std}$ of $\Cl_n^0 \simeq \M(2^{(n-2)/2}, \C)$, respectively.
\item If $n \equiv 3,5 \pmod 8$, 
we write $\Delta_n$ for the $\Z/2$-graded quaternionic representation of $\Cl_n$ whose even part is the standard representation $\Ha^{2^{(n-3)/2}}_\std$ of $\Cl_n^0 \simeq \M(2^{(n-3)/2}, \Ha)$. 
\item If $n \equiv 4 \pmod 8$, 
we write $\Delta_{n,+}$ and $\Delta_{n,-}$ for the $\Z/2$-graded quaternionic representations of $\Cl_n$ whose even parts are the standard representation $\Ha^{2^{(n-4)/2}}_\std$ of the first and the second copy of $\Cl_n^0 \simeq \M(2^{(n-4)/2}, \Ha)$, respectively. 
\end{enumerate}
\end{defn}

In (1) and (4), we call $\Delta_n$ the \emph{$\Z/2$-graded spinor representation} of $\Cl_n$. 
In (2), (3), and (5), we call $\Delta_{n,+}$ and $\Delta_{n,-}$ the \emph{positive} and the \emph{negative $\Z/2$-graded spinor representations} of $\Cl_n$, respectively. 

\begin{rmk} \label{rmk:SpinorRepresentationSign}
The above definition is slightly inprecise when $n$ is even: the outer automorphism of $\Cl_n$ exchanges $\Delta_{n,+}$ and $\Delta_{n,-}$ in (2), (3), and (5). If we fix an orientation of $\R^n$, one can distinguish them canonically. 
Let us explain one way to do so. 

Let us set $\gamma = c(e_1)c(e_2) \dots c(e_n) \in \Cl^0_n$, where $(e_1, \dots, e_n)$ is an oriented orthonormal basis of $\R^n$. It does not depend on the choice of an oriented orthonormal basis. Then we choose $\Delta_{n,+}$ (\resp $\Delta_{n,-}$) to be the one whose even part is acted by $\gamma$ as the scalar multiplication by $i^{n/2}$ (\resp $-i^{n/2}$).
\end{rmk}

\begin{rmk}
When $n=0$, the Clifford algebra $\Cl_0 = \R$ admits two irreducible $\Z/2$-graded real representations, namely, $\R_\std \oplus 0$ and $0 \oplus \R_\std$. We may call them $\Delta_{0,+}$ and $\Delta_{0,-}$, respectively. 
\end{rmk}

\begin{defn}
Let $n \geq 1$. 
\begin{enumerate}
\item The \emph{pin group} $\Pin(n)$ is the subgroup of the multiplicative group $\Cl_n^\times$ generated by the elements $c(v)$ for $v\in\R^n$ with $\| v\|=1$. 
\item The \emph{spin group} $\Spin(n)$ is the intersection $\Pin(n) \cap \Cl^0_n$. 
\end{enumerate}
\end{defn}

\begin{fact}[{see \eg \cite[\S\,3]{ABS64}}]
For $n \geq 1$, the pin group $\Pin(n)$ is a closed subgroup of $\Cl_n^\times$, and 
\[
\rho \colon \Pin(n) \to \OO(n), \qquad c(v) \mapsto r_v
\]
is a double cover of Lie groups, where $r_v \in \OO(n)$ is the orthogonal reflection with repsect to $v$. 
\end{fact}

\begin{rmk}
When $n \geq 2$, the induced homomorphism $\rho \colon \Spin(n) \to \SO(n)$ is the (unique) connected double cover of $\SO(n)$. 
\end{rmk}

\begin{defn} \label{defn:SpinorRepresentation2}
Let $n \geq 1$. 
\begin{enumerate}
\item If $n$ is odd, the even part $\Delta_n^0$ of $\Delta_n$, viewed as a representation of $\Spin(n) \subset \Cl^0_n$, is called the \emph{spinor representation} of $\Spin(n)$. 

\item If $n$ is even, the even part $\Delta_{n,+}^0$ (\resp $\Delta_{n,-}^0$) of $\Delta_{n,+}$ (\resp $\Delta_{n,-}$), viewed as a representation of $\Spin(n) \subset \Cl^0_n$, is called the \emph{positive} (\resp \emph{negative}) \emph{spinor representation} of $\Spin(n)$. 
\end{enumerate}
\end{defn}

\subsection{The Atiyah--Bott--Shapiro map} \label{ss:ABS}

For $\K = \R, \C$, or $\Ha$, let $\Rep_{\gr,\K}(\Cl_n)$ be the set of isomorphism classes of $\Z/2$-graded representations of $\Cl_n$-modules over $\K$. We denote their group completions by  
\begin{align*}
\RO_\gr(\Cl_n) &= \Ktheory(\Rep_{\gr,\R}(\Cl_n)), \\
\RU_\gr(\Cl_n) &= \Ktheory(\Rep_{\gr,\C}(\Cl_n)), \\
\RSp_\gr(\Cl_n) &= \Ktheory(\Rep_{\gr,\Ha}(\Cl_n)).
\end{align*}

Let $n \geq 0$ and let 
\[
B^n = \{ v \in \R^n \mid \| v \| \leq 1 \}, \qquad 
\Sph^{n-1} = \{ v \in \R^n \mid \| v \| = 1 \}.
\]
Let $M = M^0 \oplus M^1$ be a $\Z/2$-graded real representation of $\Cl_n$, and let $\underline{M} = \underline{M}^0 \oplus \underline{M}^1$ be the corresponding $\Z/2$-graded trivial vector bundle over $B^n$. 
For each $v \in \Sph^{n-1}$, the multiplication by $c(v)$ is an $\R$-linear isomorphism from $M^0$ to $M^1$. This gives an isomorphism of vector bundles
\[
\varphi_M \colon \underline{M}^0_{|\Sph^{n-1}} \to \underline{M}^1_{|\Sph^{n-1}}.
\]
We have thus obtained a monoid homomorphism 
\[
\chi \colon \Rep_{\gr,\R}(\Cl_n) \to \Vect_\R(\mathrm{B}^n, \Sph^{n-1}), \qquad 
M \mapsto (\underline{M}^0, \underline{M}^1, \varphi_M),
\]
which induces by Fact~\ref{fact:RelativeKO^0} a group homomorphism
\[
\chi \colon \RO_\gr(\Cl_n) \to \KO^0(\mathrm{B}^n, \Sph^{n-1}) \ (\simeq \widetilde{\KO}^0(\Sph^n)).
\]

\begin{defn}
We call $\chi \colon \RO_\gr(\Cl_n) \to \widetilde{\KO}^0(\Sph^n)$ the \emph{Atiyah--Bott-Shapiro map}. 
\end{defn}

One can similarly define the complex and the quaternionic versions of the Atiyah--Bott--Shapiro map: 
\[
\chi \colon \RU_\gr(\Cl_n) \to \widetilde{\KU}^0(\Sph^n), \qquad 
\chi \colon \RSp_\gr(\Cl_n) \to \widetilde{\KSp}^0(\Sph^n).
\]

Using the presentation $\Sph^n \simeq \Spin(n+1)/\Spin(n)$ of the sphere as a homogeneous space, we can give an alternative description of the Atiyah--Bott--Shapiro map. 

\begin{fact}[{Atiyah--Bott--Shapiro \cite[Th.\,14.3]{ABS64}}]\label{fact:ABSAlternative}
Let $\K = \R, \C$, or $\Ha$. 
Let $M = M^0 \oplus M^1$ be a $\Z/2$-graded representation of $\Cl_n$ over $\K$.
Then
\[
\chi(M) = [\Spin(n+1) \times_{\Spin(n)} M^0] - \dim_\K(M^0). 
\]
\end{fact}

\begin{rmk}\label{rmk:ABSAlternative}
The complex and the quaternionic cases are not discussed in \cite{ABS64}, but the same argument as the real case works. 
\end{rmk}

\begin{fact}[{Atiyah--Bott--Shapiro \cite[Th.\,11.5]{ABS64}}]\label{fact:ABS-RC}
For any $n \geq 0$, the following sequences are exact: 
\begin{align*}
&0 \to \RO_\gr(\Cl_{n+1}) \xrightarrow{\textup{rest.}} \RO_\gr(\Cl_n) \xrightarrow{\chi} \widetilde{\KO}^0(\Sph^n) \to 0, \\
&0 \to \RU_\gr(\Cl_{n+1}) \xrightarrow{\textup{rest.}} \RU_\gr(\Cl_n) \xrightarrow{\chi} \widetilde{\KU}^0(\Sph^n) \to 0.
\end{align*}
\end{fact}

We use in Section \ref{ss:OtherSymmetric} the following quaternionic analogue of Fact~\ref{fact:ABS-RC}: 

\begin{thm}\label{thm:ABS-H}
For any $n \geq 0$, the following sequence is exact:
\[
0 \to \RSp_\gr(\Cl_{n+1}) \xrightarrow{\textup{rest.}} \RSp_\gr(\Cl_n) \xrightarrow{\chi} \widetilde{\KSp}^0(\Sph^n) \to 0.
\]
\end{thm}

\begin{proof}
One can easily see from Fact~\ref{fact:ABSAlternative} that $\chi(\Delta_4) = \alpha$ in $\widetilde{\KSp}^0(\Sph^4)$. 
Also, for any $n \in \N$, the following map, defined via Remark~\ref{rmk:CliffordZ/2Grading}~(1), is an isomorphism of abelian groups: 
\[
\varphi \colon 
\RSp_\gr(\Cl_n) \to \RO_\gr(\Cl_{n+4}), \qquad [M] \mapsto [M \boxtimes_\Ha \overline{\Delta_4}].
\]
Hence, the following diagram commutes: 
\[
\begin{tikzcd}
0 \ar[r] & \RSp_\gr(\Cl_{n+1}) \ar[d, "\varphi"', "\simeq"] \ar[r, "\textup{rest.}"] & \RSp_\gr(\Cl_n) \ar[d, "\varphi"', "\simeq"] \ar[r, "\chi"] & \widetilde{\KSp}^0(\Sph^n) \ar[r] \ar[d, "{(-) \wedge \alpha}"', "\simeq"] & 0 \\
0 \ar[r] & \RO_\gr(\Cl_{n+5}) \ar[r, "\textup{rest.}"'] & \RO_\gr(\Cl_{n+4}) \ar[r, "\chi"'] & \widetilde{\KO}^0(\Sph^{n+4}) \ar[r] & 0
\end{tikzcd}
\]
The second row is an exact sequence, hence so is the first. 
\end{proof}

For the reader's convenience, 
we give a table of reduced $\Ktheory$-theories of spheres and their generators, obtained from Fact~\ref{fact:ABS-RC} and Theorem~\ref{thm:ABS-H}. This table is used in Section \ref{ss:OtherSymmetric}.

\begin{center}
\begin{longtable}{|c||c|c|c|c|c|c|c|c|} \hline
$n \bmod 8$ & $0$ & $1$ & $2$ & $3$ & $4$ & $5$ & $6$ & $7$ \\ \hline
$\widetilde{\KO}^0(\Sph^n)$ & $\Z$ & $\Z/2$ & $\Z/2$ & $0$ & $\Z$ & $0$ & $0$ & $0$ \\ 
Generator & $\chi(\Delta_{n,+})$ & $\chi(\Delta_n)$ & $\chi(\rr \Delta_{n,+})$ & --- & $\chi(\rr \cc' \Delta_{n,+})$ & --- & --- & --- \\ \hline
$\widetilde{\KU}^0(\Sph^n)$ & $\Z$ & $0$ & $\Z$ & $0$ & $\Z$ & $0$ & $\Z$ & $0$ \\ 
Generator & $\chi(\cc \Delta_{n,+})$ & --- & $\chi(\Delta_{n,+})$ & --- & $\chi(\cc' \Delta_{n,+})$ & --- & $\chi(\Delta_{n,+})$ & --- \\ \hline
$\widetilde{\KSp}^0(\Sph^n)$ & $\Z$ & $0$ & $0$ & $0$ & $\Z$ & $\Z/2$ & $\Z/2$ & $0$ \\ 
Generator & $\chi(\qq \cc \Delta_{n,+})$ & --- & --- & --- & $\chi(\Delta_{n,+})$ & $\chi(\Delta_n)$ & $\chi(\qq \Delta_{n,+})$ & ---\\ \hline
\caption{Reduced $\Ktheory$-theories of spheres and their generators}
\label{table:ABS}
\end{longtable}
\end{center}

\subsection{The reduced $J$-group of spheres}\label{ss:JTildeOfSphere}

Combining the Adams conjecture (Fact~\ref{fact:AdamsConjecture2}) and the above computation of $\widetilde{\KO}^0(\Sph^n)$, one can deduce the following: 

\begin{fact}[{Adams \cite[Ex.\,6.4--6.5]{Ada65a} + Quillen \cite[Th.\,1.1]{Qui71}}]
\label{fact:JTildeSphere}
Let $n \geq 1$. 
\begin{enumerate}[label = {\upshape (\arabic*)}]
  \item For $n \equiv 3,5,6,7 \pmod 8$, we have $\widetilde{J}(\Sph^n) \simeq 0$. 
  \item For $n \equiv 1,2 \pmod 8$, we have $\widetilde{J}(\Sph^n) \simeq \Z/2$. 
  \item For $n \equiv 0,4 \pmod 8$, define $s_n \in \N$ by 
\[
s_n = \prod_{\substack{p \textrm{ prime}\\2(p-1) \textrm{ divides }n}} p^{1+v_p(n)}
\]
where $v_p(n) \in \N$ is the $p$-adic valuation of $n$. Then we have $\widetilde{J}(\Sph^n) \simeq \Z/s_n$.
\end{enumerate}
\end{fact}

For small $n$, the values of $s_n$ are as follows: 

\begin{center}
\begin{longtable}{|c||c|c|c|c|c|c|c|c|c|c|} \hline
$n$ & $4$ & $8$ & $12$ & $16$ & $20$ & $24$ & $28$ & $32$ & $36$ & \dots \\ \hline
$s_n$ & $24$ & $240$ & $504$ & $480$ & $264$ & $65520$ & $24$ & $16320$ & $28728$ & \dots \\ \hline
\end{longtable}
\end{center}

\subsection{Thom isomorphism theorems and Thom--Gysin exact sequences}\label{ss:Thom}

Let $E$ be a real vector bundle of rank $d$ on a topological space $B$. 
A \emph{spin structure} on $X$ is a principal $\Spin(d)$-bundle $P$ on $B$ together with an isomorphism of vector bundles
\[
\varphi \colon P \times_{\Spin(d)} \R^d \to E. 
\]
A real vector bundle equipped with a spin structure is called a \emph{spin vector bundle}. 
We have the following Thom isomorphism theorem in $\KO$-theory, which we slightly reformulated from a more common version for convenience: 

\begin{thm}\label{fact:ThomIsom}
Let $X$ be a compact Hausdorff space, and 
let $\pi \colon E \to X$ be a spin vector bundle over $X$ of rank $d$. 
Then there exists a canonical element $\tau_E \in \KO^d(B(E), S(E))$, called the $\KO$-theoretic Thom class, such that the multiplication map  
\[
\KO^n(X) \to \KO^{n+d}(B(E), S(E)), \qquad x \mapsto \pi^\ast(x) \cdot \tau_E. 
\]
is an isomorphism for any $n \in \Z$. 
\end{thm}

Here, $B(E)$ and $S(E)$ denote respectively the unit disc bundle and unit sphere bundle of $E$.

\begin{proof}
For any $k \geq 0$, we have
\begin{align*}
\KO^{n+d+k}(B(E \oplus \underline{\R}^{\oplus k}), S(E \oplus \underline{\R}^{\oplus k}))
&= \widetilde{\KO}^{n+d+k}(B(E \oplus \underline{\R}^{\oplus k}) / S(E \oplus \underline{\R}^{\oplus k})) \\
&\simeq \widetilde{\KO}^{n+d+k}((B(E) / S(E)) \wedge \Sph^k) \\
&= \KO^{n+d}(B(E), S(E)). 
\end{align*}
Hence, by replacing $E$ with $E \oplus \underline{\R}^{\oplus k}$ whenever necessary, we may assume that $d$ is divisible by $8$. By the Bott periodicity, it suffices to show that 
there exists a canonical element $\tau_E \in \KO^0(B(E), S(E))$ such that the multiplication map  
\[
\KO^n(X) \to \KO^n(B(E), S(E)), \qquad x \mapsto \pi^\ast(x) \cdot \tau_E. 
\]
is an isomorphism for any $n \in \Z$. 
This is proved in, for instance, \cite[Th.\,12.3]{ABS64}.
Roughly speaking, $\tau_E \in \KO^0(D(E), S(E))$ is defined as the image of $\Delta_{d,+}$ under the family version of Atiyah--Bott--Shapiro map. 
\end{proof}

Let $X$ be a compact Hausdorff space, and let $\pi \colon E \to X$ be a spin vector bundle of rank $d$. 
Consider the long exact sequence for the pair $(B(E), S(E))$: 
\[
\cdots \to \KO^n(B(E), S(E)) \to \KO^n(B(E)) \to \KO^n(S(E)) 
\to \KO^{n+1}(B(E), S(E)) \to \cdots.
\]
Observing that 
$\pi \colon B(E) \to X$ induces an isomorphism $\KO^n(X) \simeq \KO^n(B(E))$ and applying Fact~\ref{fact:ThomIsom}, we can rewrite the above sequence as: 
\[
\cdots \to \KO^{n-d}(X) \to \KO^n(X) \xrightarrow{\pi^\ast} \KO^n(S(E)) \to \KO^{n-d+1}(X) \to \cdots.
\]
We call it the \emph{Thom--Gysin exact sequence in $\KO$-theory}.

\section{Explicit description of some reductive symmetric spaces} \label{s:AppendixSymmetricSpaces}

In this appendix we give a synthetic construction of various symmetric spaces that are studied in Section \ref{s:ComputationsKTheory}. Throughout this appendix, we set 
\[
I_{p,q} = \begin{pmatrix} I_p & 0 \\ 0 & -I_q \end{pmatrix} \quad (p,q \in \N), \qquad J_{2n} = \begin{pmatrix} 0 & -I_n \\ I_n & 0 \end{pmatrix} \quad (n \in \N). 
\]

\subsection{Complex and split quaternions} \label{ss:complex-split-quat}

In this Section~\ref{ss:complex-split-quat}, we recall the definitions of complex and split quaternions.
They will be used to give a clean description of involutions in later sections. 

Let us consider the complexification $\Ha_\C = \Ha \otimes_\R \C$ of the ring $\Ha$ of quaternions, which we call the \emph{complex quaternion algebra}. We write $\sqrt{-1} \in \C$ for the unit pure-imaginary number in the coefficient field $\C$ to avoid confusion with $i \in \Ha$. We define the conjugation on $\Ha_\C$ by 
\[
\overline{q \otimes z} = \overline{q} \otimes \overline{z} \qquad (q \in \Ha,\ z \in \C). 
\]
It is an anti-multiplicative and $\C$-anti-linear involution on $\Ha_\C$. 
Note that 
\[
\Ha_\C \to \M(2, \C), \qquad a+bi+cj+dk \mapsto \begin{pmatrix} a + \sqrt{-1}d & -b - \sqrt{-1}c \\ b - \sqrt{-1}c & a - \sqrt{-1}d \end{pmatrix}
\]
is a $\C$-algebra isomorphism, hence we have $\M(n, \Ha_\C) \simeq \M(2n, \C)$. This identification is compatible with conjugate-transpose. We regard $\Ha$ as an $\R$-subalgebra of $\Ha_\C$ in an obvious way. 

Now, let us set $j' = \sqrt{-1}j$ and $k' = \sqrt{-1}k$. Then the $\R$-linear subspace 
\[
\Ha' = \R 1 \oplus \R i \oplus \R j' \oplus \R k',
\]
is a subalgebra of $\Ha_\C$, which we call the \emph{split quaternion algebra}. 
It is closed under conjugation. Note that 
\[
\Ha' \to \M(2, \R), \qquad a+bi+cj'+dk' \mapsto \begin{pmatrix} a-d & -b+c \\ b+c & a+d \end{pmatrix}
\]
is an $\R$-algebra isomorphism, hence we have $\M(n, \Ha') \simeq \M(2n, \R)$. Under this identification, conjugate-transposition in $\M(n, \Ha')$ corresponds to transposition in $\M(2n, \R)$. We regard $\C$ as an $\R$-subalgebra of $\Ha'$ in an obvious way.

\subsection{Description of \texorpdfstring{$\OO^\ast(2p+2q) / (\OO^\ast(2p) \times \OO^\ast(2q))$}{O*(2p+2q)/(O*(2p) x O*(2q))}} \label{ss:O*/O*O*} 

Let $p,q \in \N$ and $\widetilde{G} = \GL(p+q, \Ha)$. 
Consider three involutions $\theta, \tau, \sigma \colon \widetilde{G} \to \widetilde{G}$ defined by 
\[
\theta(g) = ({^t}\overline{g})^{-1}, \qquad \tau(g) = -i \cdot ({^t}\overline{g})^{-1} \cdot i, \qquad \sigma(g) = I_{p,q} \cdot g \cdot I_{p,q}. 
\]
These three involutions all commute, and $\theta$ is a Cartan involution. 
The Lie group $G = \OO^\ast(2p+2q)$ is the symmetric subgroup of $\widetilde{G}$ defined by the involution $\tau$: 
\begin{align*}
G &= \{ g \in \widetilde{G} \mid \tau(g) = g \} \\
&= \{ g \in \GL(p+q, \Ha) \mid {^t}\overline{g} \cdot i \cdot g = iI_{p+q} \}. 
\end{align*}
The involutions $\theta, \sigma \colon \widetilde{G} \to \widetilde{G}$ restrict to $G$. Let $H$ be the symmetric subgroup of $G$ defined by $\sigma$: 
\[
H = \{ g \in G \mid \sigma(g) = g \}.
\]
Since the symmetric subgroup of $\widetilde{G}$ defined by $\sigma$ is $\GL(p, \Ha) \times \GL(q, \Ha)$, we have 
\[
H = \OO^\ast(2p+2q) \cap (\GL(p, \Ha) \times \GL(q, \Ha)) = \OO^\ast(2p) \times \OO^\ast(2q). 
\]

Let us now compute the associated symmetric subgroup $H^a$ of $G$: 
\begin{align*}
H^a &= \{ g \in G \mid \sigma(\theta(g)) = g \} \\
&= \{ g \in G \mid \tau(\sigma(\theta(g)) = g \}. 
\end{align*}

We have 
\[
\tau \sigma \theta \colon \widetilde{G} \to \widetilde{G}, \qquad g \mapsto -iI_{p,q} \cdot g \cdot iI_{p,q},
\]
hence the symmetric subgroup of $\widetilde{G} = \GL(p+q, \Ha)$ defined by $\tau \sigma \theta$
consists of all invertible matrices of the form 
\[
\begin{pmatrix} A & -Bj \\ \overline{C}j & \overline{D} \end{pmatrix} \qquad (A \in \M(p,\C),\ B \in \M(p,q; \C),\, C \in \M(q,p; \C),\, D \in \M(q, \C)).
\]
It belongs to $H^a$ if and only if 
\[
\begin{pmatrix} {^t}\overline{A} & -{^t}\overline{C}j \\ {^t}Bj & {^t}D \end{pmatrix} \cdot i \cdot \begin{pmatrix} A & -Bj \\ \overline{C}j & \overline{D} \end{pmatrix} = i \cdot I_{p+q},
\]
which is equivalent to 
\[
{^t}\overline{A}A - {^t}\overline{C}C = 1, \qquad {^t}\overline{B}B - {^t}\overline{D}D = -1, \qquad {^t}\overline{A}B - {^t}\overline{C}D = 0.
\]
These conditions, in turn, are equivalent to 
\[
\begin{pmatrix} A & B \\ C & D \end{pmatrix} \in \U(p,q). 
\]
In conclusion, the following is an isomorphism of Lie groups: 
\[
\U(p,q) \to H^a, \qquad \begin{pmatrix} A & B \\ C & D \end{pmatrix} \mapsto \begin{pmatrix} A & -Bj \\ \overline{C}j & \overline{D} \end{pmatrix}.
\]

\subsection{Description of \texorpdfstring{$\Sp(2p+2q, \R) / (\Sp(2p, \R) \times \Sp(2q, \R))$}{Sp(2p+2q,R)/(Sp(2p,R) x Sp(2q,R))}} \label{ss:Sp/SpSp}

Under the identification $\M(n, \Ha') \simeq \M(2n, \R)$, 
the scalar matrix $i I_n \in \M(n, \Ha')$ corresponds to 
\[
J_{2n} = \begin{pmatrix} 0 & -I_n \\ I_n & 0 \end{pmatrix} \in \M(2n, \R). 
\]
We thus obtain the following isomorphism of Lie groups: 
\[
\Sp(2n, \R) \simeq \{ g \in \M(n, \Ha') \mid {^t \overline{g}} \cdot i \cdot g = i I_n \}. 
\]

Let $p,q \in \N$ and $\widetilde{G} = \GL(p+q, \Ha')$. 
Consider three involutions $\theta, \tau, \sigma \colon \widetilde{G} \to \widetilde{G}$ defined by 
\[
\theta(g) = ({^t} \overline{g})^{-1}, \qquad \tau(g) = -i \cdot ({^t}\overline{g})^{-1} \cdot i, \qquad \sigma(g) = I_{p,q} \cdot g \cdot I_{p,q}. 
\]
These three involutions all commute, and $\theta$ is a Cartan involution. 
By the above discussion, $G = \Sp(2p+2q, \R)$ is the symmetric subgroup of $\widetilde{G}$ defined by the involution $\tau$: 
\begin{align*}
G &= \{ g \in \widetilde{G} \mid \tau(g) = g \} \\
&= \{ g \in \GL(p+q, \Ha') \mid {^t}\overline{g} \cdot i \cdot g = i I_{p+q} \}. 
\end{align*}
The involutions $\theta, \sigma \colon \widetilde{G} \to \widetilde{G}$ restrict to $G$. Let $H$ be the symmetric subgroup of $G$ defined by $\sigma$: 
\[
H = \{ g \in G \mid \sigma(g) = g \}.
\]
Since the symmetric subgroup of $\widetilde{G}$ defined by $\sigma$ is $\GL(p, \Ha') \times \GL(q, \Ha')$, we have 
\[
H = \Sp(2p+2q, \R) \cap (\GL(p, \Ha') \times \GL(q, \Ha')) = \Sp(2p, \R) \times \Sp(2q, \R). 
\]

Let us now compute the associated symmetric subgroup $H^a$ of $G$: 
\begin{align*}
H^a &= \{ g \in G \mid \sigma(\theta(g)) = g \} \\
&= \{ g \in G \mid \tau(\sigma(\theta(g)) = g \}. 
\end{align*}
We have 
\[
\tau \sigma \theta \colon \widetilde{G} \to \widetilde{G}, \qquad g \mapsto -iI_{p,q} \cdot g \cdot iI_{p,q},
\]
hence the symmetric subgroup in $\widetilde{G} = \GL(p+q, \Ha')$ 
defined by $\tau \sigma \theta$ consists of all invertible matrices of the form 
\[
\begin{pmatrix} A & Bj' \\ \overline{C}j' & \overline{D} \end{pmatrix} \qquad (A \in \M(p,\C),\ B \in \M(p,q; \C),\, C \in \M(q,p; \C),\, D \in \M(q, \C)).
\]
It belongs to $H^a$ if and only if 
\[
\begin{pmatrix} {^t}\overline{A} & {^t}\overline{C}j' \\ {^t}Bj' & {^t}D \end{pmatrix} \cdot i \cdot \begin{pmatrix} A & Bj' \\ \overline{C}j' & \overline{D} \end{pmatrix} = i \cdot I_{p+q},
\]
which is equivalent to 
\[
{^t}\overline{A}A - {^t}\overline{C}C = 1, \qquad {^t}\overline{B}B - {^t}\overline{D}D = -1, \qquad {^t}\overline{A}B - {^t}\overline{C}D = 0.
\]
As in Section~\ref{ss:O*/O*O*}, we see that the following is an isomorphism of Lie groups: 
\[
\U(p,q) \to H^a, \qquad \begin{pmatrix} A & B \\ C & D \end{pmatrix} \mapsto \begin{pmatrix} A & Bj' \\ \overline{C}j' & \overline{D} \end{pmatrix}.
\]

\subsection{Description of \texorpdfstring{$\OO(2p, 2q) / \U(p,q)$}{O(2p,2q)/U(p,q)}} \label{ss:O/U}

Let $p,q \in \N$. Under the identification $\M(p+q, \Ha') \simeq \M(2p+\nolinebreak 2q, \R)$, the diagonal matrix $I_{p,q} \in \M(p+q, \Ha')$ corresponds to $I_{2p,2q} \in \M(2p+2q, \R)$. We thus obtain the following isomorphism of Lie groups: 
\[
\OO(2p,2q) \simeq \{ g \in \M(p+q, \Ha') \mid {^t \overline{g}} \cdot I_{p,q} \cdot g = I_{p,q} \}. 
\]

Let $\widetilde{G} = \GL(n, \Ha')$, and consider three involutions $\theta, \tau, \sigma \colon \widetilde{G} \to \widetilde{G}$ defined by 
\[
\theta(g) = ({^t}\overline{g})^{-1}, \qquad \tau(g) = I_{p,q} \cdot ({^t}\overline{g})^{-1} \cdot I_{p,q}, \qquad \sigma(g) = -i \cdot g \cdot i. 
\]
These three involutions all commute, and $\theta$ is a Cartan involution. 
By the above discussion, $G = \OO(2p,2q)$ is the symmetric subgroup of $\widetilde{G}$ defined by the involution $\tau$: 
\begin{align*}
G &= \{ g \in \widetilde{G} \mid \tau(g) = g \} \\
&= \{ g \in \GL(p+q, \Ha') \mid {^t} \overline{g} \cdot I_{p,q} \cdot g = I_{p,q} \}. 
\end{align*}
The involutions $\theta, \sigma \colon \widetilde{G} \to \widetilde{G}$ restrict to $G$. Let $H$ be the symmetric subgroup of $G$ defined by $\sigma$: 
\[
H = \{ g \in G \mid \sigma(g) = g \}.
\]
Since the symmetric subgroup of $\widetilde{G}$ defined by $\sigma$ is $\GL(p+q, \C)$, we have 
\[
H = \OO(2p,2q) \cap \GL(p+q, \C) = \U(p,q). 
\]
To compute the associated symmetric subgroup $H^a$ of $G$, set
\[
S_{p,q} = \begin{pmatrix} I_p & 0 \\ 0 & j' I_q \end{pmatrix}
\]
and consider the involution
\[
\varphi \colon \widetilde{G} \to \widetilde{G}, \qquad g \mapsto S_{p,q} \cdot g \cdot S_{p,q}. 
\]
We have
\[
\theta(S_{p,q}) = S_{p,q}, \qquad 
S_{p,q} \cdot I_{p,q} = I_{p,q} \cdot S_{p,q}.
\]
It follows that $\varphi$ commutes with $\theta$ and $\tau$, hence induces an involution on $G$. Furthermore, one can deduce from the identity
\[
S_{p,q} \cdot i \cdot I_{p,q} \cdot S_{p,q} = i I_{p+q}
\]
that $\varphi \sigma \tau \varphi = \sigma \theta$ holds on $\widetilde{G}$. 
In particular, we have $\varphi \sigma \varphi = \sigma \theta$ on $G$. 
We thus see that 
\[
H^a = \{ g \in G \mid \sigma(\theta(g)) = g \} = \{ g \in G \mid \varphi(\sigma(\varphi(g))) = g \} = \varphi(H). 
\]
In conclusion, the reductive symmetric space $G/H$ is self-associated.

\subsection{Description of \texorpdfstring{$\U(2p, 2q) / \Sp(p,q)$}{U(2p,2q)/Sp(p,q)}} \label{ss:U/Sp}

Let $p,q \in \N$. Under the identification $\M(p+q, \Ha_\C) \simeq \M(2p+\nolinebreak 2q, \C)$, the diagonal matrix $I_{p,q} \in \M(p+q, \Ha_\C)$ corresponds to $I_{2p,2q} \in \M(2p+2q, \C)$. We thus obtain the following isomorphism of Lie groups: 
\[
\U(2p,2q) \simeq \{ g \in \M(p+q, \Ha_\C) \mid {^t \overline{g}} \cdot I_{p,q} \cdot g = I_{p,q} \}. 
\]

Let $\widetilde{G} = \GL(p+q, \Ha_\C)$, and consider three involutions $\theta, \tau, \sigma \colon \widetilde{G} \to \widetilde{G}$ defined by 
\[
\theta(g) = ({^t}\overline{g})^{-1}, \qquad \tau(g) = I_{p,q} \cdot ({^t}\overline{g})^{-1} \cdot I_{p,q}, \qquad \sigma(g) = \overline{g}^\C,
\]
where $\overline{(-)}^\C \colon \M(n, \Ha_\C) \to \M(n, \Ha_\C)$ is defined by  
\[
\overline{A \otimes z}^\C = A \otimes \overline{z} \qquad (A \in \M(n, \Ha), z \in \C). 
\]
These three involutions all commute, and $\theta$ is a Cartan involution. 
By the above discussion, $G = \U(2p,2q)$ is the symmetric subgroup of $\widetilde{G}$ defined by the involution $\tau$: 
\begin{align*}
G &= \{ g \in \widetilde{G} \mid \tau(g) = g \} \\
&= \{ g \in \GL(p+q, \Ha_\C) \mid {^t}\overline{g} \cdot I_{p,q} \cdot g = I_{p,q} \}. 
\end{align*}
The involutions $\theta, \sigma \colon \widetilde{G} \to \widetilde{G}$ restrict to $G$. Let $H$ be the symmetric subgroup of $G$ defined by $\sigma$: 
\[
H = \{ g \in G \mid \sigma(g) = g \}. 
\]
Since the symmetric subgroup of $\widetilde{G} = \GL(p+q, \Ha_\C)$ defined by the involution $\sigma$ is $\GL(p+q, \Ha)$, we have 
\[
H = \U(2p,2q) \cap \GL(p+q, \Ha) = \Sp(p,q). 
\]

As in Section~\ref{ss:O/U}, 
to compute the associated symmetric subgroup $H^a$ of $G$, consider the involution 
\[
\varphi \colon \widetilde{G} \to \widetilde{G}, \qquad g \mapsto S_{p,q} \cdot g \cdot S_{p,q}.
\]
It commutes with $\theta$ and $\tau$, hence induces an involution on $G$. 
Furthermore, one can deduce from the identity 
\[
S_{p,q} \cdot I_{p,q} \cdot \overline{S_{p,q}}^\C = I_{p+q}
\]
that $\varphi \sigma \tau \varphi = \sigma \theta$ holds on $\widetilde{G}$. 
In particular, we have $\varphi \sigma \varphi = \sigma \theta$ on $G$. 
We thus see that $H^a = \varphi(H)$. The reductive symmetric space $G/H$ is self-associated.


\end{document}